\documentclass[a4paper,11pt]{article}
\usepackage{mathrsfs}
\usepackage{color,amsmath,mathtools,amsthm,amssymb,amsfonts}
\usepackage[numbers,sort&compress]{natbib}
\usepackage{enumerate}

\usepackage{booktabs}  % 专业表格线
\usepackage{array}     % 表格列格式
\usepackage{multirow}  % 合并行
\usepackage{caption}   % 表格标题
\usepackage{xcolor}
\usepackage{float}  % 在导言区添加
\usepackage[table]{xcolor}
\numberwithin{equation}{section}

\allowdisplaybreaks[4]

\newtheorem{Theorem}{Theorem}[section]

\newtheorem{Lemma}{Lemma}[section]
\newtheorem{Remark}{Remark}[section]

\newtheorem{Corollary}{Corollary}[section]
\newtheorem{Proposition}{Proposition}[section]

\newcounter{sublemma}[Lemma]  % sublemma 计数器在每个 lemma 环境内重置
\makeatletter%使用罗马数字

\newcommand{\Rmnum}[1]{\expandafter\@slowromancap\romannumeral #1@}
\renewenvironment{proof}[1][\proofname]{\par
  \pushQED{\qed}%
  \normalfont \topsep6\p@\@plus6\p@\relax
  \trivlist
  \item[\hskip\labelsep
        \itshape
    #1\@addpunct{.}]\ignorespaces
}{%
  \popQED\endtrivlist\@endpefalse
}
\makeatother

\usepackage{authblk}
\usepackage[title]{appendix}
\makeatother

\begin{document}
	\title{Vanishing capillary limit for compressible Navier-Stokes-Korteweg system in a bounded interval with large initial data}
	\author[a]{Jiaxin Ling}
	\author[a]{Huanyao Wen\thanks{Corresponding author.}}
	\author[b]{Xinhua Zhao}
	\affil[a]{School of Mathematics, South China University of Technology, Guangzhou 510641, China}
\affil[b]{School of Mathematics and Systems Science, Guangdong Polytechnic Normal University, Guangzhou 510665, China}
	\date{}
	\maketitle
	\renewcommand{\thefootnote}{}
	\footnote{ {E}-mail: scut\_jxling@163.com(Ling); mahywen@scut.edu.cn(Wen); xhzhao@gpnu.edu.cn(Zhao).}
			\begin{abstract}
		In this paper, we study vanishing capillary limit for isentropic compressible Navier-Stokes-Korteweg system in a bounded interval. The main challenges focus on the capillary term and the boundary effect. A new uniform dissipative estimate in terms of the third-order derivative of density and some correctors are derived to handle such difficulties. It leads to the optimal convergence rate of the solutions in $L^\infty$ norm globally in time with arbitrarily large initial data. This work provides a rigorous derivation of the  isentropic compressible Navier-Stokes system in a bounded interval from the isentropic compressible Navier-Stokes-Korteweg system via vanishing capillary limit.

	\end{abstract}
{\noindent  \small \textbf{Keywords:} compressible Navier-Stokes-Korteweg system, compressible Navier-Stokes system, initial-boundary value problem, vanishing capillary limit.}

\vspace{2mm}
	{\noindent\textbf{AMS Subject Classification (2020):} 35Q35, 35B40, 76N10.}

	\small \linespread{1.0}
\section{Introduction}	
~~~~In 1894, Van der Waals \cite{V1894} (translated in \cite{vdW1979}) discovered the diffuse interface between fluid mixtures, which replaces the sharp interface between two fluids in a physical perspective. Later in 1958, Cahn and Hilliard \cite{CH1958} described the diffuse interface phenomenon in a more standard way, which was translated to a mathematical model by Gurtin, Warne and Vi\~nals\cite{GPV1996}. In 1901, Korteweg \cite{K1901} derived a mathematical model for the diffuse interface phenomenon by giving the Cauchy stress tensor that included density gradients which was called Korteweg stress tensor. Then, Dunn and Serrin \cite{DS1985} provided a crucial correction to the Korteweg stress in \cite{K1901}. Later, Heida and M\'alek \cite{HM2010} developed a thermodynamic framework for Korteweg-type fluids based on an entropy formulation, which compared to \cite{DS1985,K1901,vdW1979} avoids the new concepts as \emph{interstitial working} to form the mathematical model. Recently, Elbar, Gwiazda, Skrzeczkowski and Gwiazda \cite{Eetal2025} established a theoretical connection between Cahn-Hilliard type fluids and Korteweg type fluids.
Within these theoretical frameworks, the Korteweg fluid model is capable of describing diffuse interface phenomena. For more information about diffuse interface model, we can refer to \cite{AMW1998}. Mathematically, the isentropic compressible Navier-Stokes-Korteweg (NSK) system in $d$ dimensions can be stated as follows:
\begin{equation}\label{1-1}
		\left\{\begin{aligned}
		&\rho_t+div(\rho \mathbf{u})=0 ,\\
&\rho\mathbf{u}_t+ \rho \mathbf{u} \cdot \nabla \mathbf{u} +\nabla P(\rho)=\mu_1\Delta \mathbf{u} +(\lambda+\mu_1) \nabla div \mathbf{u}+div \mathbb{K},\\
		\end{aligned}
		\right.
		\end{equation}
where $\rho$, $\mathbf{u}=(u_1,...,u_d)$ and $P(\rho)
=A\rho^\gamma$ represent density, velocity field of the fluid and pressure with adiabatic index $\gamma>1$, respectively. Here $\mu_1$ and $\lambda$ are viscosity coefficients that satisfy $\mu_1>0$ and $\mu_1 +\frac{d\lambda}{2}\ge 0$, and $\mathbb{K}$ is the capillary tensor satisfying
$$
\mathbb{K}=[\rho c(\rho)\Delta \rho +\frac{1}{2}(c(\rho)+\rho c'(\rho))|\nabla\rho|^2]\mathbb{I}-c(\rho) \nabla \rho \otimes\nabla \rho,
$$
where $c$ is called capillary coefficient. When $c$ is a constant, we get $div \mathbb{K}=c\rho \nabla \Delta \rho$.

In this paper, we consider the case that $\mu:=2\mu_1+\lambda=1$, $d=1$, $A=1$, $c=\epsilon$ and $(\rho,\mathbf{u})=(\rho^\epsilon, u^\epsilon)$. Then the system (\ref{1-1}) is translated into the following one:
\begin{equation}\label{equ-epsilon}
		\left\{\begin{aligned}
		&\rho_t^\epsilon+(\rho^\epsilon u^\epsilon)_x=0 ,~(x,t)\in [0,1]\times[0,T],\\
&\rho^\epsilon u_t^\epsilon+ \rho^\epsilon u^\epsilon u_x^\epsilon +[(\rho^\epsilon)^\gamma]_x= u_{xx}^\epsilon+\epsilon \rho^\epsilon \rho_{xxx}^\epsilon,~(x,t)\in [0,1]\times[0,T],\\
&(\rho^\epsilon, u^\epsilon)|_{t=0}=(\rho_0^\epsilon,u_0^\epsilon),~x\in[0,1],\\
&(\rho_x^\epsilon,u^\epsilon)|_{x=0,1}=(0,0),~t\in[0,T].
		\end{aligned}
		\right.
		\end{equation}
In this paper, we aim to investigate the convergence of the system (\ref{equ-epsilon}) as the capillary coefficient $\epsilon$ tends to zero.

There are a number of literatures on the vanishing parameters limits of the compressible NSK system such as vanishing Mach number limit, vanishing viscous-capillary limit and vanishing capillary limit.

  For the vanishing Mach number limit, Li and Yong\cite{2016LY} established the convergence rate for the isentropic compressible NSK system in ${\mathbb{R}}^3$. In addition, it was proved that the existence time of solutions extends to infinity as the Mach number tends to zero, if the incompressible NS system admits a global smooth solution. Later, Ju and Xu \cite{JX2022} improved the results of \cite{2016LY} by relaxing the restrictions on the initial data and obtaining more precise convergence rates. Recently, Li and Yin \cite{LY2026} studied the low Mach number limit of the full compressible NSK system in $\mathbb{R}$. With well-prepared initial data and with the difference between the states at $\infty$ being suitably small, they showed that the solutions of the full compressible NSK system converge to a nonlinear diffusion wave solution globally in time. For ill-prepared initial data with arbitrarily large far-field difference, they proved weak convergence on a finite time interval.
See \cite{JLW2014} for the vanishing limit of Mach number and capillarity.

%%In the vanishing viscous-capillary limit, the NSK system converges to the Euler system.

%%Both \cite{WZ2024} and \cite{YYZ2015} investigate the convergence of the incompressible inhomogeneous NSK system to the incompressible inhomogeneous Euler system, as both the viscosity and the capillary coefficients vanish.
%In \cite{YYZ2015}, Yang, Yao and Zhu establish the existence and uniqueness of local-in-time solution for incompressible inhomogeneous NSK system in ${\mathbb{R}}^3$ and investigate the convergence rate in $H^2$ norm.
%%In \cite{WZ2024}, Wang and Zhang analyze the incompressible inhomogeneous NSK system with slip boundary condition in three dimension,
%Under the condition $\mu = \bar\mu \epsilon$ and $c=c_1 \epsilon^s$ with $s>1$ and $\bar\mu, ~c_1>0$  are positive constants, \cite{WZ2024} provide the local well-posedness and convergence rate in $H^1$ norm of the incompressible NSK system. Moreover, in \cite{WZ2024}, and no boundary layer effect is present.

 For isentropic compressible NSK system, vanishing viscous-capillary limit leads to compressible Euler system. The one-dimensional analyses in \cite{CH2013, GL2016} proceeded under the condition that $\mu$ and $c$ are functions of density, with the domains being $\mathbb{R}$ in \cite{CH2013} and either $\mathbb{R}$ or $\mathbb{T}$ in \cite{GL2016}. More specifically, under the condition $\mu=\epsilon\rho^\epsilon$, $c=\frac{\epsilon^2}{\rho^\epsilon}$ and $\gamma>\frac{5}{3}$, Charve and Haspot \cite{CH2013} established almost everywhere convergence as $\epsilon \rightarrow 0$. Similar convergence was established in \cite{GL2016} with  $\mu=\epsilon \mu(\rho^\epsilon)$, $c=\delta(\epsilon) c(\rho^\epsilon)$ and $1<\gamma\le \frac{5}{3}$, where $\delta(\epsilon) $ tends to 0 with $\epsilon \rightarrow 0$. For vanishing viscous-capillary limit in the wave regime of the compressible NSK system, refer for instance to \cite{LL2016,LZ2020, YL2022}.

The Navier-Stokes-Korteweg system can be viewed as a dispersive correction for Navier-Stokes system (\cite{HS1983,S1996}). It is interesting to investigate the vanishing capillary limit. Bian, Yao and Zhu \cite{BYZ2014} considered the Cauchy problem for the isentropic compressible NSK system in three dimensions and established the $H^1$-norm convergence as the capillary coefficient $c=\epsilon$ tends to zero at a rate of order $\epsilon$. See \cite{YY2025} for further investigation. In \cite{PG2016}, Pu and Guo considered the Cauchy problem for a compressible Korteweg-type fluid system that describes quantum hydrodynamics. By investigating its vanishing capillary limit which reduces the system to the classical hydrodynamic system, \cite{PG2016} derived the semiclassical limit and obtained the convergence rate in $H^1$ norm. Hou, Yao and Zhu in \cite{HYZ2017} obtained the zero-capillary limit convergence in $H^2$ norm for the non-isentropic compressible NSK system in $\mathbb{R}^3$. Lai, Wen and Yao \cite{LWY2017} got the zero-capillary limit convergence in $H^1$ norm for the two fluid compressible NSK system in $\mathbb{R}^3$. In space $\mathbb{R}$,  Burtea and Haspot \cite{BH2022} considered the isentropic compressible NSK system with weaker condition of $\mu$ and $c$, where the typical example is $\mu=\mu^\alpha(\rho^\epsilon)$ with $\alpha\in(0,\frac{1}{2})$ and $c=\epsilon\frac{\mu^2(\rho^\epsilon)}{(\rho^\epsilon)^3}$.
%\cite{BH2022} demonstrates a weaker form of convergence, specifically weak star convergence and weakly in the sense of measures.

In summary, the convergence rate of vanishing capillary limit of the initial-boundary value problem for the compressible NSK system is still open. This is the main aim in this paper. Formally, vanishing of capillary leads to the following compressible Navier-Stokes system:
\begin{equation}\label{equ-NS}
		\left\{\begin{aligned}
		&\rho_t^{I,0}+(\rho^{I,0} u^{I,0})_x=0 ,\\
&\rho^{I,0} u_t^{I,0}+ \rho^{I,0} u^{I,0} u_x^{I,0} +[(\rho^{I,0})^\gamma]_x=u_{xx}^{I,0},\\
&(\rho^{I,0}, u^{I,0})|_{t=0}=(\rho_0,u_0),\\
&u^{I,0}|_{x=0,1}=0,
		\end{aligned}
		\right.
		\end{equation}for $(x,t)\in [0,1]\times[0,T]$.
Our analysis is built upon the foundational framework of Prandtl's boundary layer theory in 1904 \cite{P1904}.

There are some literatures on the well-posedness of compressible Korteweg-type fluids, see for instance \cite{CHZ2017,CZ2014,H2020,K2008,KSX2021,LPW2018} and references therein. More precisely, for Cauchy problem, Chen and Zhao\cite{CZ2014} obtained global stability of stationary solutions to the full compressible NSK system in three dimensions. Later, Chen, He and Zhao\cite{CHZ2017} derived the global existence and uniqueness of smooth solutions for one-dimensional nonisothermal compressible fluid models of Korteweg type with large initial data. Kawashima, Shibata and Xu\cite{KSX2021} obtained time-decay estimates for the compressible NSK system with small data in Besov space. Li, Peng and Wang\cite{LPW2018} obtained the global existence and uniqueness of strong solution to the Cauchy problem for three-dimensional compressible Navier-Stokes-Poisson-Korteweg system in three dimensions. In addition, they investigated the vanishing capillary limit. For the initial-boundary value problem, Kotschote\cite{K2008} proved the local existence and uniqueness of strong solution in multi-dimensions, see also \cite{H2020}.

This paper is organized as follows. In Section \ref{section 1.1}, we derive the boundary layer equations, where the details will be given in \ref{appendixA}. In Section \ref{section1.2}, well-posedness result of boundary layer equations, system (\ref{equ-epsilon}) and system (\ref{equ-NS}) are obtained. %The proof of the well-posedness result of boundary layer equations is given in \ref{appendixB}.
In Section \ref{section1.3}, we present the main result. Following in Section \ref{section1.4} we define the error function. In Section \ref{section1.5}, we state the main difficulties and innovations of this paper. Based on some known results stated in Section \ref{section2}, we will prove the main result in Section \ref{section3}.

%We derive a strict mathematical proof for vanishing capillarity limit from NSK to NS under the $\epsilon^{\frac{1}{2}}$.

\subsection*{Notation}
	
\noindent (1) $\overline{f}(t):=f(0,t)$, $\widetilde{f}(t):=f(1,t)$, $\langle y \rangle:=\sqrt{1+y^2}$.\\[2mm]
\noindent (2) For $x\in [0,1]$, we define $y=\frac{x}{\sqrt{\epsilon}} \in [0,\frac{1}{\sqrt\epsilon}]$, $z=\frac{x-1}{\sqrt{\epsilon}}\in [\frac{-1}{\sqrt\epsilon},0]$ for some $\epsilon>0$. Moreover, if $\epsilon \rightarrow 0$, then we have $y\in[0,+\infty)$, $z\in(-\infty,0]$.\\[2mm]
\noindent (3) For any $1\le p,q \le \infty$ and $m\in \mathbb{N}$, we denote
 \begin{itemize}
\item $L_t^p L^q:=L^p(0,t;L^q(I))$, $L_t^p W^{1,q}:=L^p(0,t;W^{1,q}(I)),$ where $I:=(0,1)$;

\item $L_t^p L^{q}_y:=L^p(0,t;L^q_y(0,+\infty))$, $L_t^p L^{q}_z:=L^p(0,t;L^q_z(-\infty,0))$;

\item $L_t^p H^{m}_y:=L^p(0,t;H^m_y(0,+\infty))$, $L_t^p H^m_z:=L^p(0,t;H^m_z(-\infty,0))$;

\item  $H_T^0:=L^2(0,T)$, $W_T^{0,\infty}:=L^\infty(0,T)$;

\item  $H_y^0:=L^2_y(0,+\infty)$,  $W_y^{0,\infty}:=L^\infty(0,+\infty)$;

\item  $L_T^p:=L^p(0,T)$, $H_T^m:=H^m(0,T)$.

\end{itemize}

\noindent (4) Let $C\ge1$ be a general positive constant that may positively correlated with time  $T$, and depend on initial data  and some other known constants, but independent of $\epsilon$.

\subsection{Derivation of boundary layer equations}\label{section 1.1}
~~~~In Section \ref{section 1.1}, the boundary-layer equations will be presented, whose details of derivation will be given in \ref{appendixA}. Namely, we assume that the solution of (\ref{equ-epsilon}) formally satisfies
\begin{equation}\label{Plantl}
\left\{
\begin{aligned}
\rho^\epsilon=\sum_{j=0}^\infty \epsilon^{\frac{j}{2}}[\rho^{I,j}(x,t) +\rho^{B,j}(y,t)+\rho^{b,j}(z,t)],\\
u^\epsilon=\sum_{j=0}^\infty u^{\frac{j}{2}}[u^{I,j}(x,t) +u^{B,j}(y,t)+u^{b,j}(z,t)],
\end{aligned}\right.
\end{equation}
where $\rho^{B,i},~u^{B,i}$ and their derivatives decay fastly to $0$ as $y\rightarrow +\infty$, and $\rho^{b,i},~u^{b,i}$ and their derivatives decay fastly to $0$ as $z\rightarrow -\infty$ for any $i\in \mathbb{N}$.

Firstly, it holds that
\begin{equation}\label{IBb-1}
\begin{aligned}
&u^{B,0}=u^{B,1}=\rho^{B,0}=0,~u^{b,0}=u^{b,1}=\rho^{b,0}=0,~u^{I,1}=\rho^{I,1}=0.\\
\end{aligned}
\end{equation}
The profiles $(\rho^{B,1},\rho^{b,1},u^{B,2},u^{b,2})$ satisfy the following equations:
\begin{equation}\label{IBb-3}
\left\{
\begin{aligned}
&\rho^{B,1}_t+y\overline{u^{I,0}_x}\rho^{B,1}_y+[\gamma (\overline{\rho^{I,0}})^{\gamma}+\overline{u^{I,0}_x}]\rho^{B,1}=(\overline{\rho^{I,0}})^2\rho^{B,1}_{yy},\\
& \rho^{B,1}(y,0)=0,\\
& \rho_y^{B,1}(0,t)=-\rho_x^{I,0} (0,t)=-\overline{\rho_x^{I,0}},\\
\end{aligned}\right.
\end{equation}

\begin{equation}\label{Ib-3}
\left\{
\begin{aligned}
&\rho^{b,1}_t+z\widetilde{u^{I,0}_x}\rho^{b,1}_z+[\gamma (\widetilde{\rho^{I,0}})^{\gamma}+\widetilde{u^{I,0}_x}]\rho^{b,1}=(\widetilde{\rho^{I,0}})^2\rho^{b,1}_{zz},\\
&\rho^{b,1}(z,0)=0,\\
& \rho_z^{b,1}(0,t)=-\rho_x^{I,0} (1,t)=-\widetilde{\rho_x^{I,0}},
\end{aligned}\right.
\end{equation}
and
\begin{equation}\label{IBb-2}
\left\{
\begin{aligned}
&u^{B,2}=-\overline{\rho^{I,0}} \rho_y^{B,1}-\gamma (\overline{\rho^{I,0}})^{\gamma-1} \int_y^{+\infty} \rho^{B,1}(\xi,t) \mathrm{d}\xi,\\
&u^{b,2}=-\widetilde{\rho^{I,0}} \rho_z^{b,1}+\gamma (\widetilde{\rho^{I,0}})^{\gamma-1} \int_{-\infty}^{z} \rho^{b,1}(\xi,t) \mathrm{d}\xi.\\
\end{aligned}\right.
\end{equation}

The profiles $(\rho^{B,2},\rho^{b,2}, u^{B,3},u^{b,3})$ satisfy the following equations:
\begin{equation}\label{IBb-4}
\left\{
\begin{aligned}
&\rho^{B,2}_t+\overline{\rho^{I,0}_x}u^{B,2} +\rho^{B,1}u^{B,2}_y+ \rho^{B,1}_y(\overline{u^{I,2}}+u^{B,2})+\rho^{B,2}\overline{u^{I,0}_x}+y\overline{u^{I,0}_{xx}}\rho^{B,1}\\
&+y\overline{\rho^{I,0}_x}u^{B,2}_y+y\overline{u^{I,0}_{x}}\rho^{B,2}_y+\frac{y^2}{2}\overline{u^{I,0}_{xx}}\rho^{B,1}_y \\
&=-\gamma (\overline{\rho^{I,0}})^{\gamma} \rho^{B,2}-\frac{1}{2}\gamma (\gamma-1) (\overline{\rho^{I,0}})^{\gamma-1} (\rho^{B,1})^2 -\gamma (\gamma-1)y (\overline{\rho^{I,0}})^{\gamma-1} \overline{\rho^{I,0}_x}\rho^{B,1}\\
&+(\overline{ \rho^{I,0} })^2\rho^{B,2}_{yy}+\overline{ \rho^{I,0} }\rho^{B,1} \rho^{B,1}_{yy}-\frac{1}{2}\overline{ \rho^{I,0} }(\rho^{B,1}_y )^2+y\overline{ \rho^{I,0} }\overline{\rho^{I,0}_x}\rho^{B,1}_{yy}-\overline{ \rho^{I,0} }\overline{\rho^{I,0}_x}\rho^{B,1}_{y},\\
&\rho^{B,2}(y,0)=0,\\
&\rho^{B,2}_y|_{y=0}= -\overline{\rho^{I,1}_x}=0,
\end{aligned}\right.
\end{equation}
\begin{equation}\label{IBb-5}
\left\{
\begin{aligned}
&\rho^{b,2}_t+\widetilde{\rho^{I,0}_x}u^{b,2}+\rho^{b,1}u^{b,2}_z+\rho^{b,1}_z(\widetilde{u^{I,2}}+u^{b,2})+\rho^{b,2}\widetilde{u^{I,0}_x}+z\widetilde{u^{I,0}_{xx}}\rho^{b,1}\\
&+z\widetilde{\rho^{I,0}_x}u^{b,2}_z+z\widetilde{u^{I,0}_{x}}\rho^{b,2}_z+\frac{z^2}{2}\widetilde{u^{I,0}_{xx}}\rho^{b,1}_z\\%=-\widetilde{\rho^{I,0}}u^{b,3}_z\\
&=-\gamma (\widetilde{\rho^{I,0}})^{\gamma} \rho^{b,2}-\frac{1}{2}\gamma (\gamma-1) (\widetilde{\rho^{I,0}})^{\gamma-1} (\rho^{b,1})^2 -\gamma (\gamma-1)z (\widetilde{\rho^{I,0}})^{\gamma-1} \widetilde{\rho^{I,0}_x}\rho^{b,1}\\
&+(\widetilde{ \rho^{I,0} })^2\rho^{b,2}_{zz}+\widetilde{ \rho^{I,0} }\rho^{b,1} \rho^{b,1}_{zz}-\frac{1}{2}\widetilde{ \rho^{I,0} }(\rho^{b,1}_z )^2+z\widetilde{ \rho^{I,0} }\widetilde{\rho^{I,0}_x}\rho^{b,1}_{zz}-\widetilde{ \rho^{I,0} }\widetilde{\rho^{I,0}_x}\rho^{b,1}_{z},\\
& ~\rho^{b,2}(z,0)=0,\\
&~\rho^{b,2}_z|_{z=0}= -\widetilde{\rho^{I,1}_x}=0,
\end{aligned}\right.
\end{equation}
and
%The profiles $(u^{B,3},u^{b,3})$ satisfy the following equations:
\begin{equation}\label{rho-B2-inibou}
\left\{
\begin{aligned}
u^{B,3}
=&-\gamma (\overline{\rho^{I,0}})^{\gamma-1} Q_2-\frac{1}{2}\gamma (\gamma-1) (\overline{\rho^{I,0}})^{\gamma-2} \int_y^{+\infty} (\rho^{B,1})^2(\xi,t) \mathrm{d}\xi\\
& -\gamma (\gamma-1) (\overline{\rho^{I,0}})^{\gamma-2} \overline{\rho^{I,0}_x}(yQ_{1,1}+Q_{1,2})-\overline{ \rho^{I,0} }\rho^{B,2}_{y}-\rho^{B,1} \rho^{B,1}_{y}\\
&-\frac{3}{2}\int_y^{+\infty} (\rho^{B,1}_\xi )^2(\xi,t)\mathrm{d}\xi-y\overline{\rho^{I,0}_x}\rho^{B,1}_{y}+2\overline{\rho^{I,0}_x}\rho^{B,1},\\
u^{b,3}
=&\gamma (\widetilde{\rho^{I,0}})^{\gamma-1} q_2+\frac{1}{2}\gamma (\gamma-1) (\widetilde{\rho^{I,0}})^{\gamma-2} \int_{-\infty}^{z} (\rho^{b,1})^2(\zeta,t) \mathrm{d}\zeta\\
& +\gamma (\gamma-1) (\widetilde{\rho^{I,0}})^{\gamma-2} \widetilde{\rho^{I,0}_x}(zq_{1,1}-q_{1,2})-\widetilde{ \rho^{I,0} }\rho^{b,2}_{z}-\rho^{b,1} \rho^{b,1}_{z}\\
&+\frac{3}{2}\int_{-\infty}^{z} (\rho^{b,1}_\zeta )^2(\zeta,t)\mathrm{d}\zeta-z\widetilde{\rho^{I,0}_x}\rho^{b,1}_{z}+2\widetilde{\rho^{I,0}_x}\rho^{b,1},\\
\end{aligned}\right.
\end{equation}
where
$$
Q_{1,1}(y,t)=\int_y^{+\infty}\rho^{B,1}(\xi,t) \mathrm{d}\xi ,~Q_{1,2}(y,t)=\int_y^{+\infty} \int_s^{+\infty} \rho^{B,1}(\xi,t) \mathrm{d}\xi \mathrm{d}s,~Q_2=\int_y^{+\infty}\rho^{B,2}(\xi,t) \mathrm{d}\xi ,
$$
$$
q_{1,1}(y,t)=\int_{-\infty}^z \rho^{b,1}(\zeta,t) \mathrm{d}\zeta , ~q_{1,2}(y,t)=\int_{-\infty}^z  \int_{-\infty}^s \rho^{b,1}(\zeta,t) \mathrm{d}\zeta\mathrm{d}s,~q_2= \int_{-\infty}^z \rho^{b,2}(\zeta,t)\mathrm{d}\zeta,
$$
$$\begin{aligned}
\overline{u^{I,2}}=& u^{I,2}(0,t)=-u^{B,2}(0,t)
=\overline{\rho^{I,0}}\overline{\rho^{I,0}_x}+\int_0^{+\infty}\gamma (\overline{\rho^{I,0}})^{\gamma-1} \rho^{B,1}(\xi,t)\mathrm{d}\xi,\\
\widetilde{u^{I,2}}=& u^{I,2}(1,t)=-u^{b,2}(0,t)
= \widetilde{\rho^{I,0}}\widetilde{\rho^{I,0}_x}-\int_{-\infty}^0\gamma (\widetilde{\rho^{I,0}})^{\gamma-1}\rho^{b,1}(\zeta,t)\mathrm{d}\zeta .
\end{aligned}
$$

\subsection{Well-posedness}\label{section1.2}
~~~~Existence and uniqueness of global smooth solution to compressible Navier-Stokes equations and existence and uniqueness of local strong solution to the compressible Navier-Stokes-Korteweg equations are stated as follows.
\begin{Proposition}\label{exi-ns}
 Suppose that $(\rho_0,u_0)\in H^6(I)\times (H_0^1(I)\cap H^6(I))$, $\inf\limits_{I}\rho_0\ge c_0> 0$ for some constant $c_0>0$ and that  the compatibility condition
$$
(u_t^{I,0}, u_{tt}^{I,0})(x,0)|_{x=0,1} =(0,0),
$$
holds, then system (\ref{equ-NS}) admits a unique global strong solution $(\rho^{I,0},u^{I,0})$ satisfying
$$
\partial_t^i\rho^{I,0} \in C([0,T]; H^{6-i}(I)),~for~any~i=0,1,...,6,
$$
$$
\partial_t^ju^{I,0} \in C([0,T]; H_0^1(I) \cap H^{6-2j}(I))\cap L_T^2 (H_0^1 \cap H^{7-2j}),~for~any~j=0,1,2,3,
$$
$$
\frac{1}{M}\le \rho^{I,0}(x,t) \le M,
$$ where $M>1$ depends on the initial data only,
and
$$
\begin{aligned}
&\sum_{i=0}^6\sup_{t\in[0,T]}\|\partial_t^i\rho^{I,0} \|_{H^{6-i}(I)} +\!\sum_{i=0}^3(\sup_{t\in[0,T]}\|\partial_t^iu^{I,0}\|_{  H^{2(3-i)}(I) }\!+\!\|\partial_t^iu^{I,0}\|_{  L_T^2 H^{2(3-i)+1}})\\
\le & C(\|\rho_0\|_{H^6(I)},\|(\rho_0)^{-1}\|_{L^\infty(I)},\|u_0\|_{H^6(I)},T,c_0)\\
\end{aligned}
$$
for any $T>0$ and $(x,t)\in I \times [0,T]$.
\end{Proposition}
\begin{Remark}
The proof of Proposition \ref{exi-ns} is similar to that in Remark 1.3 of \cite{WWZ2023}.
\end{Remark}

\begin{Proposition}\label{exi-nsk}
 Suppose that $(\rho_0^\epsilon,u_0^\epsilon)\in H^2(I)\times H^1_0(I)$ satisfies $\partial_x\rho_0^\epsilon|_{x=0,1}=0$ and $\inf\limits_{I}\rho^{\epsilon}_0\ge c_0 > 0$ for some constant $c_0>0$, then there exists a positive constant $T_\epsilon$ such that system (\ref{equ-epsilon}) admits a unique local strong solution $(\rho^{\epsilon},u^{\epsilon})$ satisfying
 $$\rho^{\epsilon}\in C([0,T_\epsilon]; H^2(I)) \cap L^2_{T_\epsilon} H^3,\,\,(\rho^\epsilon)^{-1} \in C([0,T_\epsilon]; L^\infty(I)), u^\epsilon\in C([0,T_\epsilon]; H^1_0(I))\cap L^2_{T_\epsilon} (H^1_0\cap H^2)$$ and
$$
\begin{aligned}
&\|\sqrt{\epsilon}\rho^\epsilon_{x} \|_{C([0,T_\epsilon]; L^2(I))}^2+\|{\rho^\epsilon}\|_{C([0,T_\epsilon]; L^\gamma(I))} +\|\sqrt{\rho^\epsilon} u^\epsilon\|_{ C([0,T_\epsilon]; L^2(I))}^2+\|u^\epsilon_x\|_{ L_{T_\epsilon}^2 L^2(I)}^2
\\ &\le C(\|\rho_0^\epsilon\|_{H^1(I)},\|u_0^\epsilon\|_{L^2(I)}).
\end{aligned}
$$

\end{Proposition}
\begin{Remark} The proof of local existence and uniqueness of the compressible NSK equations stated in Proposition \ref{exi-nsk} is similar to that in \cite{H2020}, where the corresponding result with bounded domain in three dimensions was obtained.
\end{Remark}

\begin{Proposition}\label{B-b-Pro}
Under the assumption of Proposition \ref{exi-ns}, if $$\partial_x \rho_0 (x)|_{x=0,1}=0$$  holds for system (\ref{equ-NS}), then, for any $T>0$, systems (\ref{IBb-3})-(\ref{rho-B2-inibou}) have unique global solution respectively, where
for any $l\in \mathbb{N}$ and $k=0,1,2$, regularities of boundary layer profiles are stated as follows:
\[\langle y \rangle^{l} \rho^{B,1} \in W_T^{k,\infty} H_y^{5-2k} \cap H_T^k H_y^{6-2k},~ \langle z \rangle^{l} \rho^{b,1} \in W_T^{k,\infty} H_z^{5-2k} \cap H_T^k H_z^{6-2k},\]
\[\langle y \rangle^{l} \rho^{B,2} \in W_T^{k,\infty} H_y^{5-2k} \cap H_T^k H_y^{6-2k},~ \langle z \rangle^{l} \rho^{b,2} \in W_T^{k,\infty} H_z^{5-2k} \cap H_T^k H_z^{6-2k},\]
\[\langle y \rangle^{l} u^{B,2} \in W_T^{k,\infty} H_y^{4-2k} \cap H_T^k H_y^{5-2k},~ \langle z \rangle^{l} u^{b,2} \in W_T^{k,\infty} H_z^{4-2k} \cap H_T^k H_z^{5-2k},\]
\[\langle y \rangle^{l} u^{B,3} \in W_T^{k,\infty} H_y^{4-2k} \cap H_T^k H_y^{5-2k},~ \langle z \rangle^{l} u^{b,3} \in W_T^{k,\infty} H_z^{4-2k} \cap H_T^k H_z^{5-2k}.\]
\end{Proposition}

\begin{Remark}
             The proof of Proposition \ref{B-b-Pro} is similar to that in \cite{WWZ2023, WZ2024, WYZZ2021}. We omit it for brevity.
\end{Remark}

\subsection{Main result}\label{section1.3}
\begin{Theorem}\label{following}
In addition to the conditions of Propositions \ref{exi-ns}-\ref{B-b-Pro},
we assume that %$(\rho_0^\epsilon,u_0^\epsilon)$ and $(\rho_0^{I,0},u_0^{I,0})$ satisfy
\begin{equation}\label{initial-value}
\|\rho^{\epsilon}_0-\rho^{I,0}_0\|_{H^2(I)}^2 +\|u^\epsilon_0 -u^{I,0}_0\|_{H^1(I)}^2\le  C_0\epsilon^m
\end{equation}
holds for $m>2$ and $C_0>0$.

For any time $T\in(0,+\infty)$, there exists a positive constant $\epsilon_1$ determined by (\ref{ep1}) such that
if $\epsilon \in(0,\epsilon_1)$, then the following estimates
\begin{equation}\label{main-result1}
\left\{
\begin{aligned}
&\sup_{t\in[0,T]}\|u^\epsilon-u^{I,0}\|_{L^\infty(I)}\le C \epsilon^{\frac{7}{8}} ,\\
&\sup_{t\in[0,T]}\|\rho^\epsilon-\rho^{I,0}\|_{L^\infty(I)}\le C \sqrt\epsilon ,\\
&\sup_{t\in[0,T]}\|\rho^\epsilon_x-\rho^{I,0}_x-\sqrt{\epsilon}(\rho^{B,1}+\rho^{b,1})_x\|_{L^\infty(I)}\le C\sqrt\epsilon,\\
\end{aligned}\right.
\end{equation}
hold on $[0,T]$, where $(\rho^\epsilon,u^\epsilon)$, $(\rho^{I,0},u^{I,0})$, and $\rho^{B,1}$, $\rho^{b,1}$ are the solutions to (\ref{equ-epsilon}), (\ref{equ-NS}), (\ref{IBb-3}), and (\ref{Ib-3}), respectively.
\end{Theorem}

\begin{Remark}
Theorem \ref{following} provides a rigorous derivation of the isentropic compressible Navier-Stokes system from the isentropic compressible NSK system in a bounded interval. Since time $T$ is arbitrary, the convergence holds globally in time.
\end{Remark}
\begin{Remark}
We have derived optimal convergence rates for the density and its derivative in the presence of the boundary layer. The convergence rate for velocity can be improved to $O(\epsilon)$ by including the higher-order profiles $(\rho^{I,2},u^{I,2},\rho^{B,3},\rho^{b,3})$ with more suitable initial assumptions on the initial and boundary conditions.
\end{Remark}

\subsection{Error system}\label{section1.4}
~~~~Decompose the solution $(\rho^\epsilon, u^\epsilon)$ as follows:
\begin{equation}\label{zhankai}
\left\{
\begin{aligned}
&\rho^\epsilon=\rho^{I,0} +\sqrt{\epsilon}(\rho^{B,1}+\rho^{b,1})+{\epsilon}(\rho^{B,2}+\rho^{b,2})+L^\epsilon +\sqrt{\epsilon} \Phi^\epsilon,\\[2mm]
&u^\epsilon=u^{I,0} +\epsilon(u^{B,2}+u^{b,2})+\epsilon\sqrt{\epsilon}(u^{B,3}+u^{b,3})+M^\epsilon +\sqrt{\epsilon} U^\epsilon,
\end{aligned}\right.
\end{equation}
where
$$\begin{aligned}
L^\epsilon =&%-\rho_z^{b,1}(-\frac{1}{\sqrt{\epsilon}},t)-\sqrt{\epsilon}\rho_z^{b,2}(-\frac{1}{\sqrt{\epsilon}},t)
-\frac{x^2}{2}[\rho_y^{B,1}(\frac{1}{\sqrt{\epsilon}},t)+\sqrt{\epsilon}\rho_y^{B,2}(\frac{1}{\sqrt{\epsilon}},t)]-(x-\frac{x^2}{2})[\rho_z^{b,1}(-\frac{1}{\sqrt{\epsilon}},t)+\sqrt{\epsilon}\rho_z^{b,2}(-\frac{1}{\sqrt{\epsilon}},t)],\\
\end{aligned}
$$
$$\begin{aligned}
 M^\epsilon=&-\epsilon (1-x) u^{B,2}(0,t)-\epsilon(1-x)u^{b,2}(-\frac{1}{\sqrt{\epsilon}},t)-\epsilon x u^{B,2}(\frac{1}{\sqrt{\epsilon}},t)-\epsilon x u^{b,2}(0,t)\\
&-\epsilon \sqrt{\epsilon} (1-x) u^{B,3}(0,t)-\epsilon \sqrt{\epsilon}(1-x)u^{b,3}(-\frac{1}{\sqrt{\epsilon}},t)-\epsilon \sqrt{\epsilon}xu^{B,3}(\frac{1}{\sqrt{\epsilon}},t)-\epsilon \sqrt{\epsilon} x u^{b,3}(0,t).
\end{aligned}
$$
Thus, we get
$$
L^\epsilon_{xxx}=0, ~M^\epsilon_{xx}=0
$$
and
\begin{equation}\label{wucha-boun-con}
(\Phi^\epsilon_x ,U^\epsilon)|_{x=0,1}=(0,0).
\end{equation}

Putting (\ref{zhankai}) into (\ref{equ-epsilon}), and using (\ref{equ-NS}),  (\ref{IBb-3}), (\ref{IBb-2}),  (\ref{IBb-4}) and (\ref{IBb-5}), we get the error system of error functions $(\Phi^\epsilon,U^\epsilon)$ as follows:

\begin{equation}\label{zhankai-equ}
\left\{
	\begin{aligned}
&\Phi_t^\epsilon+(\rho^\epsilon U^\epsilon)_x=-h_1 -h_2-h_3-h_4-h_5,\\
&\rho^{\epsilon} U^\epsilon_t+\rho^\epsilon u^\epsilon U^\epsilon_x- U^\epsilon_{xx}-\epsilon\rho^\epsilon\Phi^{\epsilon}_{xxx}
=W,\\
&(\Phi^\epsilon ,U^\epsilon)|_{t=0}=(\frac{\rho_0^\epsilon -\rho_0}{\sqrt \epsilon}, \frac{u_0^\epsilon-u_0}{\sqrt \epsilon}),\\
&(\Phi^\epsilon_x ,U^\epsilon)|_{x=0,1}=(0,0),\\
\end{aligned}\right.
\end{equation}
where $h_1 = h_{1,1}+ h_{1,2}$ and $W=-\Theta-P(\rho^\epsilon, \rho^{I,0}) +\Omega+\alpha$ with
$$
\begin{aligned}
h_{1,1}=&  \frac{1}{\sqrt{\epsilon}}(u^{I,0}-\overline{u^{I,0}})\rho^{B,1}_y -y\overline{u^{I,0}_x}\rho^{B,1}_y + (\rho^{I,0}-\overline{\rho^{I,0}})u^{B,2}_y + \rho^{B,1} (u^{I,0}_x-\overline{u^{I,0}_x})   + \rho^{B,2}_y (u^{I,0}-\overline{u^{I,0}}) ,\\
h_{1,2}=&   \frac{1}{\sqrt{\epsilon}}(u^{I,0}-\widetilde{u^{I,0}})\rho^{b,1}_z-z\widetilde{u^{I,0}_x}\rho^{b,1}_z  + (\rho^{I,0}-\widetilde{\rho^{I,0}})u^{b,2}_z + \rho^{b,1} (u^{I,0}_x-\widetilde{u^{I,0}_x})  + \rho^{b,2}_z (u^{I,0}-\widetilde{u^{I,0}});\\
\end{aligned}
$$
$$\begin{aligned}
 h_2=&\frac{1}{\sqrt{\epsilon}}L^\epsilon_t  +u^{I,0}_x(\frac{1}{\sqrt{\epsilon}}L^\epsilon + \Phi^\epsilon) +u^{I,0}(\frac{1}{\sqrt{\epsilon}}L^\epsilon_x + \Phi^\epsilon_x)+\frac{1}{\sqrt{\epsilon}}(M^\epsilon \rho^\epsilon)_x\\
&+[(u^{B,2}_y+u^{b,2}_z)+\sqrt{\epsilon}(u^{B,3}_y+u^{b,3}_z)]\sqrt{\epsilon} \Phi^\epsilon+[\sqrt{\epsilon}(u^{B,2}+u^{b,2})+{\epsilon}(u^{B,3}+u^{b,3})]\sqrt{\epsilon} \Phi^\epsilon_x;\\
 h_3=&\sqrt{\epsilon}\rho^{B,2}_t  +u^{I,0}_x\sqrt{\epsilon}\rho^{B,2}+\sqrt{\epsilon}u^{B,3}_y\rho^{I,0} +\sqrt{\epsilon}u^{B,2}\rho^{I,0}_x+{\epsilon}u^{B,3}\rho^{I,0}_x;\\
h_4=&\sqrt{\epsilon}\rho^{b,2}_t  +u^{I,0}_x\sqrt{\epsilon}\rho^{b,2}+\sqrt{\epsilon}u^{b,3}_z\rho^{I,0} +\sqrt{\epsilon}u^{b,2}\rho^{I,0}_x+{\epsilon}u^{b,3}\rho^{I,0}_x;\\
h_5=&[(u^{B,2}_y+u^{b,2}_z)+\sqrt{\epsilon}(u^{B,3}_y+u^{b,3}_z)][\sqrt{\epsilon}(\rho^{B,1}+\rho^{b,1}) +{\epsilon}(\rho^{B,2}+\rho^{b,2})+L^\epsilon ]\\
&+[\sqrt{\epsilon}(u^{B,2}+u^{b,2})+{\epsilon}(u^{B,3}+u^{b,3})][(\rho^{B,1}_y+\rho^{b,1}_z) +\sqrt{\epsilon}(\rho^{B,2}_y+\rho^{b,2}_z)+L^\epsilon_x ];\\
\end{aligned}
$$
{$ \Theta:= \sum_{i=1}^{5} \Theta_i$} with
$$
\begin{aligned}
\Theta_1=&  (\frac{1}{\sqrt{\epsilon}}L^\epsilon +\Phi^\epsilon) u^{I,0}_t + \rho^\epsilon\frac{1}{\sqrt{\epsilon}}M^\epsilon_t+ \rho^{I,0} u^{I,0} \frac{1}{\sqrt{\epsilon}}M^\epsilon_x +\rho^{I,0} (\frac{1}{\sqrt{\epsilon}}M^\epsilon+U^\epsilon)u_x^{I,0}+ (\frac{1}{\sqrt{\epsilon}}L^\epsilon +\Phi^\epsilon)u^{I,0} u_x^{I,0},\\
\Theta_2=& \sqrt{\epsilon}\rho^{B,2} u^{I,0}_t + \rho^\epsilon(\sqrt{\epsilon}u^{B,2}_t+\epsilon u^{B,3}_t)+ \rho^{I,0} u^{I,0} \sqrt\epsilon u^{B,3}_y +\rho^{I,0} (\sqrt{\epsilon} u^{B,2}+\epsilon u^{B,3})u_x^{I,0}+ {\sqrt{\epsilon}}\rho^{B,2}u^{I,0} u_x^{I,0},\\
\Theta_3=& \sqrt{\epsilon}\rho^{b,2} u^{I,0}_t + \rho^\epsilon(\sqrt{\epsilon}u^{b,2}_t+\epsilon u^{b,3}_t)+ \rho^{I,0} u^{I,0} \sqrt\epsilon u^{b,3}_z +\rho^{I,0} (\sqrt{\epsilon} u^{b,2}+\epsilon u^{b,3})u_x^{I,0}+ {\sqrt{\epsilon}}\rho^{b,2}u^{I,0} u_x^{I,0},\\
\Theta_4=&\rho^{I,0}[\sqrt{\epsilon}(u^{B,2}+u^{b,2})+\epsilon(u^{B,3}+u^{b,3})+\frac{1}{\sqrt{\epsilon}}M^\epsilon + U^\epsilon] [M^\epsilon + \epsilon(u^{B,2}+u^{b,2})+\epsilon\sqrt{\epsilon}(u^{B,3}+u^{b,3})]_x\\
&+[(\rho^{B,1}+\rho^{b,1})+\sqrt{\epsilon}(\rho^{B,2}+\rho^{b,2})+\frac{1}{\sqrt{\epsilon}}L^\epsilon + \Phi^\epsilon]u^{I,0} [M^\epsilon + \epsilon(u^{B,2}+u^{b,2})+\epsilon\sqrt{\epsilon}(u^{B,3}+u^{b,3})]_x,\\
\Theta_5=&[(\rho^{B,1}+\rho^{b,1})+\sqrt{\epsilon}(\rho^{B,2}+\rho^{b,2})+\frac{1}{\sqrt{\epsilon}}L^\epsilon + \Phi^\epsilon][\epsilon(u^{B,2}+u^{b,2})+\epsilon\sqrt{\epsilon}(u^{B,3}+u^{b,3})\\
&+M^\epsilon +\sqrt{\epsilon} U^\epsilon] [u^{I,0}+M^\epsilon + \epsilon(u^{B,2}+u^{b,2})+\epsilon\sqrt{\epsilon}(u^{B,3}+u^{b,3})]_x;\\
\end{aligned}
$$
$$
\begin{aligned}
P(\rho^{\epsilon},\rho^{I,0}):=&\frac{1}{\sqrt{\epsilon}}[\gamma(\rho^\epsilon)^{\gamma-1}\rho^\epsilon_x-\gamma(\rho^{I,0})^{\gamma-1}\rho^{I,0}_x-\gamma (\rho^{I,0})^{\gamma-1}( \rho^{B,1}_y+\rho^{b,1}_z)\\
&-\gamma (\rho^{I,0})^{\gamma-1}\sqrt{\epsilon}(\rho^{B,2}_y+\rho^{b,2}_z)-\gamma (\gamma-1) (\rho^{I,0})^{\gamma-2} \sqrt{\epsilon}(\rho^{B,1}+\rho^{b,1}) \rho^{I,0}_x\\
&-\gamma (\gamma-1) (\rho^{I,0})^{\gamma-2}\sqrt{\epsilon}\rho^{B,1}\rho^{B,1}_y -\gamma (\gamma-1) (\rho^{I,0})^{\gamma-2}\sqrt{\epsilon}\rho^{b,1}\rho^{b,1}_z
];\\
\end{aligned}
$$
$ \Omega:= \sum_{i=1}^4 \Omega_i$ with
 $$\begin{aligned}
\Omega_1=&\sqrt{\epsilon} \rho^{I,0} \rho^{I,0}_{xxx},\\
\Omega_2=&\epsilon (\rho^{B,1}+\rho^{b,1})[\rho^{I,0}+{\epsilon}(\rho^{B,2}+\rho^{b,2})]_{xxx},\\
\Omega_3=&\sqrt{\epsilon }[{\epsilon}(\rho^{B,2}+\rho^{b,2})+L^\epsilon +\sqrt{\epsilon} \Phi^\epsilon] [\rho^{I,0}+\sqrt{\epsilon}(\rho^{B,1}+\rho^{b,1})+{\epsilon}(\rho^{B,2}+\rho^{b,2})]_{xxx},\\
\Omega_4=& (\rho^{b,1}\rho^{B,1}_{yyy}+ \rho^{B,1}\rho^{b,1}_{zzz});
\end{aligned}
$$
and $\alpha:= \frac{1}{\sqrt\epsilon}(\alpha_1 +\alpha_2)$ with
$$
\begin{aligned}
\alpha_1=&-\{\sqrt{\epsilon}\rho^{B,1}(u^{I,0}_t-\overline{u^{I,0}_t})+ \sqrt{\epsilon}u^{B,2}_y(\rho^{I,0}u^{I,0}- \overline{\rho^{I,0}}\overline{u^{I,0}})+ \sqrt{\epsilon}\rho^{B,1}(u^{I,0}u^{I,0}_x-\overline{u^{I,0}}\overline{u^{I,0}_x })\\
&+\gamma [(\rho^{I,0})^{\gamma-1}- (\overline{\rho^{I,0}})^{\gamma-1}]\sqrt{\epsilon}\rho^{B,2}_y+\gamma (\gamma-1) [(\rho^{I,0})^{\gamma-2}\rho^{I,0}_x- (\overline{\rho^{I,0}})^{\gamma-2} \overline{\rho^{I,0}_x}] \sqrt{\epsilon}\rho^{B,1} \\
&+\gamma (\gamma-1) [(\rho^{I,0})^{\gamma-2}- (\overline{\rho^{I,0}})^{\gamma-2}] \sqrt{\epsilon}\rho^{B,1}\rho^{B,1}_y-\gamma (\gamma-1)y  (\overline{\rho^{I,0}})^{\gamma-2} \overline{\rho^{I,0}_x} \sqrt{\epsilon}\rho^{B,1}_y\\
& +\gamma [(\rho^{I,0})^{\gamma-1} - (\overline{\rho^{I,0}})^{\gamma-1}] \rho^{B,1}_y\}+\sqrt{\epsilon}(\rho^{I,0}- \overline{\rho^{I,0}}) \rho^{B,2}_{yyy} -y\sqrt{\epsilon}\overline{\rho^{I,0}_x}\rho^{B,1}_{yyy}+(\rho^{I,0}- \overline{\rho^{I,0}})\rho^{B,1}_{yyy},\\
\alpha_2=&-\{\sqrt{\epsilon}\rho^{b,1}(u^{I,0}_t-\widetilde{u^{I,0}_t}) +\sqrt{\epsilon}u^{b,2}_z(\rho^{I,0}u^{I,0}- \widetilde{\rho^{I,0}}\widetilde{u^{I,0}}) +\sqrt{\epsilon} \rho^{b,1}(u^{I,0}u^{I,0}_x-\widetilde{u^{I,0}} \widetilde{u^{I,0}_x})\\
&+\gamma[(\rho^{I,0})^{\gamma-1}-(\widetilde{\rho^{I,0}})^{\gamma-1}]\sqrt{\epsilon}\rho^{b,2}_z+\gamma (\gamma-1) [(\rho^{I,0})^{\gamma-2}\rho^{I,0}_x- (\widetilde{\rho^{I,0}})^{\gamma-2} \widetilde{\rho^{I,0}_x}]  \sqrt{\epsilon}\rho^{b,1}\\
&+\gamma (\gamma-1) [(\rho^{I,0})^{\gamma-2}-(\widetilde{\rho^{I,0}})^{\gamma-2}] \sqrt{\epsilon}\rho^{b,1}\rho^{b,1}_z-\gamma (\gamma-1) z (\widetilde{\rho^{I,0}})^{\gamma-2} \widetilde{\rho^{I,0}_x}\sqrt{\epsilon}\rho^{b,1}_z \\
& +\gamma [(\rho^{I,0})^{\gamma-1}-(\widetilde{\rho^{I,0}})^{\gamma-1}]\rho^{b,1}_z\}+\sqrt{\epsilon}(\rho^{I,0}- \widetilde{\rho^{I,0}})\rho^{b,2}_{zzz}-\sqrt{\epsilon}z\widetilde{\rho^{I,0}_x} \rho^{b,1}_{zzz} + (\rho^{I,0}- \widetilde{\rho^{I,0}})\rho^{b,1}_{zzz}.\\
\end{aligned}
$$

\subsection{Difficulties and innovations}\label{section1.5}
~~~~In this work, we aim to investigate vanishing capillary limit of system (\ref{equ-epsilon}) in a bounded interval. The main challenge focuses on the boundary effect that occurs in the limit of the density gradient, which prevents the convergence between the derivative of solution of the original system and that of the target system. Our main strategy is to derive some new correctors handling difficulties due to boundary layer. More precisely, our goal is to derive convergence rates of density, spatial derivative of density and velocity in \(L^\infty\) norm. The key step is to  derive \(L^\infty\) estimates of the error functions \((\Phi^\epsilon,\Phi_x^\epsilon,U^\epsilon)\). By interpolation inequalities, it suffices to establish \(L^2\) estimates of \((\Phi^\epsilon,\Phi_x^\epsilon,U^\epsilon)\) and its derivative.

One of the main difficulties lies in the $L^2$ estimate of
$(\Phi_x^\epsilon, \sqrt\epsilon\,\Phi_{xx}^\epsilon).$
Indeed, when deriving the energy differential inequality
\[
\frac{\mathrm{d}}{\mathrm{d}t}\Big(\|\Phi^\epsilon_x\|_{L^2}^2
+ \|\sqrt{\rho^\epsilon}\,U^\epsilon_x\|_{L^2}^2
+ \|\sqrt\epsilon\,\Phi^\epsilon_{xx}\|_{L^2}^2\Big)
+ \|U^\epsilon_{xx}\|_{L^2}^2 \le C_R ,
\]
 the right-hand side $C_R$ contains the terms
$
\int_I \epsilon\rho^\epsilon_x U^\epsilon_x \Phi^\epsilon_{xxx}\,\mathrm{d}x
+ \int_I \epsilon\rho^\epsilon_{xx} U^\epsilon \Phi^\epsilon_{xxx}\,\mathrm{d}x$ (see $G_2$ and $G_3$  in Lemma \ref{Phi-U-H^1-2})
that cannot be absorbed by the dissipation on the left-hand side. Motivated by \cite{H2020}, we derive a dissipative estimate of $\epsilon\rho^\epsilon\Phi^\epsilon_{xxx}$ in $L^2_{xt}$ to handle such a difficulty. In contrast, the estimate of \(\|(\Phi^\epsilon, U^\epsilon, \sqrt\epsilon\,\Phi^\epsilon_x)\|_{L^\infty_T L^2(I)}\) can be closed %employs the energy inequality
%\[
%\frac{\mathrm{d}}{\mathrm{d}t}\|(\Phi^\epsilon, U^\epsilon, \sqrt\epsilon\,\Phi^\epsilon_x)\|_{L^2(I)}^2
%+ \|U^\epsilon_x\|_{L^2(I)}^2 \le C_r,
%\]
%and the term \(C_r\) can be controlled Gronwall's inequality
 without requiring extra dissipative estimate of \(\Phi_{xx}^\epsilon\).

\section{Preliminaries}\label{section2}

\begin{Lemma}\label{some-inequ}
(1) For any $m\in \mathbb{N}$, it holds that
\begin{equation}\label{L2-B-b}
\left\{
\begin{aligned}
& \int_0^1 |f^{b,m}(\frac{x-1}{\sqrt{\epsilon}})|^2\mathrm{d}x\le \sqrt{\epsilon} \|f^{b,m}\|_{L_z^2}^2,\,\,
\int_0^1 |f^{B,m}(\frac{x}{\sqrt{\epsilon}})|^2\mathrm{d}x\le \sqrt{\epsilon} \|f^{B,m}\|_{L_y^2}^2,\\
&%\int_0^1 |h^{b,m}(\frac{x-1}{\sqrt{\epsilon}})|^2\mathrm{d}x \le
  \|h^{b,m}(\frac{\cdot-1}{\sqrt{\epsilon}})\|_{L^\infty(I)}^2 \le C \|h^{b,m}\|_{L^\infty_z}^2,\,\,
%\int_0^1 |h^{B,m}(\frac{x}{\sqrt{\epsilon}})|^2\mathrm{d}x\le
 \|h^{B,m}(\frac{\cdot}{\sqrt{\epsilon}})\|_{L^\infty(I)}^2\le C \|h^{B,m}\|_{L^\infty_y}^2,\\
&\int_0^1 |g_x^{b,m}(\frac{x-1}{\sqrt{\epsilon}},t)|^2\mathrm{d}x\le \frac{1}{\sqrt{\epsilon} }\|g_z^{b,m}\|_{L_z^2}^2,\,\,
\int_0^1 |g_x^{B,m}(\frac{x}{\sqrt{\epsilon}},t)|^2\mathrm{d}x\le \frac{1}{\sqrt{\epsilon} }\|g_y^{B,m}\|_{L_y^2}^2,
\end{aligned}\right.
\end{equation} for $f^{b,m},f^{B,m}\in L^2(I)$, $h^{b,m},h^{B,m}\in  L^\infty (I)$ and $g^{b,m},g^{B,m}\in H^1(I)$.

(2) Under the conditions of Propositions \ref{exi-ns}, \ref{exi-nsk}, and \ref{B-b-Pro}, there holds
\begin{equation}\label{rho-epsilon-phi-epsilon}
\begin{aligned}
&\|\rho^\epsilon_x(\cdot, t)-\sqrt\epsilon\Phi^{\epsilon}_x\|_{L^2(I)}^2 \le C
\end{aligned}
\end{equation}
for any $t\in[0,T_\epsilon]\cap[0,T]$.
\end{Lemma}
\begin{proof}
Since $f^{b,m},f^{B,m}\in L^2(I)$, $h^{b,m},h^{B,m}\in  L^\infty (I)$ and $g^{b,m},g^{B,m}\in H^1(I)$, it is easy to obtain the estimates in (1).
We will prove (2) as follows. Using (\ref{L2-B-b}), (\ref{zhankai}) and Propositions \ref{exi-ns}, \ref{exi-nsk}, \ref{B-b-Pro}, we get
\begin{equation}\label{rho-epsilon-phi-epsilon}
\begin{aligned}
&\|\rho^\epsilon_x(\cdot, t)-\sqrt\epsilon\Phi^{\epsilon}_x\|_{L^2(I)}^2\\=&\|\rho^{I,0}_x +\sqrt{\epsilon}(\rho^{B,1}+\rho^{b,1})_x+{\epsilon}(\rho^{B,2}+\rho^{b,2})_x+L^\epsilon_x\|_{L^2(I)}^2 \\
\le & C(\|\rho^{I,0}_x\|_{L^2(I)}^2 +\sqrt{\epsilon}\|\rho^{B,1}_{y}\|_{L_y^2}^2+\sqrt{\epsilon}\|\rho^{b,1}_{z}\|_{L_z^2}^2 +\epsilon\sqrt{\epsilon}\|\rho^{B,2}_{y}\|_{L_y^2}^2+\epsilon\sqrt{\epsilon}\|\rho^{b,2}_{z}\|_{L_z^2}^2+\|L^\epsilon_x\|_{L^2(I)}^2)\\ \le& C
\end{aligned}
\end{equation}
for any $t\in[0,T_\epsilon]\cap[0,T]$.
\end{proof}

\begin{Lemma}\label{L-M-well-pose}
Under the conditions of Propositions \ref{B-b-Pro}, there holds that
\begin{equation}\label{L-M-esti}
\left\{
\begin{aligned}
&\|L^\epsilon\|_{L_T^\infty H^2}+\|L^\epsilon_t\|_{L_T^\infty H^1}+\|L^\epsilon_{tt}\|_{L_T^2 L^2} \le C \epsilon^{\frac{l}{2}}, ~\mathrm{for~ any} ~l\in {\mathbb{Z}}^+,\\
&\|M^\epsilon\|_{L_T^\infty H^1} +\|M^\epsilon_t\|_{L_T^\infty H^1(I)} \le C \epsilon.
\end{aligned}\right.
\end{equation}

\end{Lemma}
\begin{proof}
Firstly,
from
$$\begin{aligned}
&\|\rho^{B,i}_y(\frac{1}{\sqrt{\epsilon}},t)\|_{L^\infty_T}=\epsilon^{\frac{l}{2}}\|\epsilon^{-\frac{l}{2}}\rho^{B,i}_y(\frac{1}{\sqrt{\epsilon}},t)\|_{L^\infty_T} \\
\le &\epsilon^{\frac{l}{2}}\|\langle y \rangle^{l}\rho^{B,i}_y(y,t)\|_{L^\infty_T L^\infty_y} \le \epsilon^{\frac{l}{2}}\|\langle y \rangle^{l}\rho^{B,i}_y(y,t)\|_{L^\infty_T H^1_y} \le C \epsilon^{\frac{l}{2}},\\[2mm]
&\|\rho^{B,i}_{yt}(\frac{1}{\sqrt{\epsilon}},t)\|_{L^\infty_T}\le \epsilon^{\frac{l}{2}}\|\langle y \rangle^{l}\rho^{B,i}_{yt}(y,t)\|_{L^\infty_T H^1_y} \le C \epsilon^{\frac{l}{2}},\\[2mm]
&\|\rho^{B,i}_{ytt}(\frac{1}{\sqrt{\epsilon}},t)\|_{L^2_T} \le \epsilon^{\frac{l}{2}}\|\langle y \rangle^{l}\rho^{B,i}_{ytt}(y,t)\|_{L^2_T L^\infty_y} \le \epsilon^{\frac{l}{2}}\|\langle y \rangle^{l}\rho^{B,i}_{ytt}(y,t)\|_{L^2_T H^1_y} \le C \epsilon^{\frac{l}{2}},\\
\end{aligned}
$$
and
$$
\|\rho^{b,i}_z(-\frac{1}{\sqrt{\epsilon}},t)\|_{L^\infty_T} +\|\rho^{b,i}_{zt}(-\frac{1}{\sqrt{\epsilon}},t)\|_{L^\infty_T} + \|\rho^{b,i}_{ztt}(-\frac{1}{\sqrt{\epsilon}},t)\|_{L^2_T}\le C \epsilon^{\frac{l}{2}},
$$
for any $i=0,1,2$,
(\ref{L-M-esti})$_1$ holds.

Next, %because of the regularity of $u^{B,2}, ~u^{B,3},~u^{b,2},~u^{b,3}$ in Proposition \ref{B-b-Pro} and because of the structure of $M^\epsilon$,
from
$$
\|(\epsilon (1-x) u^{B,2}(0,t),\epsilon x u^{b,2}(0,t))\|_{L^\infty_T} \le C \epsilon ,
$$
 (\ref{L-M-esti})$_2$ holds.
\end{proof}

\section{Proof of Theorem \ref{following}}\label{section3}
%~~~~In this section, the general constant $C$ will be independent of $\frac{1}{\epsilon}$ and will only be positive related to time $T$ being defined in Theorem \ref{following}.

Propositions \ref{exi-ns} and \ref{B-b-Pro} imply that systems (\ref{equ-NS}), (\ref{IBb-3})-(\ref{rho-B2-inibou}) have unique global solution $$(\rho^{I,0},u^{I,0},\rho^{B,1},\rho^{b,1},\rho^{B,2},\rho^{b,2},u^{B,2},u^{b,2},u^{B,3},u^{b,3}).$$ Proposition \ref{exi-nsk} implies that for any fixed $\epsilon$, there exists a unique local solution $(\rho^\epsilon, u^\epsilon)$ to system (\ref{equ-epsilon}). Thus, system (\ref{zhankai-equ}) admits a local solution $(\Phi^\epsilon,U^\epsilon)$.

For any fixed $\epsilon$, suppose that $T^\epsilon>0$ is the maximum existence time of the system (\ref{zhankai-equ}) with solution $(\Phi^\epsilon,U^\epsilon)$. %Then, for any positive constant $T$ given in Theorem \ref{following}, if $T< T^\epsilon$, we can do a priori estimate of system (\ref{zhankai-equ}) in the time inteval $[0,T]$. Thus, we assume that $T^\epsilon\le T$.

To prove Theorem \ref{following}, it is necessary to prove the following result.
\begin{Theorem}\label{assumption}
%We decompose $(\rho^\epsilon, u^\epsilon)$ as (\ref{zhankai}). Let the conditions of $(\rho^{I,0}_0,u^{I,0}_0,\rho^\epsilon_0,u^\epsilon_0)$ be given in Theorem \ref{following}.

Under the conditions of Theorem \ref{following}, there exists a positive constant $\epsilon_0$ depending only on $T$ and the initial data, such that if
\begin{equation}\label{ini-val-phi-1}
\sup_{t\in[0,T_1]}\|\Phi^\epsilon(\cdot,t)\|_{L^\infty(I)} \le 1
\end{equation}
holds for $T_1\in(0,T]\cap(0,T^\epsilon)$ and $\epsilon \in(0,\epsilon_0)$ where $T$ is given in Theorem \ref{following}, there hold
\begin{equation}\label{1234}
\left\{
\begin{aligned}
&\sup_{t\in[0,T_1]}(\|U^\epsilon\|_{ L^2(I)}^2+\|\Phi^\epsilon\|_{ L^2(I)}^2+\epsilon\|\Phi^\epsilon_x\|_{ L^2(I)}^2)+\|U_x^\epsilon\|_{L_{T_1}^2 L^2}^2\le B_1 {\epsilon},\\
&\sup_{t\in[0,T_1]}(\|U_x^\epsilon\|_{ L^2(I)}^2+\|\Phi_x^\epsilon\|_{ L^2(I)}^2+\epsilon\|\Phi_{xx}^\epsilon\|_{ L^2(I)}^2)+\epsilon^2\|\Phi_{xxx}^\epsilon\|_{L_{T_1}^2 L^2}^2+\|U_{xx}^\epsilon\|_{L_{T_1}^2 L^2}^2\le B_2 {\sqrt{\epsilon}},\\
&\sup_{t\in[0,T_1]}\|\Phi^\epsilon(\cdot,t)\|_{L^\infty(I)} \le \frac{1}{2},
\end{aligned}\right.
\end{equation}
where $B_1, B_2$ are positive constants positively correlated with  $T$ and depend on $(\rho_0^\epsilon,u_0^\epsilon, \rho_0^{I,0},u_0^{I,0})$.
\end{Theorem}

\begin{Corollary}\label{following-4}
Under the conditions of Theorem \ref{following}, there exists a positive constant $\epsilon_1\le\epsilon_0$ depending only on $T$ and the initial data, such that if $\epsilon\le\epsilon_1$, then
\begin{equation}
\left\{
\begin{aligned}
&\|U^\epsilon\|_{L_{T}^\infty L^2}^2+\|\Phi^\epsilon\|_{L_{T}^\infty L^2}^2+\epsilon\|\Phi^\epsilon_x\|_{L_{T}^\infty L^2}^2+\|U_x^\epsilon\|_{L_{T}^2 L^2}^2\le B_1 {\epsilon},\\
&\|U_x^\epsilon\|_{L_{T}^\infty L^2}^2+\|\Phi_x^\epsilon\|_{L_{T}^\infty L^2}^2+\epsilon\|\Phi_{xx}^\epsilon\|_{L_{T}^\infty L^2}^2+\epsilon^2\|\Phi_{xxx}^\epsilon\|_{L_{T}^2 L^2}^2+\|U_{xx}^\epsilon\|_{L_{T}^2 L^2}^2\le B_2 {\sqrt{\epsilon}},\\
&\|\Phi^\epsilon\|_{L_{T}^\infty L^\infty} \le \frac{1}{2}\\
\end{aligned}\right.
\end{equation}
hold.
\end{Corollary}
\begin{Remark}
$\epsilon_0$ and $\epsilon_1$ are determined by (\ref{ep0-4}) and (\ref{ep1}), respectively.
\end{Remark}

\subsection{Uniform estimates for error system}
In the following estimates we only apply the case $l=4$ in (\ref{L-M-esti}), i.e., $$\|L^\epsilon\|_{L_T^\infty H^2}+\|L^\epsilon_t\|_{L_T^\infty H^1}+\|L^\epsilon_{tt}\|_{L_T^2 L^2} \le C \epsilon^2,$$ which suffices for our purpose.
\begin{Lemma}\label{clm-rho-bdd}
Under the conditions of Theorem \ref{assumption}, there holds that
$$
\frac{1}{2M}\le \rho^\epsilon \le 2 M , ~(x,t) \in I\times [0,T_1],
$$
where $M>1$ is defined in Proposition \ref{exi-ns}.
\end{Lemma}
\begin{proof}
Using (\ref{zhankai}) and $\frac{1}{ M}\le\rho^{I,0}\le M$ in Proposition \ref{exi-ns}, it holds that
$$\begin{aligned}
&\frac{1}{ M} + \sqrt{\epsilon}(\rho^{B,1}+\rho^{b,1})+{\epsilon}(\rho^{B,2}+\rho^{b,2})+L^\epsilon +\sqrt{\epsilon} \Phi^\epsilon \\
\le &\rho^\epsilon \le  M +\sqrt{\epsilon}(\rho^{B,1}+\rho^{b,1})+{\epsilon}(\rho^{B,2}+\rho^{b,2})+L^\epsilon +\sqrt{\epsilon} \Phi^\epsilon.
\end{aligned}
$$

Using  Proposition \ref{B-b-Pro} and (\ref{ini-val-phi-1}), we have
$$\begin{aligned}
&\frac{1}{ M} + \sqrt{\epsilon}(\rho^{B,1}+\rho^{b,1})+{\epsilon}(\rho^{B,2}+\rho^{b,2})+L^\epsilon +\sqrt{\epsilon} \Phi^\epsilon \\
\ge &\frac{1}{ M} - \sqrt{\epsilon}(\|\rho^{B,1}\|_{L^\infty_y}+\|\rho^{b,1}\|_{L^\infty_z})-{\epsilon}(\|\rho^{B,2}\|_{L^\infty_y}+\|\rho^{b,2}\|_{L^\infty_z})-\|L^\epsilon\|_{L^\infty(I)} -\sqrt{\epsilon} \|\Phi^\epsilon\|_{L^\infty(I)}\\
\ge &\frac{1}{M} -\sqrt\epsilon C-\sqrt \epsilon,
\end{aligned}
$$
 for any $t\in[0,T_1]$.
Choosing
\begin{equation}\label{ep0-1}
\epsilon_0 \le \frac{1}{2^2M^2( C+1)^2},
\end{equation}
and letting $\epsilon\le\epsilon_0$, we have
\begin{equation*}
0<\sqrt\epsilon \le \sqrt{\epsilon_0}\le \frac{1}{2M( C+1)}.
\end{equation*}
Hence, we obtain
$$
\rho^\epsilon \ge \frac{1}{2M}>0,
$$
and
$$\begin{aligned}
\rho^\epsilon \le
&{ M} + \sqrt{\epsilon}(\rho^{B,1}+\rho^{b,1})+{\epsilon}(\rho^{B,2}+\rho^{b,2})+L^\epsilon +\sqrt{\epsilon} \Phi^\epsilon \\
\le &{M} +\sqrt{\epsilon}(\|\rho^{B,1}\|_{L^\infty_y}+\|\rho^{b,1}\|_{L^\infty_z})+{\epsilon}(\|\rho^{B,2}\|_{L^\infty_y}+\|\rho^{b,2}\|_{L^\infty_z})+\|L^\epsilon\|_{L^\infty(I)} +\sqrt{\epsilon} \|\Phi^\epsilon\|_{L^\infty(I)}\\
\le &{ M} +\sqrt\epsilon  C(T_1)+\sqrt \epsilon \le { M} +\frac{1}{2 M}\le 2M.
\end{aligned}
$$
Thus, the proof of Lemma \ref{clm-rho-bdd} is complete.
\end{proof}

To prove Theorem \ref{assumption}, we need to estimate some terms in (\ref{zhankai-equ}).
\begin{Lemma}\label{clm:h}
Under the conditions of Theorem \ref{assumption}, there holds that
\begin{equation}\label{h-alpha-bata}
\left\{
\begin{aligned}
&\|h_1\|_{L^2(I)}^2 \le C\epsilon \sqrt{\epsilon},~\|(h_1)_x\|_{L^2(I)}^2 \le C\sqrt{\epsilon}, ~\|( h_1)_t \|_{L^2(I)}^2
\le  C \epsilon \sqrt{\epsilon},~\|\alpha\|_{L^2(I)}^2\le  C\epsilon \sqrt{\epsilon},\\[2mm]
&
|U^\epsilon(x,t)| \le \|U^\epsilon_x\|_{L^2(I)} |\sqrt \epsilon y|^{\frac{1}{2}}, ~|U^\epsilon(x,t)| \le \|U^\epsilon_x\|_{L^2(I)} |\sqrt \epsilon z|^{\frac{1}{2}}.\\
\end{aligned}\right.
\end{equation}
\end{Lemma}
\begin{proof}%[Proof of Claim \ref{clm:h}]
%This proof will use Propositions \ref{exi-ns} and \ref{B-b-Pro}.

Firstly, we define $\xi =\xi(x,t)$
%and $\zeta=\zeta(x,t)$ means
as the gengral function when we use Taylor expansion to solve the problem, which will be different from line to line.

Using  (\ref{L2-B-b}), Propositions \ref{exi-ns} and \ref{B-b-Pro}, %mean value Theorem
  and Taylor expansion
$$\begin{aligned}
u^{I,0}(x,t)-\overline{u^{I,0}} -\sqrt{\epsilon} y\overline{u^{I,0}_x}= \frac{\epsilon y^2}{2} {u^{I,0}_{\xi\xi}(\xi,t)},
\end{aligned}
$$
we have
$$\begin{aligned}
 &\| h_{1,1}\|_{L^2(I)}^2\\
=&\| [\frac{1}{\sqrt{\epsilon}}(u^{I,0}-\overline{u^{I,0}})-y\overline{u^{I,0}_x}]\rho^{B,1}_y+ (\rho^{I,0}-\overline{\rho^{I,0}})u^{B,2}_y  + \rho^{B,1} (u^{I,0}_x-\overline{u^{I,0}_x})  + \rho^{B,2}_y (u^{I,0}-\overline{u^{I,0}}) \|_{L^2(I)}^2\\
= & \| \frac{\sqrt\epsilon y^2}{2} {u^{I,0}_{xx}(\xi,t)}  \rho^{B,1}_y + y \rho^{I,0}_x (\xi, t) \sqrt{\epsilon}u^{B,2}_y + y u^{I,0}_{xx} (\xi,t)\sqrt{\epsilon}\rho^{B,1}   + y u^{I,0}_x (\xi,t)\sqrt{\epsilon}\rho^{B,2}_y \|_{L^2(I)}^2\\
\le & C\|u^{I,0}_{xx}\|_{L^\infty(I)}^2 \epsilon \sqrt{\epsilon} \| \langle y \rangle^2 \rho^{B,1}_y \|_{L_y^2}^2+ C\|\rho^{I,0}_{x}\|_{L^\infty(I)}^2 \epsilon \sqrt{\epsilon} \| \langle y \rangle u^{B,2}_y \|_{L_y^2}^2 +C\|u^{I,0}_{xx}\|_{L^\infty(I)}^2 \epsilon \sqrt{\epsilon}  \| \langle y \rangle \rho^{B,1} \|_{L_y^2}^2\\
&+ C\|u^{I,0}_{x}\|_{L^\infty(I)}^2 \epsilon \sqrt{\epsilon} \| \langle y \rangle \rho^{B,2}_y \|_{L_y^2}^2\\
\le & C\epsilon \sqrt{\epsilon}.
\end{aligned}
$$
Similarly, it holds that
$$\begin{aligned}
 &\|h_{1,2}\|_{L^2(I)}^2
\le  C\epsilon \sqrt{\epsilon}.
\end{aligned}
$$

Secondly, we use  (\ref{L2-B-b}),  Propositions \ref{exi-ns} and \ref{B-b-Pro}, %mean value Theorem
  and Taylor expansion to get
$$\begin{aligned}
 &\|\partial_x h_{1,1}\|_{L^2(I)}^2\\
=&\| \frac{1}{\sqrt{\epsilon}}[\frac{1}{\sqrt{\epsilon}}(u^{I,0}-\overline{u^{I,0}})-y\overline{u^{I,0}_x}]\rho^{B,1}_{yy}+\frac{1}{\sqrt{\epsilon}}(u^{I,0}_x-\overline{u^{I,0}_x})\rho^{B,1}_y +\frac{1}{\sqrt{\epsilon}} (\rho^{I,0}-\overline{\rho^{I,0}})u^{B,2}_{yy}+ \rho^{I,0}_x u^{B,2}_y\\
& +\frac{1}{\sqrt{\epsilon}} \rho^{B,1}_y (u^{I,0}_x-\overline{u^{I,0}_x}) + \rho^{B,1} u^{I,0}_{xx}+\frac{1}{\sqrt{\epsilon}} \rho^{B,2}_{yy} (u^{I,0}-\overline{u^{I,0}}) + \rho^{B,2}_y u^{I,0}_{x}\|_{L^2(I)}^2\\
\le & C\|u^{I,0}_{xx}\|_{L^\infty(I)}^2  \sqrt{\epsilon}( \| \langle y \rangle^2 \rho^{B,1}_{yy} \|_{L_y^2}^2+\| \langle y \rangle \rho^{B,1}_y \|_{L_y^2}^2)+C\|u^{I,0}_{xx}\|_{L^\infty(I)}^2  \sqrt{\epsilon}  \|  \rho^{B,1} \|_{L_y^2}^2\\
&+ C\|\rho^{I,0}_{x}\|_{L^\infty(I)}^2  \sqrt{\epsilon} (\| \langle y \rangle u^{B,2}_{yy} \|_{L_y^2}^2+\| u^{B,2}_{y} \|_{L_y^2}^2)+C\|u^{I,0}_{x}\|_{L^\infty(I)}^2  \sqrt{\epsilon} \| \langle y \rangle \rho^{B,2}_{yy} \|_{L_y^2}^2\\
&+C\|u^{I,0}_{x}\|_{L^\infty(I)}^2  \sqrt{\epsilon}  \|  \rho^{B,2}_y \|_{L_y^2}^2\\
\le &C \sqrt{\epsilon}.
\end{aligned}
$$
Similarly, there holds
$$\begin{aligned}
 &\|\partial_x h_{1,2}\|_{L^2(I)}^2
\le  C \sqrt{\epsilon}.
\end{aligned}
$$

Thirdly, using  (\ref{L2-B-b}),  Propositions \ref{exi-ns} and \ref{B-b-Pro}, %mean value Theorem
  and Taylor expansion, we get
$$\begin{aligned}
 &\|\partial_t h_{1,1} \|_{L^2(I)}^2\\
=&\| [\frac{1}{\sqrt{\epsilon}}(u^{I,0}-\overline{u^{I,0}})-y\overline{u^{I,0}_x}]\rho^{B,1}_{yt}+[\frac{1}{\sqrt{\epsilon}}(u^{I,0}_t-\overline{u^{I,0}_t})-y\overline{u^{I,0}_{xt}}]\rho^{B,1}_y
+ (\rho^{I,0}-\overline{\rho^{I,0}})u^{B,2}_{yt} \\
&+ (\rho^{I,0}_t-\overline{\rho^{I,0}_t})u^{B,2}_y + \rho^{B,1}_t (u^{I,0}_x-\overline{u^{I,0}_x})  + \rho^{B,1} (u^{I,0}_{xt}-\overline{u^{I,0}_{xt}})  + \rho^{B,2}_{yt} (u^{I,0}-\overline{u^{I,0}}) \\
&+ \rho^{B,2}_y (u^{I,0}_t-\overline{u^{I,0}_t}) \|_{L^2(I)}^2\\
\le & C\|u^{I,0}_{xx}\|_{L^\infty(I)}^2 \epsilon \sqrt{\epsilon} \| \langle y \rangle^2 \rho^{B,1}_{yt} \|_{L_y^2}^2+C\|u^{I,0}_{xxt}\|_{L^\infty(I)}^2 \epsilon \sqrt{\epsilon} \| \langle y \rangle^2 \rho^{B,1}_{y} \|_{L_y^2}^2\\
&+ C\|\rho^{I,0}_{x}\|_{L^\infty(I)}^2 \epsilon \sqrt{\epsilon} \| \langle y \rangle u^{B,2}_{yt} \|_{L_y^2}^2+ C\|\rho^{I,0}_{xt}\|_{L^\infty(I)}^2 \epsilon \sqrt{\epsilon} \| \langle y \rangle u^{B,2}_y \|_{L_y^2}^2\\
&+C\|u^{I,0}_{xx}\|_{L^\infty(I)}^2 \epsilon \sqrt{\epsilon}  \| \langle y \rangle\rho^{B,1}_t \|_{L_y^2}^2 +C\|u^{I,0}_{xxt}\|_{L^\infty(I)}^2 \epsilon \sqrt{\epsilon}  \| \langle y \rangle \rho^{B,1} \|_{L_y^2}^2 \\
&+ C\|u^{I,0}_{x}\|_{L^\infty(I)}^2 \epsilon \sqrt{\epsilon} \| \langle y \rangle \rho^{B,2}_{yt} \|_{L_y^2}^2+ C\|u^{I,0}_{xt}\|_{L^\infty(I)}^2 \epsilon \sqrt{\epsilon} \| \langle y \rangle \rho^{B,2}_y \|_{L_y^2}^2\\
\le & C\epsilon \sqrt{\epsilon} .
\end{aligned}
$$
Similarly, we have
$$\begin{aligned}
 &\|\partial_t h_{1,2}\|_{L^2(I)}^2
\le  C \epsilon\sqrt{\epsilon}.
\end{aligned}
$$

Because of $h_1=h_{1,1}+h_{1,2}$, above estimates imply (\ref{h-alpha-bata})$_{1,2,3}$.

Moreover,  (\ref{L2-B-b}), Propositions \ref{exi-ns} and \ref{B-b-Pro}, and Tayolr expansion
% and mean value Theorem
 also imply that
$$\begin{aligned}
 &\frac{1}{\epsilon}\|\alpha_1\|_{L^2(I)}^2\\
\le & C\|u^{I,0}_{xt}\|_{L^\infty(I)}^2 \epsilon \sqrt{\epsilon} \| \langle y \rangle \rho^{B,1} \|_{L_y^2}^2+ C(\|\rho^{I,0}_{x} u^{I,0}\|_{L^\infty(I)}^2+ \|\rho^{I,0} u^{I,0}_x\|_{L^\infty(I)}^2 )\epsilon \sqrt{\epsilon} \| \langle y \rangle u^{B,2}_y \|_{L_y^2}^2\\
&+ C(\|u^{I,0}_{xx} u^{I,0}\|_{L^\infty(I)}^2+ \| u^{I,0}_x\|_{L^\infty(I)}^4 )\epsilon \sqrt{\epsilon}  \| \langle y \rangle \rho^{B,1} \|_{L_y^2}^2+ C\|(\rho^{I,0})^{\gamma-2}\rho^{I,0}_{x}\|_{L^\infty(I)}^2 \epsilon \sqrt{\epsilon}\| \langle y \rangle \rho^{B,2}_y \|_{L_y^2}^2\\
& +C\|\partial_x^2[(\rho^{I,0})^{\gamma-1}]\|_{L^\infty(I)}^2 \epsilon \sqrt{\epsilon} \| \langle y \rangle \rho^{B,1} \|_{L_y^2}^2+ C\|(\rho^{I,0})^{\gamma-3}\rho^{I,0}_{x}\|_{L^\infty(I)}^2 \epsilon \sqrt{\epsilon}\| \rho^{B,1} \|_{L_y^\infty}^2\| \langle y \rangle \rho^{B,1}_y \|_{L_y^2}^2\\
&+C\|\partial_x^2[(\rho^{I,0})^{\gamma-1}]\|_{L^\infty(I)}^2 \epsilon \sqrt{\epsilon}\| \langle y \rangle^2 \rho^{B,1}_y \|_{L_y^2}^2\\
&+C\|\rho^{I,0}_{x}\|_{L^\infty(I)}^2 \epsilon \sqrt{\epsilon} \| \langle y \rangle \rho^{B,2}_{yyy} \|_{L_y^2}^2+C\|\rho^{I,0}_{xx}\|_{L^\infty(I)}^2 \epsilon \sqrt{\epsilon} \| \langle y \rangle^2 \rho^{B,1}_{yyy} \|_{L_y^2}^2\\
\le & C\epsilon \sqrt{\epsilon},
\end{aligned}
$$
Similarly, we have
$$\begin{aligned}
 &\frac{1}{\epsilon}\|\alpha_2 \|_{L^2(I)}^2
\le  C\epsilon \sqrt{\epsilon}.
\end{aligned}
$$

Because of $\alpha=\frac{1}{\sqrt\epsilon}(\alpha_1+\alpha_2)$, we get (\ref{h-alpha-bata})$_{4}$.

Finally, from $U^\epsilon|_{x=0,1}=0$, using H$\rm{\ddot{o}}$lder's inequality, it holds that
$$\begin{aligned}
|U^\epsilon(x,t)|= |\int_0^x U^\epsilon_a(a,t)\mathrm{d}a | \le \|U^\epsilon_x\|_{L^2(I)} |\sqrt \epsilon y|^{\frac{1}{2}},\\
|U^\epsilon(x,t)|= |\int_x^1 U^\epsilon_a(a,t)\mathrm{d}a | \le \|U^\epsilon_x\|_{L^2(I)} |\sqrt \epsilon z|^{\frac{1}{2}}.\\
\end{aligned}
$$
\end{proof}
\begin{Lemma}\label{clm:no-impor}
Under the conditions of Theorem \ref{assumption}, there holds that
\begin{equation}\label{some-h}
\left\{
\begin{aligned}
&\|h_3\|_{L^2(I)}^2 +\|h_4\|_{L^2(I)}^2 +\|h_5\|_{L^2(I)}^2  \le C\epsilon\sqrt{\epsilon},\\
&\|(h_3)_x\|_{L^2(I)}^2 +\|(h_4)_x\|_{L^2(I)}^2 +\|(h_5)_x\|_{L^2(I)}^2  \le C\sqrt{\epsilon}.\\
\end{aligned}\right.
\end{equation}
\end{Lemma}
\begin{proof}
%In our proof, we will use Propositions \ref{exi-ns}, \ref{B-b-Pro},  (\ref{L2-B-b}) in Claim \ref{some-inequ} and Remark \ref{U-B3-L-infty}.

Firstly, using H$\rm{\ddot{o}}$lder's inequality, Propositions \ref{exi-ns} and \ref{B-b-Pro}, and  (\ref{L2-B-b}) in Lemma \ref{some-inequ}, % {\color{red}and Remark \ref{U-B3-L-infty}},
it holds that
$$
\begin{aligned}
\|h_3\|_{L^2(I)}^2=&\|\sqrt{\epsilon}\rho^{B,2}_t  +u^{I,0}_x\sqrt{\epsilon}\rho^{B,2}+\sqrt{\epsilon}u^{B,3}_y\rho^{I,0} +\sqrt{\epsilon}u^{B,2}\rho^{I,0}_x+{\epsilon}u^{B,3}\rho^{I,0}_x\|_{L^2(I)}^2\\
\le & C\|\sqrt{\epsilon}\rho^{B,2}_t \|_{L^2(I)}^2 + C\epsilon \|u^{I,0}_x\|_{L^\infty(I)}^2\|\rho^{B,2}\|_{L^2(I)}^2 + C\epsilon \|u^{B,3}_y\|_{L^2(I)}^2\|\rho^{I,0} \|_{L^\infty(I)}^2 \\
&+ C\epsilon \|u^{B,2}\|_{L^2(I)}^2\|\rho^{I,0}_x \|_{L^\infty(I)}^2+ C\epsilon^2 \|u^{B,3}\|_{L^2(I)}^2\|\rho^{I,0}_x \|_{L^\infty(I)}^2\\
\le & C\epsilon\sqrt\epsilon+C\epsilon\sqrt\epsilon
+C\epsilon\sqrt\epsilon+C\epsilon\sqrt\epsilon+C\epsilon^2\sqrt\epsilon \le C\epsilon\sqrt\epsilon.
\end{aligned}
$$
Similarly, we have $ \|h_4\|_{L^2(I)}^2 \le C \epsilon \sqrt\epsilon$.

Define $h_5:=h_{5,1} +h_{5,2}$
as
$$
\begin{aligned}
h_{5,1}=&[(u^{B,2}_y+u^{b,2}_z)+\sqrt{\epsilon}(u^{B,3}_y+u^{b,3}_z)][\sqrt{\epsilon}(\rho^{B,1}+\rho^{b,1}) +{\epsilon}(\rho^{B,2}+\rho^{b,2})+L^\epsilon ],\\
h_{5,2}=&[\sqrt{\epsilon}(u^{B,2}+u^{b,2})+{\epsilon}(u^{B,3}+u^{b,3})][(\rho^{B,1}_y+\rho^{b,1}_z) +\sqrt{\epsilon}(\rho^{B,2}_y+\rho^{b,2}_z)+L^\epsilon_x ].
\end{aligned}
$$
Then, using H$\rm{\ddot{o}}$lder's inequality, Propositions \ref{exi-ns} and \ref{B-b-Pro}, and (\ref{L2-B-b}) in Lemma \ref{some-inequ}, % and Remark \ref{U-B3-L-infty},
we get
$$
\begin{aligned}
\|h_{5,1}\|_{L^2(I)}^2 \le &\|[(u^{B,2}_y+u^{b,2}_z)+\sqrt{\epsilon}(u^{B,3}_y+u^{b,3}_z)]\|_{L^2(I)}^2 \|[\sqrt{\epsilon}(\rho^{B,1}+\rho^{b,1}) +{\epsilon}(\rho^{B,2}+\rho^{b,2})+L^\epsilon ]\|_{L^\infty(I)}^2 \\
\le &C\epsilon \sqrt\epsilon,
\end{aligned}
$$
and
$$
\begin{aligned}
\|h_{5,2}\|_{L^2(I)}^2 \le &C\|[\sqrt{\epsilon}(u^{B,2}+u^{b,2})+{\epsilon}(u^{B,3}+u^{b,3})]\|_{L^2(I)}^2 \|[(\rho^{B,1}_y+\rho^{b,1}_z) +\sqrt{\epsilon}(\rho^{B,2}_y+\rho^{b,2}_z)+L^\epsilon_x ]\|_{L^\infty(I)}^2 \\
\le &C\epsilon \sqrt\epsilon.
\end{aligned}
$$
In conclusion, we have
$$\|h_3\|_{L^2(I)}^2 +\|h_4\|_{L^2(I)}^2 +\|h_5\|_{L^2(I)}^2  \le C\epsilon\sqrt{\epsilon}.$$

Secondly, using H$\rm{\ddot{o}}$lder's inequality, Propositions \ref{exi-ns} and \ref{B-b-Pro}, and (\ref{L2-B-b}) in Lemma \ref{some-inequ}, %and Remark \ref{U-B3-L-infty},
 it holds that
$$
\begin{aligned}
\|(h_3)_x\|_{L^2(I)}^2 =&\|\sqrt{\epsilon}\rho^{B,2}_{xt}  +u^{I,0}_{xx}\sqrt{\epsilon}\rho^{B,2}+u^{B,3}_{yy}\rho^{I,0} +u^{B,2}_y\rho^{I,0}_x+\sqrt{\epsilon}u^{B,3}_y\rho^{I,0}_x+u^{I,0}_x\rho^{B,2}_y+\sqrt{\epsilon}u^{B,3}_y\rho^{I,0}_x\\
& +\sqrt{\epsilon}u^{B,2}\rho^{I,0}_{xx}+{\epsilon}u^{B,3}\rho^{I,0}_{xx}\|_{L^2(I)}^2\\
\le & C \|\sqrt{\epsilon}\rho^{B,2}_{xt}\|_{L^2(I)}^2+C \|u^{I,0}_{xx}\|_{L^\infty(I)}^2 \|\sqrt{\epsilon}\rho^{B,2}\|_{L^2(I)}^2+C \|u^{B,3}_{yy}\|_{L^2(I)}^2\|\rho^{I,0}\|_{L^\infty(I)}^2\\
&+C\|u^{B,2}_y\|_{L^2(I)}^2\|\rho^{I,0}_x\|_{L^\infty(I)}^2+C\|\sqrt{\epsilon}u^{B,3}_y\|_{L^2(I)}^2\|\rho^{I,0}_x\|_{L^\infty(I)}^2+C\|u^{I,0}_x\|_{L^\infty(I)}^2\|\rho^{B,2}_y\|_{L^2(I)}^2\\
&+C\|\sqrt{\epsilon}u^{B,3}_y \|_{L^2(I)}^2\|\rho^{I,0}_x\|_{L^\infty(I)}^2\!+\! C\|\sqrt{\epsilon}u^{B,2}\|_{L^\infty(I)}^2\|\rho^{I,0}_{xx}\|_{L^2(I)}^2\!+\!C\|{\epsilon}u^{B,3}\|_{L^\infty(I)}^2\|\rho^{I,0}_{xx}\|_{L^2(I)}^2\\
\le & C\epsilon\sqrt\epsilon+C\epsilon\sqrt\epsilon+ C\sqrt\epsilon+ C\sqrt\epsilon+C\epsilon\sqrt\epsilon+ C\sqrt\epsilon+C\epsilon\sqrt\epsilon+C\epsilon+C\epsilon^2 \le C\sqrt\epsilon.
\end{aligned}
$$
Similarly, we have
$$
\|(h_4)_x\|_{L^2(I)}^2 \le C\sqrt\epsilon.
$$
We define $(h_5)_x:= \sum_{i=1}^3 h_{5,i}^x$ as
$$
\begin{aligned}
h_{5,1}^x=&[(u^{B,2}_{yy}+u^{b,2}_{zz})+\sqrt{\epsilon}(u^{B,3}_{yy}+u^{b,3}_{zz})][(\rho^{B,1}+\rho^{b,1}) +\sqrt{\epsilon}(\rho^{B,2}+\rho^{b,2})+\frac{1}{\sqrt{\epsilon}}L^\epsilon ],\\
h_{5,2}^x=&2[(u^{B,2}_{y}+u^{b,2}_{z})+\sqrt{\epsilon}(u^{B,3}_{y}+u^{b,3}_{z})][(\rho^{B,1}_y+\rho^{b,1}_z) +\sqrt{\epsilon}(\rho^{B,2}_y+\rho^{b,2}_z)+L^\epsilon_x ],\\
h_{5,3}^x=&[(u^{B,2}+u^{b,2})+\sqrt{\epsilon}(u^{B,3}+u^{b,3})][(\rho^{B,1}_{yy}+\rho^{b,1}_{zz}) +\sqrt{\epsilon}(\rho^{B,2}_{yy}+\rho^{b,2}_{zz})+\sqrt{\epsilon}L^\epsilon_{xx} ],\\
\end{aligned}
$$
then using H$\rm{\ddot{o}}$lder's inequality, Propositions \ref{exi-ns} and \ref{B-b-Pro}, and (\ref{L2-B-b}) in Lemma \ref{some-inequ},
 we get
$$\begin{aligned}
&\|(h_5)_x\|_{L^2(I)}^2 \le  C\sum_{i=1}^3\|h_{5,i}^x\|_{L^2(I)}^2\\
\le & C\|(u^{B,2}_{yy}+u^{b,2}_{zz})+\sqrt{\epsilon}(u^{B,3}_{yy}+u^{b,3}_{zz})\|_{L^2(I)}^2\|(\rho^{B,1}+\rho^{b,1}) +\sqrt{\epsilon}(\rho^{B,2}+\rho^{b,2})+\frac{1}{\sqrt{\epsilon}}L^\epsilon\|_{L^\infty(I)}^2\\
&+C\|(u^{B,2}_{y}+u^{b,2}_{z})+\sqrt{\epsilon}(u^{B,3}_{y}+u^{b,3}_{z})\|_{L^2(I)}^2\|(\rho^{B,1}_y+\rho^{b,1}_z) +\sqrt{\epsilon}(\rho^{B,2}_y+\rho^{b,2}_z)+L^\epsilon_x \|_{L^\infty(I)}^2\\
&+C\|(u^{B,2}+u^{b,2})+\sqrt{\epsilon}(u^{B,3}+u^{b,3})\|_{L^\infty(I)}^2\|(\rho^{B,1}_{yy}+\rho^{b,1}_{zz}) +\sqrt{\epsilon}(\rho^{B,2}_{yy}+\rho^{b,2}_{zz})+\sqrt{\epsilon}L^\epsilon_{xx}\|_{L^2(I)}^2 \\
\le  &C\sqrt\epsilon+C\sqrt\epsilon+C\sqrt\epsilon\le C\sqrt\epsilon.
\end{aligned}
$$

In conclusion, there holds
$$
\|(h_3)_x\|_{L^2(I)}^2 +\|(h_4)_x\|_{L^2(I)}^2 +\|(h_5)_x\|_{L^2(I)}^2  \le C\sqrt{\epsilon}.
$$
\end{proof}

\begin{Lemma}\label{clm:important}
Under the conditions of Theorem \ref{assumption}, there holds that
\begin{equation}\label{W-L2}
\begin{aligned}
\|W\|_{L^2(I)}^2 \le &C\|U^\epsilon\|_{L^2(I)}^2+C\|\Phi^\epsilon\|_{L^2(I)}^2 +C(\sqrt\epsilon +\|\Phi^\epsilon\|_{L^2(I)}^2)(1
%+\epsilon\|U^\epsilon_x\|_{L^2(I)}^2
 +\epsilon \|\Phi^\epsilon_{xx}\|_{L^2(I)}^2 )\\
& +C\sqrt \epsilon +C \|\Phi^\epsilon_x\|_{L^2(I)}^2,
\end{aligned}
\end{equation} and that
\begin{equation}\label{widetilde-W-L2}
\|W_1\|_{L^2(I)}^2 \le C\epsilon+C\|U^\epsilon\|_{L^2(I)}^2+C\|\Phi^\epsilon\|_{L^2(I)}^2,
\end{equation} where $W$ is given by (\ref{zhankai-equ})$_2$ and
\begin{equation}\label{widetilde-W}
\begin{aligned}
W_1:=W+ P(\rho^\epsilon,\rho^{I,0})-\Omega_4.
\end{aligned}
\end{equation}

\end{Lemma}
\begin{proof}
From Propositions \ref{exi-ns} and \ref{B-b-Pro},  Lemmas \ref{L-M-well-pose} and \ref{clm-rho-bdd}, %来自Pro Lemma \ref{L-M-well-pose}} 变成Lemma \ref{L-M-well-pose}}
% Remark \ref{U-B3-L-infty},
 (\ref{L2-B-b}), and H$\rm{\ddot{o}}$lder's inequality, we get estimates of $\|\Theta_1\|_{L^2(I)}^2$ and $\|\Theta_2\|_{L^2(I)}^2$ as follows:
$$
\begin{aligned}
 \|\Theta_1\|_{L^2(I)}^2=& \|(\frac{1}{\sqrt{\epsilon}}L^\epsilon +\Phi^\epsilon) u^{I,0}_t + \rho^\epsilon\frac{1}{\sqrt{\epsilon}}M^\epsilon_t+ \rho^{I,0} u^{I,0} \frac{1}{\sqrt{\epsilon}}M^\epsilon_x +\rho^{I,0} (\frac{1}{\sqrt{\epsilon}}M^\epsilon+U^\epsilon)u_x^{I,0}\\
&+ (\frac{1}{\sqrt{\epsilon}}L^\epsilon +\Phi^\epsilon)u^{I,0} u_x^{I,0}\|_{L^2(I)}^2\\
\le & C \|\frac{1}{\sqrt{\epsilon}}L^\epsilon +\Phi^\epsilon\|_{L^2(I)}^2 \|u^{I,0}_t\|_{L^\infty(I)}^2 + C \|\rho^\epsilon\|_{L^\infty(I)}^2 \|\frac{1}{\sqrt{\epsilon}}M^\epsilon_t\|_{L^2(I)}^2 \\
&+ C \| \rho^{I,0} u^{I,0}\|_{L^\infty(I)}^2 \|\frac{1}{\sqrt{\epsilon}}M^\epsilon_x \|_{L^2(I)}^2 + C \|\rho^{I,0}\|_{L^\infty(I)}^2  \| \frac{1}{\sqrt{\epsilon}}M^\epsilon+U^\epsilon\|_{L^2(I)}^2 \|u_x^{I,0}\|_{L^\infty(I)}^2\\
&+  C   \|\frac{1}{\sqrt{\epsilon}}L^\epsilon +\Phi^\epsilon\|_{L^2(I)}^2 \|u^{I,0} u_x^{I,0}\|_{L^\infty(I)}^2\\
\le & C(\epsilon^2+\|\Phi^\epsilon\|_{L^2(I)}^2) + C\epsilon+C\epsilon+ C(\epsilon+\|U^\epsilon\|_{L^2(I)}^2) +C(\epsilon^2+\|\Phi^\epsilon\|_{L^2(I)}^2)\\
\le & C\|\Phi^\epsilon\|_{L^2(I)}^2+C\epsilon+ C\|U^\epsilon\|_{L^2(I)}^2, \\
\end{aligned}
$$
$$
\begin{aligned}
 \|\Theta_2\|_{L^2(I)}^2=& \|\sqrt{\epsilon}\rho^{B,2} u^{I,0}_t + \rho^\epsilon(\sqrt{\epsilon}u^{B,2}_t+\epsilon u^{B,3}_t)+ \rho^{I,0} u^{I,0} \sqrt\epsilon u^{B,3}_y +\rho^{I,0} (\sqrt{\epsilon} u^{B,2}+\epsilon u^{B,3})u_x^{I,0}\\
&+ {\sqrt{\epsilon}}\rho^{B,2}u^{I,0} u_x^{I,0}\|_{L^2(I)}^2\\
\le & C \|\sqrt{\epsilon}\rho^{B,2}\|_{L^2(I)}^2 \|u^{I,0}_t\|_{L^\infty(I)}^2 + C \|\rho^\epsilon\|_{L^\infty(I)}^2 \|\sqrt{\epsilon}u^{B,2}_t+\epsilon u^{B,3}_t\|_{L^2(I)}^2 \\
&+ C \| \rho^{I,0} u^{I,0}\|_{L^\infty(I)}^2 \|\sqrt\epsilon u^{B,3}_y \|_{L^2(I)}^2 + C \|\rho^{I,0}\|_{L^\infty(I)}^2  \| \sqrt{\epsilon} u^{B,2}+\epsilon u^{B,3}\|_{L^2(I)}^2 \|u_x^{I,0}\|_{L^\infty(I)}^2\\
&+  C   \| {\sqrt{\epsilon}}\rho^{B,2}\|_{L^2(I)}^2 \|u^{I,0} u_x^{I,0}\|_{L^\infty(I)}^2\\
\le & C\epsilon\sqrt\epsilon +C\epsilon\sqrt\epsilon +C\epsilon\sqrt\epsilon+C\epsilon\sqrt\epsilon +C\epsilon\sqrt\epsilon
\le  C\epsilon\sqrt\epsilon .\\
\end{aligned}
$$

Similar to the estimate of $ \|\Theta_2\|_{L^2(I)}^2$, we have
$$
\|\Theta_3\|_{L^2(I)}^2 \le C\epsilon \sqrt\epsilon.
$$

Moreover, using Propositions \ref{exi-ns} and \ref{B-b-Pro}, Lemma \ref{L-M-well-pose}, %Remark \ref{U-B3-L-infty},
(\ref{L2-B-b}), and H$\rm{\ddot{o}}$lder's inequality, we have
$$
\begin{aligned}
\|\Theta_4\|_{L^2(I)}^2=&\|\rho^{I,0}[\sqrt{\epsilon}(u^{B,2}+u^{b,2})+\epsilon(u^{B,3}+u^{b,3})+\frac{1}{\sqrt{\epsilon}}M^\epsilon + U^\epsilon] [M^\epsilon + \epsilon(u^{B,2}+u^{b,2})\\
&+\epsilon\sqrt{\epsilon}(u^{B,3}+u^{b,3})]_x
+[(\rho^{B,1}+\rho^{b,1})+\sqrt{\epsilon}(\rho^{B,2}+\rho^{b,2})+\frac{1}{\sqrt{\epsilon}}L^\epsilon + \Phi^\epsilon]u^{I,0} [M^\epsilon \\
&+ \epsilon(u^{B,2}+u^{b,2})+\epsilon\sqrt{\epsilon}(u^{B,3}+u^{b,3})]_x\|_{L^2(I)}^2\\
\le & C \|\rho^{I,0}\|_{L^\infty(I)}^2 \|\sqrt{\epsilon}(u^{B,2}+u^{b,2})+\epsilon(u^{B,3}+u^{b,3})+\frac{1}{\sqrt{\epsilon}}M^\epsilon + U^\epsilon\|_{L^2(I)}^2 \|[M^\epsilon\\
& + \epsilon(u^{B,2}+u^{b,2}) +\epsilon\sqrt{\epsilon}(u^{B,3}+u^{b,3})]_x\|_{L^\infty(I)}^2\\
&+C \|(\rho^{B,1}+\rho^{b,1})+\sqrt{\epsilon}(\rho^{B,2}+\rho^{b,2})+\frac{1}{\sqrt{\epsilon}}L^\epsilon + \Phi^\epsilon\|_{L^2(I)}^2 \|u^{I,0}\|_{L^\infty(I)}^2 \|[M^\epsilon\\
&+ \epsilon(u^{B,2}+u^{b,2})+\epsilon\sqrt{\epsilon}(u^{B,3}+u^{b,3})]_x\|_{L^\infty(I)}^2\\
\le &C(\epsilon +\|U^\epsilon\|_{L^2(I)}^2)\epsilon +C(\sqrt\epsilon +\|\Phi^\epsilon\|_{L^2(I)}^2)\epsilon \le C\epsilon (\|\Phi^\epsilon\|_{L^2(I)}^2+\|U^\epsilon\|_{L^2(I)}^2+\sqrt\epsilon).
\end{aligned}
$$

Using (\ref{ini-val-phi-1}), Propositions \ref{exi-ns} and \ref{B-b-Pro}, Lemma \ref{L-M-well-pose}, %Remark \ref{U-B3-L-infty},
 (\ref{L2-B-b}), and H$\rm{\ddot{o}}$lder's inequality, we get
$$
\begin{aligned}
\|\Theta_5\|_{L^2(I)}^2=&\|
 [(\rho^{B,1}+\rho^{b,1})+\sqrt{\epsilon}(\rho^{B,2}+\rho^{b,2})+\frac{1}{\sqrt{\epsilon}}L^\epsilon + \Phi^\epsilon][\epsilon(u^{B,2}+u^{b,2})+\epsilon\sqrt{\epsilon}(u^{B,3}+u^{b,3})\\
&+M^\epsilon +\sqrt{\epsilon} U^\epsilon] [u^{I,0}+M^\epsilon + \epsilon(u^{B,2}+u^{b,2})+\epsilon\sqrt{\epsilon}(u^{B,3}+u^{b,3})]_x\|_{L^2(I)}^2\\
\le & C \|(\rho^{B,1}+\rho^{b,1})+\sqrt{\epsilon}(\rho^{B,2}+\rho^{b,2})+\frac{1}{\sqrt{\epsilon}}L^\epsilon + \Phi^\epsilon\|_{L^\infty(I)}^2 \|\epsilon(u^{B,2}+u^{b,2})\\
&+\epsilon\sqrt{\epsilon}(u^{B,3}+u^{b,3})+M^\epsilon +\sqrt{\epsilon} U^\epsilon\|_{L^2(I)}^2 \|[u^{I,0}+M^\epsilon + \epsilon(u^{B,2}+u^{b,2})\\
&+\epsilon\sqrt{\epsilon}(u^{B,3}+u^{b,3})]_x\|_{L^\infty(I)}^2\\
\le &C(\epsilon^2+\epsilon\|U^\epsilon\|_{L^2(I)}^2 ).
\end{aligned}
$$

Using (\ref{zhankai})$_1$, mean value Theorem, and Lemma \ref{clm-rho-bdd}, it holds that
$$\begin{aligned}
P(\rho^{\epsilon},\rho^{I,0})
=& \frac{1}{\sqrt{\epsilon}}\{[\gamma (\rho^\epsilon)^{\gamma-1}-\gamma(\rho^{I,0})^{\gamma-1}]\rho^\epsilon_x + \gamma(\rho^{I,0})^{\gamma-1}[\sqrt{\epsilon}(\rho^{B,2}_y+\rho^{b,2}_z)+L^\epsilon_x+\sqrt{\epsilon}\Phi^{\epsilon}_x]\\
&-\gamma (\rho^{I,0})^{\gamma-1}\sqrt{\epsilon}(\rho^{B,2}_y+\rho^{b,2}_z)-\gamma (\gamma-1) (\rho^{I,0})^{\gamma-2} \sqrt{\epsilon}(\rho^{B,1}+\rho^{b,1}) \rho^{I,0}_x\\
&-\gamma (\gamma-1) (\rho^{I,0})^{\gamma-2}\sqrt{\epsilon}\rho^{B,1}\rho^{B,1}_y -\gamma (\gamma-1) (\rho^{I,0})^{\gamma-2}\sqrt{\epsilon}\rho^{b,1}\rho^{b,1}_z\}\\
=& \gamma(\gamma-1) [\theta\rho^\epsilon+(1-\theta)\rho^{I,0}]^{\gamma-2}[(\rho^{B,1}+\rho^{b,1})+\sqrt{\epsilon}(\rho^{B,2}+\rho^{b,2})+\frac{1}{\sqrt{\epsilon}}L^\epsilon + \Phi^\epsilon]\rho^\epsilon_x\\
& + \gamma(\rho^{I,0})^{\gamma-1}[(\rho^{B,2}_y+\rho^{b,2}_z)+\frac{1}{\sqrt{\epsilon}}L^\epsilon_x+\Phi^{\epsilon}_x]\\
&-\gamma (\rho^{I,0})^{\gamma-1}(\rho^{B,2}_y+\rho^{b,2}_z)-\gamma (\gamma-1) (\rho^{I,0})^{\gamma-2} (\rho^{B,1}+\rho^{b,1}) \rho^{I,0}_x\\
&-\gamma (\gamma-1) (\rho^{I,0})^{\gamma-2}\rho^{B,1}\rho^{B,1}_y -\gamma (\gamma-1) (\rho^{I,0})^{\gamma-2}\rho^{b,1}\rho^{b,1}_z.
\\
\end{aligned}
$$
Thus, using Lemmas \ref{L-M-well-pose} and \ref{clm:h}, Propositions \ref{exi-ns} and \ref{B-b-Pro}, (\ref{L2-B-b}), and Poincar\'e's inequality $\|\Phi^\epsilon_x\|_{L^\infty(I)} \le C \|\Phi^\epsilon_{xx}\|_{L^2(I)} $, we get
$$\begin{aligned}
&\|P(\rho^{\epsilon},\rho^{I,0})\|_{L^2(I)}^2 + \|\alpha\|_{L^2(I)}^2\\
\le & C (\|\rho^{\epsilon}_x-\sqrt \epsilon \Phi^\epsilon_x\|_{L^\infty(I)}^2 + \epsilon \|\Phi^\epsilon_{xx}\|_{L^2(I)}^2)(\sqrt\epsilon +\|\Phi^\epsilon\|_{L^2(I)}^2) +C( \sqrt{\epsilon} +\|\Phi^\epsilon_x\|_{L^2(I)}^2)+C \epsilon\sqrt{\epsilon} \\
\le & C\sqrt \epsilon +C(1+ \epsilon \|\Phi^\epsilon_{xx}\|_{L^2(I)}^2) (\sqrt\epsilon +\|\Phi^\epsilon\|_{L^2(I)}^2) +C \|\Phi^\epsilon_x\|_{L^2(I)}^2.
\end{aligned}
$$

Finally, using Propositions \ref{exi-ns} and \ref{B-b-Pro}, Lemma \ref{L-M-well-pose}, (\ref{L2-B-b}), and H$\rm{\ddot{o}}$lder's inequality, we have the following estimates:
$$
\begin{aligned}
\|\Omega_1\|_{L^2(I)}^2 \le &C\|\sqrt{\epsilon} \rho^{I,0} \|_{L^\infty(I)}^2 \|\rho^{I,0}_{xxx}\|_{L^2(I)}^2 \le C\epsilon,\\
\|\Omega_2\|_{L^2(I)}^2 \le &C\|\epsilon (\rho^{B,1}+\rho^{b,1})\|_{L^\infty(I)}^2 \|[\rho^{I,0}+{\epsilon}(\rho^{B,2}+\rho^{b,2})]_{xxx}\|_{L^2(I)}^2 \le C\epsilon^2(1+\frac{1}{\sqrt\epsilon}) \le C\epsilon \sqrt\epsilon,\\
\|\Omega_3\|_{L^2(I)}^2 \le & C\|\sqrt{\epsilon }[{\epsilon}(\rho^{B,2}+\rho^{b,2})+L^\epsilon +\sqrt{\epsilon} \Phi^\epsilon]\|_{L^2(I)}^2 \|[\rho^{I,0}+\sqrt{\epsilon}(\rho^{B,1}+\rho^{b,1})+{\epsilon}(\rho^{B,2}+\rho^{b,2})]_{xxx}\|_{L^\infty(I)}^2\\
\le & C \epsilon(\epsilon^2\sqrt\epsilon + \epsilon \|\Phi^\epsilon\|_{L^2(I)}^2)(1+\frac{1}{\epsilon^2}) \le C(\epsilon\sqrt\epsilon +\|\Phi^\epsilon\|_{L^2(I)}^2), \\
\end{aligned}
$$
and
$$\begin{aligned}
\|\Omega_4\|_{L^2(I)}^2 \le C\|\rho^{b,1}\|_{L^\infty(I)}^2 \|\rho^{B,1}_{yyy}\|_{L^2(I)}^2+C \| \rho^{B,1}\|_{L^\infty(I)}^2 \|\rho^{b,1}_{zzz}\|_{L^2(I)}^2\le C\sqrt\epsilon.\\
\end{aligned}
$$

Then, above estimates imply (\ref{W-L2}) and (\ref{widetilde-W-L2}).
\end{proof}

\subsection{Proof of Theorem \ref{assumption}}
~~~~Proof of Theorem \ref{assumption} will be obtained by Corollaries \ref{Phi-U-L2-all}, \ref{Phi-U-H^1-all} and \ref{Phi-Linfty-shoulian}. %We will prove the following Lemmas  under Propositions \ref{exi-ns}, \ref{exi-nsk} and \ref{B-b-Pro}.
\begin{Lemma}\label{Phi-U-L2}
Under the conditions of Theorem \ref{assumption}, there holds that
\begin{equation}\label{small-phi-L2}
\begin{aligned}
\frac{1}{2} \frac{\mathrm{d}}{\mathrm{d}t} \int_I (\Phi^\epsilon)^2\mathrm{d}x\le &\frac{1}{4} \int_I (U^\epsilon_x)^2 \mathrm{d}x + C \| \Phi^\epsilon\|_{L^2(I)}^2+C\epsilon.
\end{aligned}
\end{equation}
\end{Lemma}

\begin{proof}

Firstly, multiplying (\ref{zhankai-equ})$_1$ by $\Phi^\epsilon$ and integrating the result over $I$, we have
\begin{equation}\label{3_1-L2-1}
\begin{aligned}
&\frac{1}{2} \frac{\mathrm{d}}{\mathrm{d}t} \int_I (\Phi^\epsilon)^2\mathrm{d}x=-\int_I (\rho^\epsilon U^\epsilon)_x\Phi^\epsilon\mathrm{d}x-\int_I \Phi^\epsilon\sum_{i=1}^5 h_i\mathrm{d}x.\\
\end{aligned}
\end{equation}
 Using integration by parts, Lemma \ref{clm-rho-bdd},
(\ref{rho-epsilon-phi-epsilon}), (\ref{ini-val-phi-1}), and Poincar\'e's inequality $\|U^\epsilon\|_{L^\infty(I)} \le C \|U^\epsilon_x\|_{L^2(I)}$, we have
\begin{equation}\label{3_1-L2-2}
\begin{aligned}
& -\int_I (\rho^\epsilon_x U^\epsilon +\rho^\epsilon U^\epsilon_x)\Phi^\epsilon\mathrm{d}x=  -\int_I [(\rho^\epsilon_x-\sqrt\epsilon \Phi^\epsilon_{x}) U^\epsilon  +\sqrt\epsilon \Phi^\epsilon_{x} U^\epsilon+\rho^\epsilon U^\epsilon_x]\Phi^\epsilon\mathrm{d}x\\
= & \int_I [-(\rho^\epsilon_x-\sqrt\epsilon \Phi^\epsilon_{x}) U^\epsilon  +\frac{1}{2}\sqrt\epsilon \Phi^\epsilon U^\epsilon_{x} -\rho^\epsilon U^\epsilon_x]\Phi^\epsilon\mathrm{d}x\\
\le  &\|\rho^\epsilon_x-\sqrt\epsilon \Phi^\epsilon_{x}\|_{L^2(I)}\|U^\epsilon\|_{L^\infty(I)}\|\Phi^\epsilon\|_{L^2(I)}+C \sqrt\epsilon \|\Phi^\epsilon\|_{L^\infty(I)}\|U_x^\epsilon\|_{L^2(I)}\|\Phi^\epsilon\|_{L^2(I)}\\
&+C\|\rho^\epsilon\|_{L^\infty(I)}\|U_x^\epsilon\|_{L^2(I)}\|\Phi^\epsilon\|_{L^2(I)}\\
\le & \frac{1}{4} \int_I (U^\epsilon_x)^2 \mathrm{d}x + C \| \Phi^\epsilon\|_{L^2(I)}^2.
\end{aligned}
\end{equation}

%From $(u^\epsilon,u^{I,0},U^\epsilon)|_{x=0,1}=0$ and (\ref{zhankai})$_2$,
Recalling that
$$
[u^{I,0}+\epsilon(u^{B,2}+u^{b,2})+\epsilon\sqrt{\epsilon}(u^{B,3}+u^{b,3})+M^\epsilon]|_{x=0,1}=0,
$$
we have
$$
\begin{aligned}
&-\int_I \Phi^\epsilon \{u^{I,0} +M^\epsilon   +[\sqrt{\epsilon}(u^{B,2}+u^{b,2})+{\epsilon}(u^{B,3}+u^{b,3})]\sqrt{\epsilon} \}\Phi^\epsilon_x\mathrm{d}x\\
=& \int_I \frac{1}{2}(\Phi^\epsilon)^2 \{u^{I,0}_x +M^\epsilon_x   +\sqrt{\epsilon}[(u^{B,2}_y+u^{b,2}_z)+{\epsilon}(u^{B,3}_y+u^{b,3}_z)] \}\mathrm{d}x,
\end{aligned}
$$
which yields
$$
\begin{aligned}
&-\int_I \Phi^\epsilon  h_2\mathrm{d}x \\
=&-\int_I \Phi^\epsilon[\frac{1}{\sqrt{\epsilon}}L^\epsilon_t+u_{x}^{I,0}\frac{1}{\sqrt{\epsilon}}L^\epsilon+u^{I,0}\frac{1}{\sqrt{\epsilon}}L^\epsilon_x+\frac{1}{\sqrt{\epsilon}}M^\epsilon_x \rho^\epsilon+\frac{1}{\sqrt{\epsilon}}M^\epsilon (\rho^\epsilon_x-\sqrt\epsilon\Phi^\epsilon_x)] \mathrm{d}x \\
&+\frac{1}{2}\int_I (\Phi^\epsilon)^2\{M^\epsilon_x-[u^{I,0}_x +\sqrt{\epsilon}(u^{B,2}_y+u^{b,2}_z)+{\epsilon}(u^{B,3}_y+u^{b,3}_z)] \}\mathrm{d}x\\
:=& -\int_I \Phi^\epsilon h_{2,1} \mathrm{d}x + \frac{1}{2}\int_I (\Phi^\epsilon)^2 h_{2,2} \mathrm{d}x.
\end{aligned}
$$
Using Propositions   \ref{exi-ns}, \ref{exi-nsk} and \ref{B-b-Pro}, (\ref{rho-epsilon-phi-epsilon}), H$\rm{\ddot{o}}$lder's inequality, Young's inequality, and Lemmas \ref{some-inequ}  and \ref{L-M-well-pose}, we get
\begin{equation}\label{3_1-L2-3}
\begin{aligned}
&-\int_I \Phi^\epsilon  h_2\mathrm{d}x = -\int_I \Phi^\epsilon h_{2,1} \mathrm{d}x + \frac{1}{2}\int_I (\Phi^\epsilon)^2 h_{2,2} \mathrm{d}x\\
\le & C\|h_{2,1}\|_{L^2(I)}^2+C(1+\|h_{2,2}\|_{L^\infty(I)})\|\Phi^\epsilon\|_{L^2(I)}^2 \le C\epsilon +C\|\Phi^\epsilon\|_{L^2(I)}^2.
\end{aligned}
\end{equation}

Using Lemmas \ref{clm:h} and \ref{clm:no-impor}, H$\rm{\ddot{o}}$lder's inequality, and Young's inequality, we have
\begin{equation}\label{3_1-L2-4}
\begin{aligned}
\int_I \Phi^\epsilon (h_1+h_3+h_4+h_5) \mathrm{d}x
\le  C\|\Phi^{\epsilon}\|_{L^2(I)}^2 + C (\sum_{i=1,3,4,5} \|h_i\|_{L^2(I)}^2)
\le  C\|\Phi^{\epsilon}\|_{L^2(I)}^2+ C\epsilon\sqrt\epsilon.
\end{aligned}
\end{equation}

In conclusion, putting (\ref{3_1-L2-2})-(\ref{3_1-L2-4}) into (\ref{3_1-L2-1}), we get (\ref{small-phi-L2}).
\end{proof}

\begin{Lemma}\label{Phi-U-L2-U}
Under the conditions of Theorem \ref{assumption}, there holds that
 \begin{equation}\label{L2-Esti-2}
\begin{aligned}
&\frac{1}{2}\frac{\mathrm{d}}{\mathrm{d}t} \int_I \rho^{\epsilon} (U^\epsilon)^2\mathrm{d}x +\int_I (U^\epsilon_x)^2 \mathrm{d}x-\int_I \epsilon \rho^\epsilon \Phi^\epsilon_{xxx} U^\epsilon\mathrm{d}x\\
\le & C\epsilon+C\|\sqrt{\rho^{\epsilon}}U^\epsilon\|_{L^2(I)}^2+C\|\Phi^\epsilon\|_{L^2(I)}^2+\frac{1}{4}\|U^{\epsilon}_x\|_{L^2(I)}^2.
\end{aligned}
\end{equation}
\end{Lemma}
\begin{proof}
 Multiplying (\ref{zhankai-equ})$_2$ by $U^\epsilon$ and integrating the result over $I$, using (\ref{widetilde-W}), we get
\begin{equation}\label{forme}
\begin{aligned}
&\frac{1}{2}\frac{\mathrm{d}}{\mathrm{d}t} \int_I \rho^{\epsilon} (U^\epsilon)^2\mathrm{d}x +\int_I (U^\epsilon_x)^2 \mathrm{d}x-\int_I \epsilon \rho^\epsilon \Phi^\epsilon_{xxx} U^\epsilon \mathrm{d}x\\
=& \int_I W_1U^\epsilon\mathrm{d}x-\int_IP(\rho^\epsilon,\rho^{I,0})U^\epsilon \mathrm{d}x+\int_I\Omega_4 U^\epsilon \mathrm{d}x :=  \sum_{k=1}^{3} J_{k} .
\end{aligned}
\end{equation}

Fristly, using Lemma \ref{clm-rho-bdd}, the inequality $$\|U^\epsilon\|_{L^2(I)} \le C\|(\sqrt{\rho^{\epsilon}})^{-1}\|_{L^2(I)}\|\sqrt{\rho^{\epsilon}}U^\epsilon\|_{L^2(I)} \le C \|\sqrt{\rho^{\epsilon}}U^\epsilon\|_{L^2(I)}$$
holds.

Then, from Lemma \ref{clm:important}, we get
$$\begin{aligned}
 J_1=\int_I W_1U^\epsilon\mathrm{d}x\le C\| W_1\|_{L^2(I)}^2 +C \|\sqrt{\rho^{\epsilon}}U^\epsilon\|_{L^2(I)}^2\le  C\epsilon+C\|\sqrt{\rho^{\epsilon}}U^\epsilon\|_{L^2(I)}^2+C\|\Phi^\epsilon\|_{L^2(I)}^2.\\
\end{aligned}
$$

Here, we use mean value Theorem to deal with $J_{2}$. Similar to the proof of Lemma \ref{clm:h}, we define $\xi(x,t) ,\theta(x,t)\in (0,1)$ when we use mean value Theorem. Now, using integration by parts on $J_{2}$, we get
$$\begin{aligned}
J_{2}=&-\int_I P(\rho^\epsilon, \rho^{I,0}) U^\epsilon\mathrm{d}x\\
=&-\int_I\frac{1}{\sqrt{\epsilon}}[\gamma(\rho^\epsilon)^{\gamma-1}\rho^\epsilon_x-\gamma(\rho^{I,0})^{\gamma-1}\rho^{I,0}_x-\gamma (\rho^{I,0})^{\gamma-1}( \rho^{B,1}_y+\rho^{b,1}_z)\\
&-\gamma (\rho^{I,0})^{\gamma-1}\sqrt{\epsilon}(\rho^{B,2}_y+\rho^{b,2}_z)-\gamma (\gamma-1) (\rho^{I,0})^{\gamma-2} \sqrt{\epsilon}(\rho^{B,1}+\rho^{b,1}) \rho^{I,0}_x\\
&-\gamma (\gamma-1) (\rho^{I,0})^{\gamma-2} \sqrt{\epsilon}\rho^{B,1} \rho^{B,1}_y -\gamma (\gamma-1) (\rho^{I,0})^{\gamma-2} \sqrt{\epsilon}\rho^{b,1}\rho^{b,1}_z] U^\epsilon\mathrm{d}x\\
= & J_{2,1} + J_{2,2}+J_{2,3},
\end{aligned}
$$
where
$$\begin{aligned}
J_{2,1}:=&\int_I\frac{1}{\sqrt{\epsilon}}[(\rho^\epsilon)^{\gamma}-(\rho^{I,0})^{\gamma}-\gamma (\rho^{I,0})^{\gamma-1}\sqrt{\epsilon}( \rho^{B,1}+\rho^{b,1})-\gamma (\rho^{I,0})^{\gamma-1} {\epsilon}(\rho^{B,2}+\rho^{b,2})]U^\epsilon_x\mathrm{d}x,\\
J_{2,2}:=&-\int_I\frac{1}{\sqrt{\epsilon}}\gamma (\gamma-1)(\rho^{I,0})^{\gamma-2}  \frac{\epsilon}{2}[(\rho^{B,1})^2+(\rho^{b,1})^2]  U^\epsilon_x\mathrm{d}x\\
&-\int_I\frac{1}{\sqrt{\epsilon}}\gamma (\gamma-1)(\gamma-2) (\rho^{I,0})^{\gamma-3} \rho^{I,0}_x \frac{\epsilon}{2}[(\rho^{B,1})^2+(\rho^{b,1})^2]  U^\epsilon\mathrm{d}x,\\
J_{2,3}:=&-\int_I\frac{1}{\sqrt{\epsilon}}\gamma(\gamma-1) (\rho^{I,0})^{\gamma-2} \rho^{I,0}_x{\epsilon}(\rho^{B,2}+\rho^{b,2})U^\epsilon\mathrm{d}x.\\
\end{aligned}
$$

For $J_{2,1}$, we use mean value Theorem twice, the boundness of $\rho^\epsilon$(see Lemma \ref{clm-rho-bdd}) and $\rho^{I,0}$(see Proposition \ref{exi-ns}), Propositions   \ref{exi-ns} and \ref{B-b-Pro}, Lemma \ref{L-M-well-pose}, (\ref{L2-B-b}), and
$$\begin{aligned}
\|\rho^{\epsilon}-\rho^{I,0}\|_{L^2(I)}^2 =&\|\sqrt{\epsilon}(\rho^{B,1}+\rho^{b,1})+{\epsilon}(\rho^{B,2}+\rho^{b,2})+L^\epsilon +\sqrt{\epsilon} \Phi^\epsilon\|_{L^2(I)}^2\\
\le & C\epsilon \sqrt\epsilon + C\epsilon\|\Phi^\epsilon\|_{L^2(I)}^2
\end{aligned}
$$
 to get the estimate as follows:
$$\begin{aligned}
J_{2,1}
=&\int_I\frac{1}{\sqrt{\epsilon}}\gamma[\theta\rho^\epsilon+(1-\theta)\rho^{I,0}]^{\gamma-1}[\sqrt{\epsilon}(\rho^{B,1}+\rho^{b,1})+{\epsilon}(\rho^{B,2}+\rho^{b,2})+L^\epsilon +\sqrt{\epsilon} \Phi^\epsilon]U^\epsilon_x\mathrm{d}x\\
&-\int_I\frac{1}{\sqrt{\epsilon}}\gamma (\rho^{I,0})^{\gamma-1}\sqrt{\epsilon}( \rho^{B,1}+\rho^{b,1})U^\epsilon_x\mathrm{d}x-\int_I\frac{1}{\sqrt{\epsilon}}[\gamma(\rho^{I,0})^{\gamma-1} {\epsilon}(\rho^{B,2}+\rho^{b,2})]U^\epsilon_x\mathrm{d}x\\
=&\int_I\frac{1}{\sqrt{\epsilon}}\gamma(\gamma-1)[\xi\theta\rho^\epsilon+(1-\xi\theta)\rho^{I,0}]^{\gamma-2}(\theta\rho^\epsilon-\theta\rho^{I,0})[\sqrt{\epsilon}(\rho^{B,1}+\rho^{b,1})+{\epsilon}(\rho^{B,2}+\rho^{b,2})]U^\epsilon_x\mathrm{d}x\\
&+\int_I\frac{1}{\sqrt{\epsilon}}\gamma[\theta\rho^\epsilon+(1-\theta)\rho^{I,0}]^{\gamma-1}(L^\epsilon +\sqrt{\epsilon} \Phi^\epsilon)U^\epsilon_x\mathrm{d}x\\
\le & \frac{1}{12} \int_I (U_x^\epsilon)^2 \mathrm{d}x + C [(\|\rho^{B,1}\|_{L_y^\infty}^2 +\|\rho^{b,1}\|_{L_z^\infty}^2)+\epsilon(\|\rho^{B,2}\|_{L_y^\infty}^2 +\|\rho^{b,2}\|_{L_z^\infty}^2)]\|\rho^{\epsilon}-\rho^{I,0}\|_{L^2(I)}^2\\
&+C\|\Phi^{\epsilon}\|_{L^2(I)}^2+\frac{C}{\epsilon}\|L^{\epsilon}\|_{L^2(I)}^2\\
 \le & \frac{1}{12} \int_I (U_x^\epsilon)^2 \mathrm{d}x +C\|\Phi^{\epsilon}\|_{L^2(I)}^2+C\epsilon\sqrt\epsilon.\\
\end{aligned}
$$
Now we estimate $J_{2,2}$ and $J_{2,3}$. Using Propositions   \ref{exi-ns} and \ref{B-b-Pro}, (\ref{L2-B-b}), H$\rm{\ddot{o}}$lder's inequality, and Young's inequality, we get
$$\begin{aligned}
J_{2,2} \le & C\frac{1}{\sqrt\epsilon}\|(\rho^{I,0})^{\gamma-2} \|_{L^\infty(I)}\|\frac{\epsilon}{2}[(\rho^{B,1})^2+(\rho^{b,1})^2] \|_{L^2(I)}\|U^\epsilon_x\|_{L^2(I)}\\
& +C\frac{1}{\sqrt{\epsilon}}\| (\rho^{I,0})^{\gamma-3} \rho^{I,0}_x \|_{L^\infty(I)}\|\frac{\epsilon}{2}[(\rho^{B,1})^2+(\rho^{b,1})^2] \|_{L^2(I)}\|U^\epsilon\|_{L^2(I)}\\
\le & (\frac{1}{12}\|U^{\epsilon}_x\|_{L^2(I)}^2+C\epsilon\sqrt\epsilon)+(C\epsilon\sqrt\epsilon+C\|\sqrt{\rho^{\epsilon}}U^\epsilon\|_{L^2(I)}^2) \le \frac{1}{12}\|U^{\epsilon}_x\|_{L^2(I)}^2+C\epsilon\sqrt\epsilon +C\|\sqrt{\rho^{\epsilon}}U^\epsilon\|_{L^2(I)}^2\\
\end{aligned}
$$
and
$$\begin{aligned}
J_{2,3} \le & C\frac{1}{\sqrt\epsilon}\|(\rho^{I,0})^{\gamma-2} \rho^{I,0}_x\|_{L^\infty(I)}\|{\epsilon}(\rho^{B,2}+\rho^{b,2})\|_{L^2(I)}\|U^\epsilon\|_{L^2(I)}
\le  C\epsilon\sqrt\epsilon+C\|\sqrt{\rho^{\epsilon}}U^\epsilon\|_{L^2(I)}^2.\\
\end{aligned}
$$

Thus, we have
$$\begin{aligned}
J_2=&J_{2,1} + J_{2,2}+J_{2,3}
\le  \frac{1}{6}\|U^{\epsilon}_x\|_{L^2(I)}^2 +C\|\Phi^{\epsilon}\|_{L^2(I)}^2+C\epsilon\sqrt\epsilon+C\|\sqrt{\rho^{\epsilon}}U^\epsilon\|_{L^2(I)}^2.
\end{aligned}
$$

From Proposition \ref{B-b-Pro} and Lemma \ref{clm:h},  we get
$$
\begin{aligned}
J_{3}=&\int_I\Omega_4 U^\epsilon\mathrm{d}x=\int_I(\rho^{b,1}\rho^{B,1}_{yyy}+ \rho^{B,1}\rho^{b,1}_{zzz} )U^\epsilon\mathrm{d}x\\
\le & \int_I|\rho^{b,1}\rho^{B,1}_{yyy}| \left|\int_0^x U_a^\epsilon(a,t)\mathrm{d}a\right|\mathrm{d}x + \int_I|\rho^{B,1}\rho^{b,1}_{zzz}| \left|\int_x^1 U_a^\epsilon(a,t)\mathrm{d}a\right|\mathrm{d}x\\
\le &\frac{1}{12} \int_I (U^\epsilon_x)^2
 + C{\epsilon} \| \rho^{b,1}\|_{L_y^2}^2 \| \langle y \rangle\rho^{B,1}_{yyy}\|_{L_y^2}^2 + C{\epsilon} \| \rho^{B,1}\|_{L_z^2}^2 \| \langle z \rangle\rho^{b,1}_{zzz}\|_{L_z^2}^2\\
\le &\frac{1}{12} \int_I (U^\epsilon_x)^2 +C\epsilon.
\end{aligned}
$$
Thus, substituting the estimates of $J_{1}, ~J_2$ and $J_3$ into (\ref{forme}), we get the result (\ref{L2-Esti-2}).
\end{proof}
\begin{Lemma}\label{Phi-phix-L2}
Under the conditions of Theorem \ref{assumption}, there holds that
\begin{equation}\label{phi-epsilon-x-L2-esti}
\begin{aligned}
&\frac{1}{2}\frac{\mathrm{d}}{\mathrm{d}t} \int_I \epsilon(\Phi^\epsilon_x)^2\mathrm{d}x+\int_I \epsilon \rho^\epsilon \Phi^\epsilon_{xxx} U^\epsilon \mathrm{d}x
\le   C \epsilon\|\Phi^\epsilon_x\|_{L^2(I)}^2+ C\epsilon +C\epsilon\|\Phi^\epsilon\|_{L^2(I)}^2.
\end{aligned}
\end{equation}
\end{Lemma}
\begin{proof}
Using (\ref{zhankai-equ})$_{1}$ and (\ref{zhankai-equ})$_{4}$, we get
\begin{equation}\label{L2-3_2-4}
\begin{aligned}
&\int_I \epsilon \rho^\epsilon \Phi^\epsilon_{xxx} U^\epsilon\mathrm{d}x=- \int_I \epsilon ( \rho^\epsilon U^\epsilon)_x \Phi^\epsilon_{xx} \mathrm{d}x
= \int_I \epsilon( \Phi_t^\epsilon+ h_1+h_2+h_3+h_4+h_5) \Phi^\epsilon_{xx} \mathrm{d}x \\
=& -\frac{1}{2}\frac{\mathrm{d}}{\mathrm{d}t}\int_I \epsilon (\Phi^\epsilon_x)^2\mathrm{d}x+\int_I \epsilon h_2 \Phi^\epsilon_{xx} \mathrm{d}x -\int_I \epsilon (h_1+h_3 +h_4 +h_5)_x \Phi^\epsilon_{x} \mathrm{d}x.\\
\end{aligned}
\end{equation}

Using
$$
\begin{aligned}
&\int_I  \epsilon [u^{I,0} + \sqrt\epsilon M+{\epsilon}(u^{B,2}+u^{b,2})+{\epsilon}\sqrt\epsilon(u^{B,3}+u^{b,3})] \Phi^{\epsilon}_{xx} \Phi^{\epsilon}_{x}  \mathrm{d}x\\
= &- \int_I  \epsilon [u^{I,0}_x + \sqrt\epsilon M_x+\sqrt{\epsilon}(u^{B,2}_y+u^{b,2}_z)+{\epsilon}(u^{B,3}_y+u^{b,3}_z)] \frac{1}{2}(\Phi^{\epsilon}_{x})^2  \mathrm{d}x
\end{aligned}
$$
and integration by parts on $\int_I \epsilon h_2 \Phi^\epsilon_{xx} \mathrm{d}x$, we get
\begin{equation}\label{L2-3_2-1}
\begin{aligned}
&- \int_I \epsilon (h_2)_x \Phi^\epsilon_{x} \mathrm{d}x\\
=& -\int_I \epsilon \{\frac{1}{\sqrt{\epsilon}}L^\epsilon_{xt}  +u^{I,0}_{xx}(\frac{1}{\sqrt{\epsilon}}L^\epsilon + \Phi^\epsilon) + \frac{3}{2}u^{I,0}_{x} \Phi^\epsilon_x + 2u^{I,0}_{x}\frac{1}{\sqrt{\epsilon}}L^\epsilon_x +u^{I,0}\frac{1}{\sqrt{\epsilon}}L^\epsilon_{xx}\\
&+\frac{1}{\sqrt{\epsilon}}[M^\epsilon (\rho^\epsilon-\sqrt\epsilon\Phi^\epsilon)_{xx}+\frac{3}{2}M^\epsilon_x \sqrt\epsilon\Phi^\epsilon_{x}+2M^\epsilon_x (\rho^\epsilon_{x}-\sqrt\epsilon\Phi^\epsilon_x)]\\
&+[\frac{1}{\sqrt\epsilon}(u^{B,2}_{yy}+u^{b,2}_{zz})+(u^{B,3}_{yy}+u^{b,3}_{zz})]\sqrt{\epsilon} \Phi^\epsilon+\frac{3}{2}[(u^{B,2}_y+u^{b,2}_z)+\sqrt{\epsilon}(u^{B,3}_y+u^{b,3}_z)]\sqrt{\epsilon} \Phi^\epsilon_x\}\Phi^\epsilon_{x} \mathrm{d}x\\
:=& h_{2,1}^x - \int_I \epsilon \frac{3}{2}[ u^{I,0}_{x}
+M^\epsilon_x +\sqrt{\epsilon}(u^{B,2}_y+u^{b,2}_z)+{\epsilon}(u^{B,3}_y+u^{b,3}_z)] (\Phi^\epsilon_{x})^2 \mathrm{d}x.\\
\le & h_{2,1}^x + C \| u^{I,0}_{x}
+M^\epsilon_x +\sqrt{\epsilon}(u^{B,2}_y+u^{b,2}_z)+{\epsilon}(u^{B,3}_y+u^{b,3}_z)\|_{L^\infty(I)} \epsilon\|\Phi^\epsilon_x\|_{L^2(I)}^2\\
 \le & h_{2,1}^x + C \epsilon\|\Phi^\epsilon_x\|_{L^2(I)}^2.
\end{aligned}
\end{equation}

Using $\|(\rho^\epsilon-\sqrt\epsilon\Phi^\epsilon)_{xx}\|_{L^2(I)}^2 \le C(1 +\frac{1}{\sqrt\epsilon})$, Propositions \ref{exi-ns} and \ref{B-b-Pro},  Lemma \ref{L-M-well-pose}, %Remark \ref{U-B3-L-infty},
 (\ref{L2-B-b}), (\ref{rho-epsilon-phi-epsilon}), H$\rm{\ddot{o}}$lder's inequality, and Young's inequality, it holds that
\begin{equation}\label{L2-3_2-2}
\begin{aligned}
h_{2,1}^x=&-\int_I \{L^\epsilon_{xt}  +u^{I,0}_{xx}(L^\epsilon + \sqrt\epsilon\Phi^\epsilon) + 2u^{I,0}_{x}L^\epsilon_x+u^{I,0}L^\epsilon_{xx}+M^\epsilon (\rho^\epsilon-\sqrt\epsilon\Phi^\epsilon)_{xx}\\
&+2M^\epsilon_x (\rho^\epsilon_{x}-\sqrt\epsilon\Phi^\epsilon_x)+[(u^{B,2}_{yy}+u^{b,2}_{zz})+\sqrt\epsilon(u^{B,3}_{yy}+u^{b,3}_{zz})]\sqrt{\epsilon} \Phi^\epsilon\}\sqrt\epsilon\Phi^\epsilon_{x} \mathrm{d}x\\
\le  &C\epsilon\|\Phi^\epsilon_{x}\|_{L^2(I)}^2 + C\|L^\epsilon_{xt}\|_{L^2(I)}^2+C\|u^{I,0}_{xx}\|_{L^\infty(I)}^2\|L^\epsilon + \sqrt\epsilon\Phi^\epsilon \|_{L^2(I)}^2\\
&+C\|u^{I,0}_{x}\|_{L^\infty(I)}^2\|L^\epsilon_x\|_{L^2(I)}^2+C\|u^{I,0}\|_{L^\infty(I)}^2\|L^\epsilon_{xx}\|_{L^2(I)}^2\\
&+C\|M^\epsilon \|_{L^\infty(I)}^2\|(\rho^\epsilon-\sqrt\epsilon\Phi^\epsilon)_{xx}\|_{L^2(I)}^2+C\|M^\epsilon_x \|_{L^\infty(I)}^2\|\rho^\epsilon_{x}-\sqrt\epsilon\Phi^\epsilon_x\|_{L^2(I)}^2\\
&+C\|(u^{B,2}_{yy}+u^{b,2}_{zz})+\sqrt\epsilon(u^{B,3}_{yy}+u^{b,3}_{zz})]|_{L^\infty(I)}^2\|\sqrt{\epsilon} \Phi^\epsilon\|_{L^2(I)}^2\\
\le &C(\epsilon +\epsilon\|\Phi^\epsilon\|_{L^2(I)}^2+\epsilon\|\Phi^\epsilon_{x}\|_{L^2(I)}^2)+C\epsilon\|\Phi^\epsilon_{x}\|_{L^2(I)}^2 \le C(\epsilon +\epsilon\|\Phi^\epsilon\|_{L^2(I)}^2+\epsilon\|\Phi^\epsilon_{x}\|_{L^2(I)}^2).
\end{aligned}
\end{equation}

Using Lemmas \ref{clm:h} and \ref{clm:no-impor}, we get
\begin{equation}\label{L2-3_2-3}
-\int_I \epsilon (h_1 +h_3 +h_4 +h_5)_x \Phi^\epsilon_{x} \mathrm{d}x \le C\epsilon\|\Phi^\epsilon_{x}\|_{L^2(I)}^2 + C\epsilon \sum_{i\neq 2}\|(h_i)_x\|_{L^2(I)}^2 \le C\epsilon\|\Phi^\epsilon_{x}\|_{L^2(I)}^2 + C\epsilon \sqrt\epsilon.
\end{equation}

Hence, putting (\ref{L2-3_2-1})-(\ref{L2-3_2-3}) into (\ref{L2-3_2-4}), we get (\ref{phi-epsilon-x-L2-esti}).
\end{proof}

\begin{Corollary}\label{Phi-U-L2-all}
Under the conditions of Theorem \ref{assumption}, there holds that
\begin{equation}\label{Ux-Phix-L2L22-1}
\sup_{t\in[0,T_1]}(\|U^\epsilon\|_{ L^2(I)}^2+\|\Phi^\epsilon\|_{L^2(I)}^2+\epsilon\|\Phi^\epsilon_{x}\|_{L^2(I)}^2) +\|U_{x}^\epsilon\|_{L_{T_1}^2 L^2}^2\le  B_1(T_1)\epsilon \le  B_1(T)\epsilon,
\end{equation}
where $B_1(s)=C(1+s)+C  s(1+s) [1+s\exp\{Cs\}]$ for any $s\in[0,T]$.
\end{Corollary}
\begin{proof}
In conclusion, together with Lemmas \ref{Phi-U-L2}, \ref{Phi-U-L2-U} and \ref{Phi-phix-L2}, we obtain
\begin{equation}\label{U-L2-half}
\begin{aligned}
&\frac{1}{2}\frac{\mathrm{d}}{\mathrm{d}t} \int_I [\rho^{\epsilon} (U^\epsilon)^2+(\Phi^\epsilon)^2+\epsilon(\Phi^\epsilon_x)^2] \mathrm{d}x +\int_I (U^\epsilon_x)^2 \mathrm{d}x\\
\le &\frac{1}{2} \int_I (U^\epsilon_x)^2\mathrm{d}x +C\|\sqrt{\rho^{\epsilon}}U^\epsilon\|_{L^2(I)}^2 +C\| \Phi^\epsilon\|_{L^2(I)}^2+C{\epsilon}\| \Phi^\epsilon_x\|_{L^2(I)}^2 +C\epsilon.
\end{aligned}
\end{equation}

Using (\ref{equ-epsilon}), (\ref{equ-NS}) (\ref{IBb-3})-(\ref{initial-value}), and definition of $L^\epsilon$, $M^\epsilon$, we get the estimates of initial value of system (\ref{zhankai-equ}) as follows:
\begin{equation}\begin{aligned}\label{ini-val-phi-11}
&\|\Phi^\epsilon(\cdot,0) \|_{L^2(I)}^2+\|U^\epsilon(\cdot,0)\|_{L^2(I)}^2 +\epsilon\|\Phi^\epsilon_x(\cdot,0) \|_{L^2(I)}^2\\
=&\frac{1}{\epsilon}\|\rho^{\epsilon}_0-\rho^{I,0}_0\|_{L^2(I)}^2 +\frac{1}{\epsilon}\|u^\epsilon_0 -u^{I,0}_0\|_{L^2(I)}^2+\|(\rho^{\epsilon}_0-\rho^{I,0}_0)_x\|_{L^2(I)}^2\le C\epsilon^{m-1},
\end{aligned}
 \end{equation}for any $m> 2$.

Thus, using Gronwall's inequality, it holds that
$$\begin{aligned}
&\sup_{t\in[0,T_1]} \int_I [\rho^{\epsilon} (U^\epsilon)^2+(\Phi^\epsilon)^2+\epsilon(\Phi^\epsilon_x)^2] \mathrm{d}x\\
 \le &(\|\Phi^\epsilon(\cdot,0) \|_{L^2(I)}^2 +\epsilon\|\Phi^\epsilon_x(\cdot,0) \|_{L^2(I)}^2+\|\rho^{I,0}_0\|_{L^\infty(I)} \|U^\epsilon(\cdot,0)\|_{L^2(I)}^2 + CT_1\epsilon)\\
& +CT_1 (\|\Phi^\epsilon(\cdot,0) \|_{L^2(I)}^2 +\epsilon\|\Phi^\epsilon_x(\cdot,0) \|_{L^2(I)}^2+\|\rho^{I,0}_0\|_{L^\infty(I)} \|U^\epsilon(\cdot,0)\|_{L^2(I)}^2 + CT_1 \epsilon)\exp\{CT_1\}\\
\le & C \epsilon (1+T_1) [1+T_1\exp\{CT_1\}],
\end{aligned}
$$
which implies that
\begin{equation}\label{Gronwall-L2}
\begin{aligned}
&\sup_{t\in[0,T_1]} \int_I [\rho^{\epsilon} (U^\epsilon)^2+(\Phi^\epsilon)^2+\epsilon(\Phi^\epsilon_x)^2] \mathrm{d}x +\int_0^{T_1}\int_I (U^\epsilon_x)^2 \mathrm{d}x \mathrm{d}t\\
 \le &(\|\Phi^\epsilon(\cdot,0) \|_{L^2(I)}^2 +\epsilon\|\Phi^\epsilon_x(\cdot,0) \|_{L^2(I)}^2+\|\rho^{I,0}_0\|_{L^\infty(I)} \|U^\epsilon(\cdot,0)\|_{L^2(I)}^2 + CT_1 \epsilon)\\
& +C \epsilon T_1(1+T_1) [1+T_1\exp\{CT_1\}]\\
\le & C(1+T_1)\epsilon+C \epsilon T_1(1+T_1) [1+T_1\exp\{CT_1\}]:=B_1(T_1)\epsilon\\
\le & B_1(T) \epsilon.
\end{aligned}
\end{equation}

Thus, we get (\ref{Ux-Phix-L2L22-1}).
\end{proof}

Corollary \ref{Phi-U-L2-all} immediately yields the following result.
\begin{Corollary}\label{clm:some-impor} Under the conditions of Corollary \ref{Phi-U-L2-all}, there holds that
\begin{equation}\label{some-rho-1}
\left\{
\begin{aligned}
&\|\rho^\epsilon_x\|_{L^2(I)}^2 \le C,~ \| (\rho^\epsilon-\sqrt\epsilon\Phi^\epsilon)_{xx}\|_{L^2(I)}^2  \le \frac{C}{\sqrt\epsilon}, \\
& \| \rho^\epsilon_{xx}\|_{L^2(I)}^2  \le \frac{C}{\sqrt\epsilon} + C\epsilon \| \Phi^\epsilon_{xx}\|_{L^2(I)}^2,\\
&\|\Phi^\epsilon\|_{L^\infty(I)}^2 \le C\|\Phi^\epsilon\|_{L^2(I)}^2 + C \|\Phi^\epsilon\|_{L^2(I)}\|\Phi^\epsilon_x\|_{L^2(I)} \le C\sqrt\epsilon,\\
&\|\rho^\epsilon_{t}\|_{L^2(I)}^2 \le C +C\epsilon \|\Phi^\epsilon_t\|_{L^2(I)}^2 ,\\
& \| \rho^\epsilon_{xt}\|_{L^2(I)}^2  \le C +C\epsilon \|\Phi^\epsilon_{xt}\|_{L^2(I)}^2,
\end{aligned}
\right.
\end{equation}
and that
\begin{equation}\label{some-h-2}
\left\{
\begin{aligned}
&\|(h_2)_x-u^{I,0}\Phi^\epsilon_{xx}-2u_x^{I,0} \Phi^\epsilon_x\|_{L^2(I)}^2 \le
 C\sqrt\epsilon + C \epsilon^2 \|\Phi^\epsilon_{xx}\|_{L^2(I)}^2,\\
& \|\Phi^\epsilon_t\|_{L^2(I)}^2 %\le   C \epsilon%+\frac{C_2}{\epsilon} \|L^\epsilon_{t}\|_{L^2(I)}^2
 %+ C \|\Phi^\epsilon_x\|_{L^2(I)}^2+C\|\sqrt{\rho^\epsilon}U^\epsilon_x\|_{L^2(I)}^2
 \le C +C\|\sqrt{\rho^\epsilon}U^\epsilon_x\|_{L^2(I)}^2,\\
& \|\Phi^\epsilon_{xt}\|_{L^2(I)}^2 \le   C\sqrt\epsilon + C(1+\epsilon\|U^\epsilon_x\|_{L^2(I)}^2) \|\Phi^\epsilon_{xx}\|_{L^2(I)}^2 + \frac{C}{\sqrt\epsilon}\|U^\epsilon_x\|_{L^2(I)}^2
+ C  \|U^\epsilon_{xx}\|_{L^2(I)}^2,\\
&\|(h_{2})_t\|_{L^2(I)}^2 \le \frac{C}{\epsilon} \|L^\epsilon_{tt}\|_{L^2(I)}^2 + C
+ C \|\Phi^\epsilon_{t}\|_{L^2(I)}^2+C \|\Phi^\epsilon_{xt}\|_{L^2(I)}^2 +C\epsilon^2 \sqrt\epsilon\|\Phi^\epsilon_{xx}\|_{L^2(I)}^2,\\
&\|(h_3)_t\|_{L^2(I)}^2 \le {C}{\epsilon}\sqrt\epsilon \|\rho^{B,2}_{tt}\|_{L^2_y}^2  +C\epsilon ,\\
&\|(h_4)_t\|_{L^2(I)}^2 \le {C}{\epsilon}\sqrt\epsilon \|\rho^{b,2}_{tt}\|_{L^2_z}^2  +C\epsilon  ,\\
&\|(h_{5})_t\|_{L^2(I)}^2
 \le C \epsilon \sqrt\epsilon.\\
\end{aligned}\right.
\end{equation}
\end{Corollary}
\begin{proof}
Firstly, using Propositions \ref{exi-ns} and \ref{B-b-Pro}, Lemma \ref{L-M-well-pose}, (\ref{L2-B-b}), Corollary \ref{Phi-U-L2-all}, and Sobolev's inequality, we get (\ref{some-rho-1}) after the direct estimates.

Secondly, we estimate (\ref{some-h-2})$_1$.

Redefine $(h_2)_x-u^{I,0}\Phi^\epsilon_{xx}-2u_x^{I,0} \Phi^\epsilon_x:= h_{2,1}' + h_{2,2}' + h_{2,3}'$
 as
$$
\begin{aligned}
h_{2,1}'=&\frac{1}{\sqrt{\epsilon}}L^\epsilon_{xt}  +u^{I,0}_{xx}(\frac{1}{\sqrt{\epsilon}}L^\epsilon + \Phi^\epsilon) +2u^{I,0}_x\frac{1}{\sqrt{\epsilon}}L^\epsilon_x +u^{I,0}\frac{1}{\sqrt{\epsilon}}L^\epsilon_{xx}, \\
h_{2,2}'=&\frac{2}{\sqrt{\epsilon}}M^\epsilon_x \rho^\epsilon_x+\frac{1}{\sqrt{\epsilon}}M^\epsilon (\rho^\epsilon-\sqrt\epsilon\Phi^\epsilon)_{xx} + \frac{1}{\sqrt{\epsilon}}M^\epsilon \sqrt\epsilon \Phi^\epsilon_{xx},\\
h_{2,3}'=&[(u^{B,2}_{yy}+u^{b,2}_{zz})+\sqrt{\epsilon}(u^{B,3}_{yy}+u^{b,3}_{zz})] \Phi^\epsilon+[2(u^{B,2}_y+u^{b,2}_z)+2\sqrt{\epsilon}(u^{B,3}_y+u^{b,3}_z)]\sqrt{\epsilon} \Phi^\epsilon_x\\
&+[\sqrt{\epsilon}(u^{B,2}+u^{b,2})+{\epsilon}(u^{B,3}+u^{b,3})]\sqrt{\epsilon} \Phi^\epsilon_{xx}.
\end{aligned}
$$
Then, using Propositions \ref{exi-ns} and \ref{B-b-Pro},  Lemma \ref{L-M-well-pose},
Corollary \ref{Phi-U-L2-all}, (\ref{L2-B-b}), (\ref{some-rho-1}), and H$\rm{\ddot{o}}$lder's inequality, we get
$$
\begin{aligned}
\|h_{2,1}'\|_{L^2(I)}^2 \le & C\|\frac{1}{\sqrt{\epsilon}}L^\epsilon_{xt} \|_{L^2(I)}^2+ C\|u^{I,0}_{xx}\|_{L^\infty(I)}^2 \|\frac{1}{\sqrt{\epsilon}}L^\epsilon + \Phi^\epsilon\|_{L^2(I)}^2+ C\|u^{I,0}_x\|_{L^\infty(I)}^2 \|\frac{1}{\sqrt{\epsilon}}L^\epsilon_x \|_{L^2(I)}^2\\
&+ C\|u^{I,0}\|_{L^\infty(I)}^2 \|\frac{1}{\sqrt{\epsilon}}L^\epsilon_{xx}\|_{L^2(I)}^2\\
\le & C\epsilon + C \epsilon + C \epsilon + C\epsilon \le C\epsilon,
\end{aligned}
$$
$$
\begin{aligned}
\|h_{2,2}'\|_{L^2(I)}^2 \le & C\|\frac{2}{\sqrt{\epsilon}}M^\epsilon_x \|_{L^\infty(I)}^2 \|\rho^\epsilon_x\|_{L^2(I)}^2+ C\|\frac{1}{\sqrt{\epsilon}}M^\epsilon\|_{L^\infty(I)}^2 \| (\rho^\epsilon-\sqrt\epsilon\Phi^\epsilon)_{xx} \|_{L^2(I)}^2\\
&+ C\|\frac{1}{\sqrt{\epsilon}}M^\epsilon \|_{L^\infty(I)}^2 \|\sqrt\epsilon \Phi^\epsilon_{xx}\|_{L^2(I)}^2\\
\le & C\epsilon+ C\sqrt \epsilon + C\epsilon^2 \|\Phi^\epsilon_{xx}\|_{L^2(I)}^2 \le C\sqrt \epsilon + C\epsilon^2 \|\Phi^\epsilon_{xx}\|_{L^2(I)}^2,
\end{aligned}
$$
and
$$
\begin{aligned}
\|h_{2,3}'\|_{L^2(I)}^2 \le & C\|(u^{B,2}_{yy}+u^{b,2}_{zz})+\sqrt{\epsilon}(u^{B,3}_{yy}+u^{b,3}_{zz})\|_{L^2(I)}^2 \|\Phi^\epsilon\|_{L^\infty(I)}^2\\
&+ C\|2(u^{B,2}_y+u^{b,2}_z)+2\sqrt{\epsilon}(u^{B,3}_y+u^{b,3}_z)\|_{L^\infty(I)}^2 \|\sqrt{\epsilon} \Phi^\epsilon_x\|_{L^2(I)}^2\\
&+ C\|[\sqrt{\epsilon}(u^{B,2}+u^{b,2})+{\epsilon}(u^{B,3}+u^{b,3})]\|_{L^\infty(I)}^2 \|\sqrt{\epsilon} \Phi^\epsilon_{xx}\|_{L^2(I)}^2\\
\le & C\epsilon+C\epsilon + C \epsilon^2 \|\Phi^\epsilon_{xx}\|_{L^2(I)}^2 \le C\epsilon + C \epsilon^2 \|\Phi^\epsilon_{xx}\|_{L^2(I)}^2,
\end{aligned}
$$
which imply that
$$
\|(h_2)_x-u^{I,0}\Phi^\epsilon_{xx}-2u_x^{I,0} \Phi^\epsilon_x\|_{L^2(I)}^2 \le C \sum_{i=1}^3\|h_{2,i}'\|_{L^2(I)}^2 \le C\sqrt\epsilon + C \epsilon^2 \|\Phi^\epsilon_{xx}\|_{L^2(I)}^2.
$$

Thirdly, we estimate (\ref{some-h-2})$_{5,6}$.

Using Propositions \ref{exi-ns} and \ref{B-b-Pro},
Corollary \ref{Phi-U-L2-all},  Lemmas \ref{L-M-well-pose} and \ref{clm-rho-bdd}, (\ref{L2-B-b}), (\ref{some-rho-1}), and H$\rm{\ddot{o}}$lder's inequality, it holds that
$$
\begin{aligned}
\|h_2\|_{L^2(I)}^2= &\|\frac{1}{\sqrt{\epsilon}}L^\epsilon_t  +u^{I,0}_x(\frac{1}{\sqrt{\epsilon}}L^\epsilon + \Phi^\epsilon) +u^{I,0}(\frac{1}{\sqrt{\epsilon}}L^\epsilon_x + \Phi^\epsilon_x)+\frac{1}{\sqrt{\epsilon}}M^\epsilon_x \rho^\epsilon+\frac{1}{\sqrt{\epsilon}}M^\epsilon \rho^\epsilon_x\\
&+[(u^{B,2}_y+u^{b,2}_z)+\sqrt{\epsilon}(u^{B,3}_y+u^{b,3}_z)]\sqrt{\epsilon} \Phi^\epsilon+[\sqrt{\epsilon}(u^{B,2}+u^{b,2})+{\epsilon}(u^{B,3}+u^{b,3})]\sqrt{\epsilon} \Phi^\epsilon_x\|_{L^2(I)}^2\\
\le & C\|\frac{1}{\sqrt{\epsilon}}L^\epsilon_t\|_{L^2(I)}^2+C\|u^{I,0}_x\|_{L^\infty(I)}^2\|\frac{1}{\sqrt{\epsilon}}L^\epsilon + \Phi^\epsilon\|_{L^2(I)}^2+C\|u^{I,0}\|_{L^\infty(I)}^2\|\frac{1}{\sqrt{\epsilon}}L^\epsilon_x + \Phi^\epsilon_x\|_{L^2(I)}^2\\
& +C\|\frac{1}{\sqrt{\epsilon}}M^\epsilon_x \|_{L^2(I)}^2\|\rho^\epsilon\|_{L^\infty(I)}^2+C\|\frac{1}{\sqrt{\epsilon}}M^\epsilon \|_{L^\infty(I)}^2\|\rho^\epsilon_x\|_{L^2(I)}^2\\
&+C\|(u^{B,2}_y+u^{b,2}_z)+\sqrt{\epsilon}(u^{B,3}_y+u^{b,3}_z)\|_{L^\infty(I)}^2\|\sqrt{\epsilon} \Phi^\epsilon\|_{L^2(I)}^2\\
&+C\|\sqrt{\epsilon}(u^{B,2}+u^{b,2})+{\epsilon}(u^{B,3}+u^{b,3})\|_{L^\infty(I)}^2\|\sqrt{\epsilon} \Phi^\epsilon_x\|_{L^2(I)}^2\\
\le & {C}{\epsilon} +C\epsilon + C(\epsilon + \|\Phi^\epsilon_x\|_{L^2(I)}^2) + C\epsilon +C\epsilon+ C\epsilon^2 + C\epsilon^2 \le C(\epsilon + \|\Phi^\epsilon_x\|_{L^2(I)}^2).
\end{aligned}
$$
Thus, using (\ref{zhankai-equ})$_1$,  Poincar\'e's inequality, (\ref{some-rho-1})$_{1,2}$, (\ref{some-h-2})$_1$, Lemmas \ref{clm-rho-bdd}, \ref{clm:h} and \ref{clm:no-impor}, and above estimate, we get
$$\begin{aligned}
\|\Phi^\epsilon_t\|_{L^2(I)}^2 \le & C\sum_{i=1}^5\|h_{i}\|_{L^2(I)}^2 + C\|\rho^\epsilon_x U^\epsilon + \rho^\epsilon U^\epsilon_x\|_{L^2(I)}^2\\
\le & C \epsilon + C \|\Phi^\epsilon_x\|_{L^2(I)}^2 + C(\|\rho^\epsilon_x\|_{L^2(I)}^2+\|\rho^\epsilon\|_{L^\infty(I)}^2)\| \sqrt{\rho^\epsilon}U^\epsilon_x\|_{L^2(I)}^2\\
\le & C  + C\| \sqrt{\rho^\epsilon}U^\epsilon_x\|_{L^2(I)}^2,
\end{aligned}
$$
and
$$\begin{aligned}
\|\Phi^\epsilon_{xt}\|_{L^2(I)}^2 \le & C\sum_{i=1}^5\|(h_{i})_x\|_{L^2(I)}^2 + C\|(\rho^\epsilon-\sqrt\epsilon\Phi^\epsilon)_{xx} U^\epsilon +\sqrt\epsilon\Phi^\epsilon_{xx} U^\epsilon + 2\rho^\epsilon_x U^\epsilon_x+ \rho^\epsilon U^\epsilon_{xx}\|_{L^2(I)}^2\\
\le & C\sqrt\epsilon + C \|\Phi^\epsilon_{xx}\|_{L^2(I)}^2 + \frac{C}{\sqrt\epsilon}\|U^\epsilon_x\|_{L^2(I)}^2 + C\epsilon \|\Phi^\epsilon_{xx}\|_{L^2(I)}^2\|U^\epsilon_x\|_{L^2(I)}^2\\
& +(\frac{C}{\sqrt\epsilon} + C\epsilon \| \Phi^\epsilon_{xx}\|_{L^2(I)}^2)\|U^\epsilon_x\|_{L^2(I)}^2+ C\|\rho^\epsilon\|_{L^\infty(I)}^2 \|U^\epsilon_{xx}\|_{L^2(I)}^2\\
\le &  C\sqrt\epsilon + C(1+\epsilon\|U^\epsilon_x\|_{L^2(I)}^2) \|\Phi^\epsilon_{xx}\|_{L^2(I)}^2 + \frac{C}{\sqrt\epsilon}\|U^\epsilon_x\|_{L^2(I)}^2 + C   \|U^\epsilon_{xx}\|_{L^2(I)}^2.
\end{aligned}
$$

Fourthly, we estimate (\ref{some-h-2})$_{2,3,4,7}$ by using Propositions \ref{exi-ns} and \ref{B-b-Pro}, Lemma \ref{L-M-well-pose}, %Remark \ref{U-B3-L-infty}
 and (\ref{L2-B-b}).

 Redefine $(h_2)_t:= h_{2,1}^t + h_{2,2}^t + h_{2,3}^t+h_{2,4}^t$ as
$$
\begin{aligned}
h_{2,1}^t = &\frac{1}{\sqrt{\epsilon}}L^\epsilon_{tt}  +u^{I,0}_{xt}(\frac{1}{\sqrt{\epsilon}}L^\epsilon + \Phi^\epsilon) +u^{I,0}_{x}(\frac{1}{\sqrt{\epsilon}}L^\epsilon_t + \Phi^\epsilon_t) +u^{I,0}_t(\frac{1}{\sqrt{\epsilon}}L^\epsilon_x + \Phi^\epsilon_x) +u^{I,0}(\frac{1}{\sqrt{\epsilon}}L^\epsilon_{xt} + \Phi^\epsilon_{xt}),\\
h_{2,2}^t =&\frac{1}{\sqrt{\epsilon}}(M^\epsilon_t \rho^\epsilon_x+M^\epsilon \rho^\epsilon_{xt} +M^\epsilon_{xt} \rho^\epsilon+M^\epsilon_x \rho^\epsilon_{t} ),\\
h_{2,3}^t =&[(u^{B,2}_{yt}+u^{b,2}_{zt})+\sqrt{\epsilon}(u^{B,3}_{yt}+u^{b,3}_{zt})]\sqrt{\epsilon} \Phi^\epsilon +[(u^{B,2}_y+u^{b,2}_z)+\sqrt{\epsilon}(u^{B,3}_y+u^{b,3}_z)]\sqrt{\epsilon} \Phi^\epsilon_t,\\
h_{2,4}^t =&[\sqrt{\epsilon}(u^{B,2}_t+u^{b,2}_t)+{\epsilon}(u^{B,3}_t+u^{b,3}_t)]\sqrt{\epsilon} \Phi^\epsilon_x +[\sqrt{\epsilon}(u^{B,2}+u^{b,2})+{\epsilon}(u^{B,3}+u^{b,3})]\sqrt{\epsilon} \Phi^\epsilon_{xt}.\\
\end{aligned}
$$
Then, using  Poincar\'e's inequality, H$\rm{\ddot{o}}$lder's inequality, and  (\ref{some-rho-1}),
we get
$$\begin{aligned}
\|h_{2,1}^t\|_{L^2(I)}^2 \le &C\|\frac{1}{\sqrt{\epsilon}}L^\epsilon_{tt}\|_{L^2(I)}^2 + C\|u^{I,0}_{xt}\|_{L^2(I)}^2\|\frac{1}{\sqrt{\epsilon}}L^\epsilon + \Phi^\epsilon\|_{L^\infty(I)}^2+ C\|u^{I,0}_{x}\|_{L^\infty(I)}^2\|(\frac{1}{\sqrt{\epsilon}}L^\epsilon_t + \Phi^\epsilon_t)\|_{L^2(I)}^2\\
&+ C\|u^{I,0}_t \|_{L^\infty(I)}^2\|\frac{1}{\sqrt{\epsilon}}L^\epsilon_x + \Phi^\epsilon_x\|_{L^2(I)}^2+ C\|u^{I,0}\|_{L^\infty(I)}^2\|\frac{1}{\sqrt{\epsilon}}L^\epsilon_{xt} + \Phi^\epsilon_{xt}\|_{L^2(I)}^2\\
\le & \frac{C}{\epsilon} \|L^\epsilon_{tt}\|_{L^2(I)}^2 + C \sqrt\epsilon + C(\epsilon+\|\Phi^\epsilon_t\|_{L^2(I)}^2)+C + C(\epsilon +\|\Phi^\epsilon_{xt}\|_{L^2(I)}^2)\\
\le & \frac{C}{\epsilon} \|L^\epsilon_{tt}\|_{L^2(I)}^2  +C +C \|\Phi^\epsilon_t\|_{L^2(I)}^2 +C \|\Phi^\epsilon_{xt}\|_{L^2(I)}^2 ,
\end{aligned}
$$
$$\begin{aligned}
\|h_{2,2}^t\|_{L^2(I)}^2 \le & \frac{C}{{\epsilon}}\|M^{\epsilon}_{t}\|_{L^\infty(I)}^2\|\rho^\epsilon_x\|_{L^2(I)}^2+ \frac{C}{{\epsilon}}\|M^{\epsilon}\|_{L^\infty(I)}^2\|\rho^\epsilon_{xt}\|_{L^2(I)}^2+ \frac{C}{{\epsilon}}\|M^{\epsilon}_{xt} \|_{L^2(I)}^2\|\rho^\epsilon\|_{L^\infty(I)}^2\\
&+ \frac{C}{{\epsilon}}\|M^{\epsilon}_x\|_{L^\infty(I)}^2\|\rho^\epsilon_t\|_{L^2(I)}^2\\
\le & C\epsilon+ C(\epsilon+\epsilon^2\|\Phi^\epsilon_{xt}\|_{L^2(I)}^2)+C\epsilon+C\epsilon(1 +\epsilon \|\Phi^\epsilon_t\|_{L^2(I)}^2)\\
\le & C\epsilon+ C\epsilon^2(\|\Phi^\epsilon_{xt}\|_{L^2(I)}^2+\|\Phi^\epsilon_t\|_{L^2(I)}^2),
\end{aligned}
$$
$$\begin{aligned}
\|h_{2,3}^t\|_{L^2(I)}^2 \le & C\|(u^{B,2}_{yt}+u^{b,2}_{zt})+\sqrt{\epsilon}(u^{B,3}_{yt}+u^{b,3}_{zt})\|_{L^2(I)}^2 \|\sqrt{\epsilon} \Phi^\epsilon\|_{L^\infty(I)}^2 \\
&+C\|(u^{B,2}_y+u^{b,2}_z)+\sqrt{\epsilon}(u^{B,3}_y+u^{b,3}_z)\|_{L^\infty(I)}^2 \|\sqrt{\epsilon} \Phi^\epsilon_t\|_{L^2(I)}^2 \\
\le & C\epsilon^2+ C\epsilon\|\Phi^\epsilon_t\|_{L^2(I)}^2,
\end{aligned}
$$
and
$$\begin{aligned}
\|h_{2,4}^t\|_{L^2(I)}^2 \le & C\|\sqrt{\epsilon}(u^{B,2}_t+u^{b,2}_t)+{\epsilon}(u^{B,3}_t+u^{b,3}_t)\|_{L^2(I)}^2 \|\sqrt{\epsilon} \Phi^\epsilon_x\|_{L^\infty(I)}^2 \\
&+C\| \sqrt{\epsilon}(u^{B,2}+u^{b,2})+{\epsilon}(u^{B,3}+u^{b,3})\|_{L^\infty(I)}^2 \|\sqrt{\epsilon} \Phi^\epsilon_{xt}\|_{L^2(I)}^2 \\
\le & C\epsilon^2 \sqrt\epsilon\|\Phi^\epsilon_{xx}\|_{L^2(I)}^2+C\epsilon^2\|\Phi^\epsilon_{xt}\|_{L^2(I)}^2.
\end{aligned}
$$
Thus, we get
$$\begin{aligned}
&\|(h_{2})_t\|_{L^2(I)}^2 \le C \sum_{i=1}^4\|h_{2,i}^t\|_{L^2(I)}^2
 \le  \frac{C}{\epsilon} \|L^\epsilon_{tt}\|_{L^2(I)}^2 + C+ C \|\Phi^\epsilon_{t}\|_{L^2(I)}^2+C \|\Phi^\epsilon_{xt}\|_{L^2(I)}^2+C\epsilon^2 \sqrt\epsilon\|\Phi^\epsilon_{xx}\|_{L^2(I)}^2.
\end{aligned}
$$

Moreover, using H$\rm{\ddot{o}}$lder's inequality, we have
$$\begin{aligned}
\|(h_{3})_t\|_{L^2(I)}^2  =&\|\sqrt{\epsilon}\rho^{B,2}_{tt}  +u^{I,0}_{xt}\sqrt{\epsilon}\rho^{B,2} +u^{I,0}_x\sqrt{\epsilon}\rho^{B,2}_t+\sqrt{\epsilon}u^{B,3}_{yt}\rho^{I,0} +\sqrt{\epsilon}u^{B,3}_y\rho^{I,0}_t\\
&+\sqrt{\epsilon}u^{B,2}_t\rho^{I,0}_x+\sqrt{\epsilon}u^{B,2}\rho^{I,0}_{xt}+{\epsilon}u^{B,3}_t\rho^{I,0}_x +{\epsilon}u^{B,3}\rho^{I,0}_{xt}\|_{L^2(I)}^2\\
\le & C\|\sqrt{\epsilon}\rho^{B,2}_{tt}\|_{L^2(I)}^2+C\|u^{I,0}_{xt}\|_{L^2(I)}^2\|\sqrt{\epsilon}\rho^{B,2} \|_{L^\infty(I)}^2+C\|u^{I,0}_x\|_{L^\infty(I)}^2\|\sqrt{\epsilon}\rho^{B,2}_t\|_{L^2(I)}^2\\
&+C\|\sqrt{\epsilon}u^{B,3}_{yt}\|_{L^2(I)}^2\|\rho^{I,0} \|_{L^\infty(I)}^2+C\|\sqrt{\epsilon}u^{B,3}_y\|_{L^\infty(I)}^2\|\rho^{I,0}_t\|_{L^2(I)}^2\\
&+C\|\sqrt{\epsilon}u^{B,2}_t\|_{L^2(I)}^2\|\rho^{I,0}_x\|_{L^\infty(I)}^2+C\|\sqrt{\epsilon}u^{B,2}\|_{L^\infty(I)}^2\|\rho^{I,0}_{xt}\|_{L^2(I)}^2\\
&+C\|{\epsilon}u^{B,3}_t\|_{L^2(I)}^2\|\rho^{I,0}_x \|_{L^\infty(I)}^2+C\|{\epsilon}u^{B,3}\|_{L^\infty(I)}^2\|\rho^{I,0}_{xt}\|_{L^2(I)}^2\\
\le & {C}{\epsilon}\sqrt\epsilon \|\rho^{B,2}_{tt}\|_{L^2_y(I)}^2 +C\epsilon+C\epsilon \sqrt\epsilon +C\epsilon \sqrt\epsilon + C\epsilon + C\epsilon \sqrt\epsilon + C\epsilon + C\epsilon^2\sqrt\epsilon + C\epsilon^2.\\
%\le  & {C}{\epsilon}\sqrt\epsilon \|\rho^{B,2}_{tt}\|_{L^2_y(I)}^2 +C\epsilon.
\end{aligned}
$$
Hence, we get
$$
\|(h_3)_t\|_{L^2(I)}^2 \le {C}{\epsilon}\sqrt\epsilon \|\rho^{B,2}_{tt}\|_{L^2_y}^2  +C\epsilon .
$$
Similarly, we have
$$
\|(h_4)_t\|_{L^2(I)}^2 \le {C}{\epsilon}\sqrt\epsilon \|\rho^{b,2}_{tt}\|_{L^2_z}^2  +C\epsilon  .
$$

Redefine $(h_5)_t =h_{5,1}^t+ h_{5,2}^t+ h_{5,3}^t+h_{5,4}^t$ as
$$
\begin{aligned}
h_{5,1}^t=&[(u^{B,2}_{yt}+u^{b,2}_{zt})+\sqrt{\epsilon}(u^{B,3}_{yt}+u^{b,3}_{zt})][\sqrt{\epsilon}(\rho^{B,1}+\rho^{b,1}) +{\epsilon}(\rho^{B,2}+\rho^{b,2})+L^\epsilon ],\\
h_{5,2}^t=&[\sqrt{\epsilon}(u^{B,2}_t+u^{b,2}_t)+{\epsilon}(u^{B,3}_t+u^{b,3}_t)][(\rho^{B,1}_y+\rho^{b,1}_z) +\sqrt{\epsilon}(\rho^{B,2}_y+\rho^{b,2}_z)+L^\epsilon_x ],\\
h_{5,3}^t=&[(u^{B,2}_y+u^{b,2}_z)+\sqrt{\epsilon}(u^{B,3}_y+u^{b,3}_z)][\sqrt{\epsilon}(\rho^{B,1}_t+\rho^{b,1}_t) +{\epsilon}(\rho^{B,2}_t+\rho^{b,2}_t)+L^\epsilon_t ],\\
h_{5,4}^t=&[\sqrt{\epsilon}(u^{B,2}+u^{b,2})+{\epsilon}(u^{B,3}+u^{b,3})][(\rho^{B,1}_{yt}+\rho^{b,1}_{zt}) +\sqrt{\epsilon}(\rho^{B,2}_{yt}+\rho^{b,2}_{zt})+L^\epsilon_{xt} ].\\
\end{aligned}
$$
Then, using H$\rm{\ddot{o}}$lder's inequality, it holds that
$$\begin{aligned}
&\|(h_{5})_t\|_{L^2(I)}^2 \le C \sum_{i=1}^4\|h_{5,i}^t\|_{L^2(I)}^2\\
\le & C\|(u^{B,2}_{yt}+u^{b,2}_{zt})+\sqrt{\epsilon}(u^{B,3}_{yt}+u^{b,3}_{zt})\|_{L^2(I)}^2\|\sqrt{\epsilon}(\rho^{B,1}+\rho^{b,1}) +{\epsilon}(\rho^{B,2}+\rho^{b,2})+L^\epsilon\|_{L^\infty(I)}^2\\
 &+ C\|\sqrt{\epsilon}(u^{B,2}_t+u^{b,2}_t)+{\epsilon}(u^{B,3}_t+u^{b,3}_t)\|_{L^2(I)}^2\|(\rho^{B,1}_y+\rho^{b,1}_z) +\sqrt{\epsilon}(\rho^{B,2}_y+\rho^{b,2}_z)+L^\epsilon_x\|_{L^\infty(I)}^2\\
&+C\|(u^{B,2}_y+u^{b,2}_z)+\sqrt{\epsilon}(u^{B,3}_y+u^{b,3}_z)\|_{L^\infty(I)}^2\|\sqrt{\epsilon}(\rho^{B,1}_t+\rho^{b,1}_t) +{\epsilon}(\rho^{B,2}_t+\rho^{b,2}_t)+L^\epsilon_t\|_{L^2(I)}^2\\
&+C\|\sqrt{\epsilon}(u^{B,2}+u^{b,2})+{\epsilon}(u^{B,3}+u^{b,3})\|_{L^\infty(I)}^2\|(\rho^{B,1}_{yt}+\rho^{b,1}_{zt}) +\sqrt{\epsilon}(\rho^{B,2}_{yt}+\rho^{b,2}_{zt})+L^\epsilon_{xt}\|_{L^2(I)}^2\\
 \le &C \epsilon\sqrt\epsilon+C \epsilon\sqrt\epsilon+C \epsilon\sqrt\epsilon+C \epsilon\sqrt\epsilon\le C \epsilon\sqrt\epsilon.
\end{aligned}
$$
\end{proof}

\begin{Lemma}\label{Phi-U-H^1-1} Under the conditions of Theorem \ref{assumption}, there holds that
\begin{equation}\label{4-5}
\begin{aligned}
&\frac{1}{2}\frac{\mathrm{d}}{\mathrm{d}t} \int_I \rho^\epsilon (U^\epsilon_x)^2 \mathrm{d}x +\int_I (U^\epsilon_{xx})^2 \mathrm{d}x+\int_I \epsilon\rho^\epsilon \Phi^\epsilon_{xxx} U^\epsilon_{xx}\mathrm{d}x+\int_I \rho^{\epsilon}_x U^\epsilon_{x}\epsilon \Phi^\epsilon_{xxx}\mathrm{d}x\\
\le &\frac{3}{2(12M +4)}\int_I ( U^\epsilon_{xx})^2\mathrm{d}x +C\|W\|_{L^2(I)}^2+C(1+ \|u^\epsilon_x\|_{L^2(I)}^2)\|U^\epsilon_x\|_{L^2(I)}^2\\
&+C\|U^\epsilon_x\|_{L^2(I)}^2[\frac{1}{\sqrt \epsilon}+{\epsilon} \|\Phi^\epsilon_{xx}\|_{L^2(I)}^2],\\
\end{aligned}
\end{equation}
where
%\begin{equation}\label{m-for-what}
%m=12M +4
%\end{equation}
% with
 ${M}$ is defined in Proposition \ref{exi-ns}.
\end{Lemma}
\begin{proof}
Firstly, multiplying (\ref{zhankai-equ})$_2$ by $U^\epsilon_{xx}$ and integrating the result over $I$, there holds
$$
\begin{aligned}
& \int_I (-\rho^{\epsilon} U^\epsilon_t-  \rho^\epsilon u^\epsilon U^\epsilon_x )U^\epsilon_{xx} \mathrm{d}x+ \int_I (U^\epsilon_{xx})^2 \mathrm{d}x
+\int_I \epsilon\rho^\epsilon \Phi^\epsilon_{xxx} U^\epsilon_{xx}\mathrm{d}x= -\int_I W U^\epsilon_{xx}\mathrm{d}x.\\
\end{aligned}
$$

From
$$\begin{aligned}
&\int_I (-\rho^{\epsilon} U^\epsilon_t-  \rho^\epsilon u^\epsilon U^\epsilon_x )U^\epsilon_{xx} \mathrm{d}x= \int_I (\rho^{\epsilon}_x U^\epsilon_t+\rho^{\epsilon} U^\epsilon_{xt})U^\epsilon_{x}\mathrm{d}x-  \frac{1}{2}\int_I\rho^\epsilon u^\epsilon \partial_x (U^\epsilon_x )^2 \mathrm{d}x\\
=&  \int_I \rho^{\epsilon}_x U^\epsilon_t U^\epsilon_{x}\mathrm{d}x + \frac{1}{2}\frac{\mathrm{d}}{\mathrm{d}t} \int_I \rho^\epsilon(U^\epsilon_{x})^2- \frac{1}{2} \int_I \rho^\epsilon_t(U^\epsilon_{x})^2 \mathrm{d}x+ \frac{1}{2}\int_I(\rho^\epsilon u^\epsilon)_x  (U^\epsilon_x )^2 \mathrm{d}x\\
=& \int_I \rho^{\epsilon}_x U^\epsilon_t U^\epsilon_{x}\mathrm{d}x + \frac{1}{2}\frac{\mathrm{d}}{\mathrm{d}t} \int_I \rho^\epsilon(U^\epsilon_{x})^2+\int_I(\rho^\epsilon u^\epsilon)_x  (U^\epsilon_x )^2 \mathrm{d}x\\
=  &\frac{1}{2}\frac{\mathrm{d}}{\mathrm{d}t} \int_I \rho^\epsilon(U^\epsilon_{x})^2+\int_I(\rho^\epsilon u^\epsilon)_x  (U^\epsilon_x )^2 \mathrm{d}x +\int_I \rho^{\epsilon}_x U^\epsilon_{x}(-u^\epsilon U^\epsilon_x+\frac{U^\epsilon_{xx}}{\rho^{\epsilon}}+\frac{W}{\rho^\epsilon}+\epsilon \Phi^\epsilon_{xxx})\mathrm{d}x,
\end{aligned}
$$
 we get
$$\begin{aligned}
&\frac{1}{2}\frac{\mathrm{d}}{\mathrm{d}t} \int_I \rho^\epsilon (U^\epsilon_x)^2 \mathrm{d}x +\int_I (U^\epsilon_{xx})^2 \mathrm{d}x+\int_I \epsilon\rho^\epsilon \Phi^\epsilon_{xxx} U^\epsilon_{xx}\mathrm{d}x+\int_I \rho^{\epsilon}_x U^\epsilon_{x}\epsilon \Phi^\epsilon_{xxx}\mathrm{d}x\\
=&- \int_I W U^\epsilon_{xx}\mathrm{d}x-\int_I(\rho^\epsilon u^\epsilon)_x  (U^\epsilon_x )^2 \mathrm{d}x -\int_I \rho^{\epsilon}_x U^\epsilon_{x}(-u^\epsilon U^\epsilon_x+\frac{U^\epsilon_{xx}}{\rho^{\epsilon}}+\frac{W}{\rho^\epsilon})\mathrm{d}x\\
=&- \int_I W U^\epsilon_{xx}\mathrm{d}x -\int_I\rho^\epsilon u^\epsilon_x  U^\epsilon_x  U^\epsilon_x\mathrm{d}x - \int_I \rho^{\epsilon}_x U^\epsilon_{x}\frac{U^\epsilon_{xx}}{\rho^{\epsilon}}\mathrm{d}x - \int_I \rho^{\epsilon}_x U^\epsilon_{x}\frac{W}{\rho^{\epsilon}}\mathrm{d}x \\
:= & H_1 +H_2 +H_3 +H_4.
\end{aligned}
$$

Now we estimate $H_1$, $H_2$, $H_3$ and $H_4$. Using Lemma \ref{clm-rho-bdd}, (\ref{some-rho-1}), H$\rm{\ddot{o}}$lder's inequlity, Young's inequality, Sobolev's inequality $\|U^\epsilon_x\|_{L^\infty(I)} \le C \|U^\epsilon_x\|_{H^1(I)}$, and Poincar\'e's inequality $\|\rho^{\epsilon}_x\|_{L^\infty(I)} \le C\|\rho^{\epsilon}_{xx}\|_{L^2(I)}$, there holds
$$\begin{aligned}
H_1 \le &\frac{1}{2(12M +4)}\| U^\epsilon_{xx}\|_{L^2(I)}^2 +C\|W\|_{L^2(I)}^2,\\
H_2 \le &C \|u^\epsilon_x\|_{L^2(I)}^2\|U^\epsilon_x\|_{L^2(I)}^2 + (C \|U^\epsilon_x\|_{L^2(I)}^2+\frac{1}{2(12M +4)}\|U^\epsilon_{xx}\|_{L^2(I)}^2),\\
H_3 \le & (\frac{C}{\sqrt\epsilon} + C\epsilon \|\Phi^\epsilon_{xx}\|_{ L^2(I)}^2)\|U^\epsilon_x\|_{L^2(I)}^2+\frac{1}{2(12M +4)}\| U^\epsilon_{xx}\|_{L^2(I)}^2,\\
H_4 \le & C(\frac{1}{\sqrt\epsilon} + \epsilon \|\Phi^\epsilon_{xx}\|_{ L^2(I)}^2)\|U^\epsilon_x\|_{L^2(I)}^2 +C\|W\|_{L^2(I)}^2.
\end{aligned}
$$

Then, we get (\ref{4-5}).
\end{proof}

\begin{Lemma}\label{Phi-U-H^1-2} Under the conditions of Theorem \ref{assumption}, there holds that
\begin{equation}\label{4-6}
\begin{aligned}
&\frac{1}{2}\frac{\mathrm{d}}{\mathrm{d}t} \int_I \epsilon(\Phi^\epsilon_{xx})^2 \mathrm{d}x -\int_I \epsilon \rho^\epsilon\Phi^\epsilon_{xxx} U^\epsilon_{xx}\mathrm{d}x- \int_I \epsilon\rho^{\epsilon}_x U^\epsilon_{x}  \Phi^\epsilon_{xxx}\mathrm{d}x\\
\le &\frac{3}{2(12M +4)} \epsilon^2 \|\sqrt {\rho^\epsilon}\Phi^\epsilon_{xxx}\|_{L^2(I)}^2 +C(\frac{1}{\sqrt\epsilon} + \epsilon \| \Phi^\epsilon_{xx}\|_{L^2(I)}^2)\|U^\epsilon_x\|_{L^2(I)}^2  + C\sqrt\epsilon \\
&+C \|\Phi^\epsilon_{x}\|_{L^2(I)}^2+ C \epsilon \|\Phi^\epsilon_{xx}\|_{L^2(I)}^2,
\end{aligned}
\end{equation}
where
 ${M}$ is defined in Proposition \ref{exi-ns}.
\end{Lemma}
\begin{proof}
Differentiating (\ref{zhankai-equ})$_1$ with respect to $x$, we have
\begin{equation}\label{Phixt-1}
\begin{aligned}
&\Phi_{xt}^\epsilon+\rho^\epsilon_{xx} U^\epsilon+2\rho^\epsilon_x U^\epsilon_x+\rho^\epsilon U^\epsilon_{xx}=-\sum_{i=1}^5 (h_i)_x.\\
\end{aligned}
\end{equation}
 Then, multiplying (\ref{Phixt-1}) by $ \epsilon\Phi^\epsilon_{xxx}$ and integrating the result over $I$, we get
\begin{equation}\label{G-for-five}
\begin{aligned}
&-\int_I \epsilon \rho^\epsilon\Phi^\epsilon_{xxx} U^\epsilon_{xx}\mathrm{d}x- \int_I \epsilon\rho^{\epsilon}_x U^\epsilon_{x}  \Phi^\epsilon_{xxx}\mathrm{d}x\\
:=& \int_I \epsilon \Phi_{xt}^\epsilon  \Phi^\epsilon_{xxx}\mathrm{d}x + \int_I \epsilon\rho^{\epsilon}_x U^\epsilon_{x}  \Phi^\epsilon_{xxx}\mathrm{d}x +\int_I \epsilon\rho^{\epsilon}_{xx} U^\epsilon \Phi^\epsilon_{xxx}\mathrm{d}x\\
&+\int_I \epsilon(\sum_{i=1}^5 (h_i)_x-u^{I,0}\Phi^\epsilon_{xx}-2 u^{I,0}_x\Phi^\epsilon_{x}) \Phi^\epsilon_{xxx}\mathrm{d}x\\
&+ \int_I  \epsilon  (u^{I,0}\Phi^\epsilon_{xx}+2 u^{I,0}_x\Phi^\epsilon_{x})\Phi^\epsilon_{xxx}\mathrm{d}x:=\sum_{i=1}^5 G_i.\\
\end{aligned}
\end{equation}

We estimate $\sum_{i=1}^5 G_i$ terms by terms.
$$
\begin{aligned}
G_1=\int_I \epsilon \Phi_{xt}^\epsilon  \Phi^\epsilon_{xxx}\mathrm{d}x=- \int_I \epsilon \Phi_{xxt}^\epsilon  \Phi^\epsilon_{xx}\mathrm{d}x =-\frac{\epsilon}{2}\frac{\mathrm{d}}{\mathrm{d}t} \int_I (\Phi^\epsilon_{xx})^2 \mathrm{d}x.
\end{aligned}
$$
The estimate of $G_2,~G_3,~G_4$ and $G_5$ will use H$\rm{\ddot{o}}$lder's inequality and  Young's inequality.

Using Poincar\'e's inequality and (\ref{some-rho-1})$_2$, we get

$$
\begin{aligned}
&G_2+G_3= \int_I \epsilon\rho^{\epsilon}_x U^\epsilon_{x}  \Phi^\epsilon_{xxx}\mathrm{d}x+\int_I \epsilon\rho^{\epsilon}_{xx} U^\epsilon \Phi^\epsilon_{xxx}\mathrm{d}x\\
\le & C\epsilon \|\rho^\epsilon_x\|_{L^\infty(I)} \|U^\epsilon_x\|_{L^2(I)} \|\Phi^\epsilon_{xxx}\|_{L^2(I)} + C\epsilon \|\rho^\epsilon_{xx}\|_{L^2(I)} \|U^\epsilon\|_{L^\infty(I)} \|\Phi^\epsilon_{xxx}\|_{L^2(I)} \\
\le &  C\epsilon \|\rho^\epsilon_{xx}\|_{L^2(I)} \|U^\epsilon_x\|_{L^2(I)} \|\Phi^\epsilon_{xxx}\|_{L^2(I)}\\
\le & \frac{1}{2(12M +4)} \epsilon^2 \|\sqrt {\rho^\epsilon}\Phi^\epsilon_{xxx}\|_{L^2(I)}^2 + (\frac{C}{\sqrt\epsilon} + C\epsilon \| \Phi^\epsilon_{xx}\|_{L^2(I)}^2) \|U^\epsilon_x\|_{L^2(I)}^2.
\end{aligned}
$$

Using Lemmas \ref{clm:h} and \ref{clm:no-impor}, and Corollary \ref{clm:some-impor},    we get
$$
\begin{aligned}
G_4=&\int_I \epsilon(\sum_{i=1}^5 (h_i)_x-u^{I,0}\Phi^\epsilon_{xx}-2 u^{I,0}_x\Phi^\epsilon_{x}) \Phi^\epsilon_{xxx}\mathrm{d}x\\
\le & \frac{1}{2(12M +4)} \epsilon^2 \|\sqrt {\rho^\epsilon}\Phi^\epsilon_{xxx}\|_{L^2(I)}^2 + C\sum_{i=1,i\neq 2}^5\| (h_i)_x\|_{L^2(I)}^2 + C \| (h_2)_x-u^{I,0}\Phi^\epsilon_{xx}-2 u^{I,0}_x\Phi^\epsilon_{x}\|_{L^2(I)}^2 \\
\le & \frac{1}{2(12M +4)} \epsilon^2 \|\sqrt {\rho^\epsilon}\Phi^\epsilon_{xxx}\|_{L^2(I)}^2+ C\sqrt\epsilon + C \epsilon^2 \|\Phi^\epsilon_{xx}\|_{L^2(I)}^2.
\end{aligned}
$$

Integrating $G_5$ by parts, we get
$$
\begin{aligned}
G_5=&  \int_I  \epsilon  (u^{I,0}\Phi^\epsilon_{xx}+2 u^{I,0}_x\Phi^\epsilon_{x})\Phi^\epsilon_{xxx}\mathrm{d}x=2 \int_I  \epsilon  u^{I,0}_x\Phi^\epsilon_{x}\Phi^\epsilon_{xxx}\mathrm{d}x-\frac{1}{2}\int_I  \epsilon  u^{I,0}_x(\Phi^\epsilon_{xx})^2\mathrm{d}x\\
\le &  \frac{1}{2(12M +4)} \epsilon^2 \|\sqrt {\rho^\epsilon}\Phi^\epsilon_{xxx}\|_{L^2(I)}^2+C  \|u_x^{I,0}\|_{L^\infty(I)}^2 \|\Phi^\epsilon_{x}\|_{L^2(I)}^2+C \epsilon \|u_x^{I,0}\|_{L^\infty(I)} \|\Phi^\epsilon_{xx}\|_{L^2(I)}^2\\
\le &  \frac{1}{2(12M +4)} \epsilon^2 \|\sqrt {\rho^\epsilon}\Phi^\epsilon_{xxx}\|_{L^2(I)}^2+C \|\Phi^\epsilon_{x}\|_{L^2(I)}^2+C \epsilon  \|\Phi^\epsilon_{xx}\|_{L^2(I)}^2.\\
\end{aligned}
$$

In conclusion, substituting the estimates of $G_{i}$ for $i=1,2,3,4,5$ into (\ref{G-for-five}), we have (\ref{4-6}).
\end{proof}

It is necessary to estimate $ \epsilon^2\|\sqrt{\rho^{\epsilon}}\Phi_{xxx}^\epsilon\|_{ L^2(I)}^2$ because of $G_2$ and $G_3$.

\begin{Lemma}\label{Phi-U-H^1-3}Under the conditions of Theorem \ref{assumption}, there holds that
\begin{equation}\label{4-7}
\begin{aligned}
&\int_I\epsilon U^\epsilon_{xx} \Phi^\epsilon_{xxx}\mathrm{d}x+\int_I\epsilon^2 \rho^\epsilon(\Phi^\epsilon_{xxx})^2\mathrm{d}x {-\epsilon\frac{{\mathrm{d}}}{\mathrm{d}t} \int_I\Phi^\epsilon_{xx}\Phi^\epsilon_{t}\mathrm{d}x}\\
\le &\frac{1}{12M +4}\int_I\epsilon^2 \rho^\epsilon(\Phi^\epsilon_{xxx})^2\mathrm{d}x+C\epsilon  \|U^\epsilon_{xx}\|_{L^2(I)}^2 + C \epsilon + C \|W\|_{L^2(I)}^2\\
&+ C(1+\epsilon\|U^\epsilon_x\|_{L^2(I)}^2) \epsilon\|\Phi^\epsilon_{xx}\|_{L^2(I)}^2 +C (\| u^\epsilon_x\|_{L^2(I)}^2 + 1+\|U^\epsilon_x\|_{L^2(I)}^2) \|U^\epsilon_x \|_{L^2(I)}^2+{C} \|L^\epsilon_{tt}\|_{L^2(I)}^2\\
&
+ {C}{\epsilon}^2\sqrt\epsilon (\|\rho^{B,2}_{tt}\|_{L^2_y}^2+\|\rho^{b,2}_{tt}\|_{L^2_z}^2),  \\
\end{aligned}
\end{equation}
where
 ${M}$ is defined in Proposition \ref{exi-ns}.
\end{Lemma}
\begin{proof}

Multiplying (\ref{zhankai-equ})$_2$ by $\epsilon \Phi^\epsilon_{xxx}$ and integrating the result over $I$ with respect to $x$, we get
\begin{equation}\label{zhankai-equ-s}
	\begin{aligned}
&\int_I\epsilon U^\epsilon_{xx} \Phi^\epsilon_{xxx}\mathrm{d}x+\int_I\epsilon^2 \rho^\epsilon(\Phi^\epsilon_{xxx})^2\mathrm{d}x\\
=&\int_I\epsilon (\rho^{\epsilon} U^\epsilon)_t\Phi^\epsilon_{xxx}\mathrm{d}x+ \int_I\epsilon (\rho^\epsilon u^\epsilon U^\epsilon)_x\Phi^\epsilon_{xxx}\mathrm{d}x-\int_I\epsilon W \Phi^\epsilon_{xxx}\mathrm{d}x:=K_1+K_2+K_3.\\
\end{aligned}
\end{equation}

Now we estimate $K_1$, $K_2$ and $K_3$. Differentiating (\ref{zhankai-equ})$_1$ with respect to $t$, it holds that
\begin{equation}\label{wucha-1-t}
	\begin{aligned}
&\Phi_{tt}^\epsilon+(\rho^\epsilon U^\epsilon)_{xt}+\sum_{i=1}^5 (h_i)_t=0.
\end{aligned}
\end{equation}

From (\ref{wucha-1-t}), integrating $K_1$ by parts, we get
$$
\begin{aligned}
K_1=&\int_I\epsilon (\rho^{\epsilon} U^\epsilon)_t\Phi^\epsilon_{xxx}\mathrm{d}x=-\int_I\epsilon (\rho^{\epsilon} U^\epsilon)_{xt}\Phi^\epsilon_{xx}\mathrm{d}x
=\int_I\epsilon\Phi_{tt}^\epsilon\Phi^\epsilon_{xx}\mathrm{d}x+\int_I\epsilon [\sum_{i=1}^5 (h_i)_t]\Phi^\epsilon_{xx}\mathrm{d}x\\
= &{\epsilon\frac{{\mathrm{d}}}{\mathrm{d}t} \int_I\Phi^\epsilon_{xx} \Phi^\epsilon_{t}\mathrm{d}x}+\epsilon \int_I(\Phi^\epsilon_{xt})^2\mathrm{d}x + \int_I\epsilon [\sum_{i=1}^5 (h_i)_t]\Phi^\epsilon_{xx}\mathrm{d}x.
\end{aligned}
$$
Using (\ref{h-alpha-bata}), (\ref{some-h-2}), and Young's inequality, we have
$$
\begin{aligned}
&\epsilon \int_I(\Phi^\epsilon_{xt})^2\mathrm{d}x + \int_I\epsilon [\sum_{i=1}^5 (h_i)_t]\Phi^\epsilon_{xx}\mathrm{d}x\\
\le & \epsilon \|\Phi^\epsilon_{xt}\|_{L^2(I)}^2 + C\epsilon\sum_{i=1}^5\| (h_i)_t\|_{L^2(I)}^2 + C \epsilon \|\Phi^\epsilon_{xx}\|_{L^2(I)}^2\\
\le & (C +1) \epsilon \|\Phi^\epsilon_{xt}\|_{L^2(I)}^2 + C \epsilon \|\Phi^\epsilon_{xx}\|_{L^2(I)}^2+{C} \|L^\epsilon_{tt}\|_{L^2(I)}^2 + C \epsilon  + C\epsilon \|\Phi^\epsilon_{t}\|_{L^2(I)}^2 \\
&+ {C}{\epsilon}^2\sqrt\epsilon (\|\rho^{B,2}_{tt}\|_{L^2_y}^2+\|\rho^{b,2}_{tt}\|_{L^2_z}^2) \\
\le & C \epsilon   \|U^\epsilon_{xx}\|_{L^2(I)}^2 + C \epsilon + C(1+\epsilon\|U^\epsilon_x\|_{L^2(I)}^2) \epsilon\|\Phi^\epsilon_{xx}\|_{L^2(I)}^2 \\
&+ C{\sqrt\epsilon}\|\sqrt{\rho^\epsilon}U^\epsilon_x\|_{L^2(I)}^2+{C} \|L^\epsilon_{tt}\|_{L^2(I)}^2  + {C}{\epsilon}^2\sqrt\epsilon (\|\rho^{B,2}_{tt}\|_{L^2_y(I)}^2+\|\rho^{b,2}_{tt}\|_{L^2_z(I)}^2).  \\
\end{aligned}
$$

For $K_2$, using (\ref{equ-epsilon})$_1$, (\ref{some-h-2}), and Young's inequality, we have
$$
\begin{aligned}
K_{2}=& \int_I\epsilon (\rho^\epsilon u^\epsilon U^\epsilon)_x\Phi^\epsilon_{xxx}\mathrm{d}x\\
\le &\frac{1}{2(12M +4)}\int_I\epsilon^2 \rho^\epsilon(\Phi^\epsilon_{xxx})^2\mathrm{d}x +C\|(\rho^\epsilon)^{-1}\|_{L^\infty(I)}\|  \rho^\epsilon u^\epsilon U^\epsilon_x -\rho^\epsilon_t U^\epsilon\|_{L^2(I)}^2\\
\le &\frac{1}{2(12M +4)}\int_I\epsilon^2 \rho^\epsilon(\Phi^\epsilon_{xxx})^2\mathrm{d}x +C(\| u^\epsilon_x\|_{L^2(I)}^2+\| \rho^\epsilon_t -\sqrt\epsilon\Phi^\epsilon_t\|_{L^2(I)}^2 +\epsilon\|\Phi^\epsilon_t\|_{L^2(I)}^2) \|U^\epsilon_x \|_{L^2(I)}^2 \\
\le &\frac{1}{2(12M +4)}\int_I\epsilon^2 \rho^\epsilon(\Phi^\epsilon_{xxx})^2\mathrm{d}x +C (\| u^\epsilon_x\|_{L^2(I)}^2 + 1+\|U^\epsilon_x\|_{L^2(I)}^2) \|U^\epsilon_x \|_{L^2(I)}^2,
\end{aligned}
$$
and
$$
\begin{aligned}
K_{3}=& -\int_I\epsilon W \Phi^\epsilon_{xxx}\mathrm{d}x
 \le   \frac{1}{2(12M +4)}\int_I\epsilon^2 \rho^\epsilon(\Phi^\epsilon_{xxx})^2\mathrm{d}x+ C \|W\|_{L^2(I)}^2.
\end{aligned}
$$

Hence, putting estimates of $K_1,K_2$ and $K_3$ into (\ref{zhankai-equ-s}), we get (\ref{4-7}).
\end{proof}

\begin{Lemma}\label{Phi-U-H^1-4} Under the conditions of Theorem \ref{assumption}, there holds that
\begin{equation}\label{4-8}
\begin{aligned}
&\frac{1}{2}\frac{{\mathrm{d}}}{\mathrm{d}t} \int_I(\Phi^\epsilon_{x})^2\mathrm{d}x\\
\le &\frac{1}{2(12M+4)}\int_I ( U^\epsilon_{xx})^2\mathrm{d}x+C \|\Phi^\epsilon_{x}\|_{L^2(I)}^2 + (\frac{C}{\sqrt\epsilon} + C\epsilon \| \Phi^\epsilon_{xx}\|_{L^2(I)}^2) \|U^\epsilon_x\|_{L^2(I)}^2\\
& + C\sqrt\epsilon + C \epsilon^2 \|\Phi^\epsilon_{xx}\|_{L^2(I)}^2,
\end{aligned}
\end{equation}
where
 ${M}$ is defined in Proposition \ref{exi-ns}.
\end{Lemma}
\begin{proof}
Multipying (\ref{Phixt-1}) by $\Phi_x^\epsilon$ and integrating the result over $I$, it holds that
\begin{equation}\label{sigma-for}
\begin{aligned}
&\frac{1}{2}\frac{{\mathrm{d}}}{\mathrm{d}t} \int_I(\Phi^\epsilon_{x})^2\mathrm{d}x\\
=&-\int_I(\rho^\epsilon U^\epsilon)_{xx}\Phi^\epsilon_{x}\mathrm{d}x-\int_I(\sum_{i=1}^5 (h_i)_x-u^{I,0}\Phi^\epsilon_{xx}-2 u^{I,0}_x\Phi^\epsilon_{x}) \Phi^\epsilon_{x}\mathrm{d}x\\
& -\int_I(u^{I,0}\Phi^\epsilon_{xx}+2 u^{I,0}_x\Phi^\epsilon_{x}) \Phi^\epsilon_{x}\mathrm{d}x\\
:=&\Gamma_{1}+\Gamma_{2}+\Gamma_{3}.
\end{aligned}
\end{equation}

Similar to the estimates of $G_2+ G_3$ and $G_4$ in the proof of Lemma \ref{Phi-U-H^1-2}, we have
$$
\begin{aligned}
&\Gamma_1= -\int_I (2\rho^{\epsilon}_x U^\epsilon_{x}+\rho^{\epsilon}_{xx} U^\epsilon)  \Phi^\epsilon_{x}\mathrm{d}x-\int_I \rho^{\epsilon}U^\epsilon_{xx}  \Phi^\epsilon_{x}\mathrm{d}x\\
\le & \frac{1}{2(12M+4)}\int_I ( U^\epsilon_{xx})^2\mathrm{d}x+C \|\Phi^\epsilon_{x}\|_{L^2(I)}^2 + (\frac{C}{\sqrt\epsilon} + C\epsilon \| \Phi^\epsilon_{xx}\|_{L^2(I)}^2) \|U^\epsilon_x\|_{L^2(I)}^2,
\end{aligned}
$$
and
$$
\begin{aligned}
\Gamma_2=&-\int_I(\sum_{i=1}^5 (h_i)_x-u^{I,0}\Phi^\epsilon_{xx}-2 u^{I,0}_x\Phi^\epsilon_{x}) \Phi^\epsilon_{x}\mathrm{d}x
\le  C \|\Phi^\epsilon_{x}\|_{L^2(I)}^2 + C\sqrt\epsilon + C \epsilon^2 \|\Phi^\epsilon_{xx}\|_{L^2(I)}^2.
\end{aligned}
$$

Integrating $\Gamma_3$ by parts, we have
$$
\begin{aligned}
&\Gamma_3=- \int_I(u^{I,0}\Phi^\epsilon_{xx}+2 u^{I,0}_x\Phi^\epsilon_{x}) \Phi^\epsilon_{x}\mathrm{d}x=-\frac{3}{2} \int_I u^{I,0}_x(\Phi^\epsilon_{x})^2\mathrm{d}x \le C \|\Phi^\epsilon_{x}\|_{L^2(I)}^2.
\end{aligned}
$$

In conclusion, putting the estimates of $\Gamma_1$, $\Gamma_2$ and $\Gamma_3$ into (\ref{sigma-for}), we get (\ref{4-8}).
\end{proof}

\begin{Corollary}\label{Phi-U-H^1-all}Under the conditions of Theorem \ref{assumption}, there exists a positive constant $\epsilon_0$ satisfying
\begin{equation}\label{epsilon2}
\begin{aligned}
 &\epsilon_0\le  \min\{(\frac{1}{2}-\frac{1}{{3M}+1})\frac{M}{C },\frac{{2M}-1}{C},  \frac{1}{4M^2(C+1)^2}\},
\end{aligned}
\end{equation}
such that if $\epsilon\le\epsilon_0$, there holds that
\begin{equation}\label{Uxx-Phix-L2L22-all}
\sup_{t\in[0,T_1]}(\|U^\epsilon_x\|_{L^2(I)}^2+\|\Phi^\epsilon_x\|_{L^2(I)}^2+\epsilon\|\Phi^\epsilon_{xx}\|_{L^2(I)}^2)+\|U_{xx}^\epsilon\|_{L_{T_1}^2 L^2}^2+\epsilon^2\|\Phi^\epsilon_{xxx}\|_{L_{T_1}^2 L^2}^2\le  {B_2(T)}{\sqrt{\epsilon}},
\end{equation}
where $B_2(s)= C(1+s)+Cs(1+s) +Cs(1+s) (B_1(s)+1) \exp\{C (B_1 (s)+1)\}$ for any $s\in[0,T]$.
\end{Corollary}
\begin{proof}
We will prove Corollary \ref{Phi-U-H^1-all} by using Propositions  \ref{exi-ns},  \ref{exi-nsk} and \ref{B-b-Pro}.

Using (\ref{equ-epsilon}), (\ref{equ-NS}), (\ref{IBb-3})-(\ref{initial-value}), and definition of $L^\epsilon$ and $M^\epsilon$, we get the estimate of initial value of system (\ref{zhankai-equ}) as follows:
\begin{equation}\begin{aligned}\label{ini-val-phi-111}
&\|\Phi^\epsilon_x(\cdot,0) \|_{L^2(I)}^2+\|U^\epsilon_x(\cdot,0)\|_{L^2(I)}^2 +\epsilon\|\Phi^\epsilon_{xx}(\cdot,0) \|_{L^2(I)}^2\\
=&\frac{1}{\epsilon}\|\partial_x(\rho^{\epsilon}_0-\rho^{I,0}_0)\|_{L^2(I)}^2 +\frac{1}{\epsilon}\|\partial_x(u^\epsilon_0 -u^{I,0}_0)\|_{L^2(I)}^2+\|\partial_x^2(\rho^{\epsilon}_0-\rho^{I,0}_0)\|_{L^2(I)}^2\le C\epsilon^{m-1},
\end{aligned}
 \end{equation} for any $m> 2$.
In addition, we have
\begin{equation}\label{2}
\begin{aligned}
&\|\Phi^\epsilon_t(\cdot,0)\|_{L^2(I)}^2 = \|[-(\rho^\epsilon U^\epsilon)_x-\sum_{i=1}^5 h_i ](\cdot,0)\|_{L^2(I)}^2 \le C \sqrt\epsilon.
\end{aligned}
\end{equation}
(\ref{ini-val-phi-111}) and (\ref{2}) yields
\begin{equation}
\begin{aligned}
&-\epsilon \int_I\Phi^\epsilon_{xx}(x,0)\Phi^\epsilon_{t}(x,0)\mathrm{d}x \le \epsilon \|\Phi^\epsilon_{xx}(\cdot,0)\|_{L^2(I)}  \|\Phi^\epsilon_t(\cdot,0)\|_{L^2(I)} \le C\epsilon.
\end{aligned}
\end{equation}

From Lemma \ref{clm-rho-bdd}, it holds that $$ \|\rho^\epsilon, (\rho^\epsilon)^{-1}\|_{C([0,T_1;L^\infty(I)])} \le 2M.$$

We firstly add (\ref{4-5}), (\ref{4-6}) and (\ref{4-8}) together, multiplying the result by $2M$, then add the resulting inequality to inequality (\ref{4-7}), and finally integrating this combined result with respect to $t$ over the interval $[0,t]$ for any $0<t\le T_1$, we get
$$\begin{aligned}
&\frac{2M}{2} \int_I [\rho^\epsilon (U^\epsilon_x)^2+\epsilon(\Phi^\epsilon_{xx})^2] \mathrm{d}x +2M\int_0^{t}\|U^\epsilon_{xx}\|_{L^2(I)}^2 \mathrm{d}t+\frac{2M}{2}\int_I(\Phi^\epsilon_{x})^2\mathrm{d}x\\
&+\int_0^{t} \int_I\epsilon U^\epsilon_{xx} \Phi^\epsilon_{xxx}\mathrm{d}x \mathrm{d}+\int_0^{t}\epsilon^2 \|\sqrt {\rho^\epsilon}\Phi^\epsilon_{xxx}\|_{L^2(I)}^2\mathrm{d} t{-\epsilon \int_I\Phi^\epsilon_{xx}\Phi^\epsilon_{t}\mathrm{d}x}\\
\le &[ \frac{2M}{6M+2}+C \epsilon ]  \int_0^{t} \|U^\epsilon_{xx}\|_{L^2(I)}^2 \mathrm{d}t+ \frac{1}{4} \int_0^{t}\epsilon^2 \|\sqrt {\rho^\epsilon}\Phi^\epsilon_{xxx}\|_{L^2(I)}^2 \mathrm{d}t\\
& + C(1+T_1)\sqrt\epsilon +C\int_0^{t}(\|\Phi^\epsilon_{x}\|_{L^2(I)}^2 +\epsilon \|\Phi^\epsilon_{xx}\|_{L^2(I)}^2)\mathrm{d}t\\
& + C\int_0^{t} (\| u^\epsilon_x\|_{L^2(I)}^2 + 1+\|U^\epsilon_x\|_{L^2(I)}^2 +\epsilon \|\Phi^\epsilon_{xx}\|_{L^2(I)}^2) \|U^\epsilon_x \|_{L^2(I)}^2\mathrm{d}t.
\end{aligned}
$$

Choosing $\epsilon_0$ small enough such that $\epsilon\le \epsilon_0$ and that
$$\begin{aligned}
 &  \frac{{2M}}{6M+2}+ C\epsilon \le \frac{{2M}}{6M+2}+ C\epsilon_0   \le \frac{{2M}}{4},
\end{aligned}
$$
i.e.,
\begin{equation}\label{ep0-2}
\begin{aligned}
 \epsilon_0  \le (\frac{1}{2}-\frac{1}{{3M}+1})\frac{{M}}{C },
\end{aligned}
\end{equation}
and using
$$\begin{aligned}
&- \epsilon\int_I\Phi^\epsilon_{xx} \Phi^\epsilon_{t}\mathrm{d}x\ge -\frac{1}{2}\int_I\epsilon(\Phi^\epsilon_{xx})^2 \mathrm{d}x- \frac{\epsilon}{2}\|\Phi^\epsilon_t\|_{L^\infty_{T_1}L^2(I)}^2\\
& \ge  -\frac{1}{2}\int_I\epsilon(\Phi^\epsilon_{xx})^2 \mathrm{d}x- \frac{\epsilon}{2}(C+ C\|\sqrt{\rho^\epsilon}U^\epsilon_x\|_{L^2(I)}^2)\\
\end{aligned}
$$
and (\ref{some-h-2})$_5$, we have
\begin{equation}\label{3}
\begin{aligned}
& \int_I [\frac{{2M}-C\epsilon}{2}\rho^\epsilon (U^\epsilon_x)^2+\frac{{2M}-1}{2}\epsilon(\Phi^\epsilon_{xx})^2] \mathrm{d}x +\frac{3}{4}\int_0^{t} \int_I{2M}(U^\epsilon_{xx})^2 \mathrm{d}x \mathrm{d}t+\frac{{2M}}{2}\int_I(\Phi^\epsilon_{x})^2\mathrm{d}x\\
&+\int_0^{t} \int_I\epsilon U^\epsilon_{xx} \Phi^\epsilon_{xxx}\mathrm{d}x \mathrm{d}+\frac{3}{4}\int_0^{t}\epsilon^2 \|\sqrt {\rho^\epsilon}\Phi^\epsilon_{xxx}\|_{L^2(I)}^2\mathrm{d} t\\
\le&  C(1+T_1)\sqrt\epsilon +C\int_0^{t}(\|\Phi^\epsilon_{x}\|_{L^2(I)}^2 +\epsilon \|\Phi^\epsilon_{xx}\|_{L^2(I)}^2)\mathrm{d}t\\
& + C\int_0^{t} (\| u^\epsilon_x\|_{L^2(I)}^2 + 1+\|U^\epsilon_x\|_{L^2(I)}^2 +\epsilon \|\Phi^\epsilon_{xx}\|_{L^2(I)}^2) \|U^\epsilon_x \|_{L^2(I)}^2\mathrm{d}t.
\end{aligned}
\end{equation}
In view of (\ref{ep0-1}) and (\ref{ep0-2}), we let
\begin{equation}\label{ep0-3}
0<\epsilon \le \epsilon_0 \le \min\{\frac{1}{2^2M^2( C+1)^2},(\frac{1}{2}-\frac{1}{{3M}+1})\frac{M}{C },\frac{2{M}-1}{C}\}.
\end{equation}
(\ref{3}) together with the known inequalities: $$\int_I\epsilon U^\epsilon_{xx} \Phi^\epsilon_{xxx}\mathrm{d}x+  \frac{1}{2}\int_I \frac{(U^\epsilon_{xx})^2}{\rho^\epsilon} \mathrm{d}x+\frac{1}{2}\int_I\epsilon^2 \rho^\epsilon(\Phi^\epsilon_{xxx})^2\mathrm{d}x\ge 0$$
and $\frac{6M}{4}\ge \frac{3}{4 \rho^\epsilon}$ yields
\begin{equation}\label{H1-for-detail}
\begin{aligned}
& \int_I [\frac{1}{2}\rho^\epsilon (U^\epsilon_x)^2+(\frac{{2M}-1}{2})\epsilon(\Phi^\epsilon_{xx})^2] \mathrm{d}x +\frac{1}{4}\int_0^t\int_I \frac{(U^\epsilon_{xx})^2}{\rho^\epsilon} \mathrm{d}x\mathrm{d}t\\
&+\frac{{2M}}{2}\int_I(\Phi^\epsilon_{x})^2\mathrm{d}x+\frac{1}{4}\int_0^{t}\epsilon^2 \|\sqrt {\rho^\epsilon}\Phi^\epsilon_{xxx}\|_{L^2(I)}^2\mathrm{d} t\\
\le&  C(1+T_1)\sqrt\epsilon +C\int_0^{t}(\|\Phi^\epsilon_{x}\|_{L^2(I)}^2 +\epsilon \|\Phi^\epsilon_{xx}\|_{L^2(I)}^2)\mathrm{d}t\\
& + C\int_0^{t} (\| u^\epsilon_x\|_{L^2(I)}^2+ 1+\|U^\epsilon_x\|_{L^2(I)}^2 +\epsilon \|\Phi^\epsilon_{xx}\|_{L^2(I)}^2) \|U^\epsilon_x \|_{L^2(I)}^2\mathrm{d}t.\\
\end{aligned}
\end{equation}
Using Gronwall's inequality, we have
\begin{equation}\label{gronwall-H1}
\begin{aligned}
& \int_I [(U^\epsilon_x)^2+\epsilon(\Phi^\epsilon_{xx})^2] \mathrm{d}x+\int_I(\Phi^\epsilon_{x})^2\mathrm{d}x\\
\le&  C(1+T_1)\sqrt\epsilon +C(1+T_1)\sqrt\epsilon (B_1(T_1)+1) \exp\{C (B_1(T_1)+1)\}:=B_2(T_1) \sqrt\epsilon\\
\le& B_2(T) \sqrt\epsilon.\\
\end{aligned}
\end{equation}
Then, putting (\ref{gronwall-H1}) into (\ref{H1-for-detail}), we have
\begin{equation}\label{gronwall-H1}
\begin{aligned}
& \sup_{t\in[0,T_1]}(\|U^\epsilon_x\|_{L^2(I)}^2+\|\Phi^\epsilon_x\|_{L^2(I)}^2+\epsilon\|\Phi^\epsilon_{xx}\|_{L^2(I)}^2)+\|U_{xx}^\epsilon\|_{L_{T_1}^2 L^2}^2+\epsilon^2\|\Phi^\epsilon_{xxx}\|_{L_{T_1}^2 L^2}^2\\
\le&  C(1+T_1)\sqrt\epsilon+ T_1C(1+T_1)\sqrt\epsilon \\
&+CT_1(1+T_1)\sqrt\epsilon (B_1(T_1)+1) \exp\{C (B_1(T_1)+1)\}:=B_2(T_1) \sqrt\epsilon\\
\le& B_2(T) \sqrt\epsilon.\\
\end{aligned}
\end{equation}

 Hence, we get (\ref{Uxx-Phix-L2L22-all}).
\end{proof}

\begin{Corollary}\label{Phi-Linfty-shoulian}
Under the conditions of Theorem \ref{assumption}, there holds that
\begin{equation}\label{Ux-Phix-L2L22}
\sup_{t\in[0,T_1]}\|\Phi^\epsilon\|_{L^\infty(I)}\le \frac{1}{2}.
\end{equation}
\end{Corollary}
\begin{proof}
Using Sobolev's inequality and Corollary \ref{Phi-U-L2-all}, we have
\begin{equation}
\begin{aligned}
\sup_{t\in[0,T_1]}\|\Phi^\epsilon\|_{L^\infty(I)}\le & \sup_{t\in[0,T_1]}( C\|\Phi^\epsilon\|_{L^2(I)}+C\|\Phi^\epsilon\|_{L^2(I)}^{\frac{1}{2}}\|\Phi^\epsilon_x\|_{L^2(I)}^{\frac{1}{2}})\\
\le & C\sqrt{B_1(T)} (\sqrt\epsilon +\epsilon^{\frac{1}{4}}) \le 2C\sqrt{B_1(T)}\epsilon^{\frac{1}{4}}.
\end{aligned}
\end{equation}
Choosing
\begin{equation}\label{ep0-4}
\epsilon \le \epsilon_0 \le  \min\{\frac{1}{2^2M^2( C+1)^2},(\frac{1}{2}-\frac{1}{{3M}+1})\frac{M}{C },\frac{2{M}-1}{C},\frac{1}{4^4 C^4 B_1^2(T)}\},
\end{equation}
we get (\ref{Ux-Phix-L2L22}).
\end{proof}

\medskip

With Corollaries \ref{Phi-U-L2-all}, \ref{Phi-U-H^1-all} and \ref{Phi-Linfty-shoulian}, the proof of Theorem \ref{assumption} is completed.

\subsection{Proof of  Corollary \ref{following-4}}\label{3-3-1}

For any fixed $\epsilon$, system (\ref{zhankai-equ}) has a local solution $(\Phi^\epsilon, U^\epsilon)$. From the initial conditions (\ref{ini-val-phi-11}) and (\ref{ini-val-phi-111}),
using Sobolev's inequality, it holds that
$$\begin{aligned}
\|\Phi^\epsilon(\cdot,0)\|_{L^\infty(I)}\le &C\|\Phi^\epsilon(\cdot,0)\|_{L^2(I)}+C\|\Phi^\epsilon(\cdot,0)\|_{L^2(I)}^{\frac{1}{2}}\|\Phi^\epsilon_x(\cdot,0)\|_{L^2(I)}^{\frac{1}{2}}\\
\le& C  \sqrt\epsilon +C  \epsilon^{\frac{1}{4}} \le C \epsilon^{\frac{1}{4}}.
\end{aligned}
$$
Choosing
\begin{equation}\label{ep1}
\epsilon_1 \le \min\{\epsilon_0, \frac{1}{16 C^4 }\}
\end{equation} where $\epsilon_0$ is determined by (\ref{ep0-4}), and letting $\epsilon\le\epsilon_1$, we have
$$
\|\Phi^\epsilon(\cdot,0)\|_{L^\infty(I)}\le C\epsilon^{\frac{1}{4}} \le \frac{1}{2}<1.
$$
 Using the continuity of the solution with respect to time, there exists a time $T_1(\epsilon) \in(0,T]\cap(0,T^\epsilon)$ such that the {\it a priori} assumption (\ref{ini-val-phi-1}) in Theorem \ref{assumption} holds.

Then, from Theorem \ref{assumption}, suppose that $T_1=T_1(\epsilon)$, it holds that
\begin{equation}
\begin{aligned}
&\|\Phi^\epsilon(\cdot,T_1(\epsilon))\|_{L^2(I)}^2 + \|U^\epsilon(\cdot,T_1(\epsilon))\|_{L^2(I)}^2+\|\Phi^\epsilon_x(\cdot,T_1(\epsilon))\|_{L^2(I)}^2 \le B_1 \epsilon,\\
&\|\Phi^\epsilon_x(\cdot,T_1(\epsilon))\|_{L^2(I)}^2 + \|U^\epsilon(\cdot,T_1(\epsilon))\|_{L^2(I)}^2+\|\Phi^\epsilon_{xx}(\cdot,T_1(\epsilon))\|_{L^2(I)}^2 \le B_2 \sqrt\epsilon,\\
&\|\Phi^\epsilon(\cdot,T_1(\epsilon))\|_{L^\infty(I)} \le \frac{1}{2},
\end{aligned}
\end{equation}
which satisfies the initial condition in Theorem \ref{assumption}. Using the continuity of the solution with respect to time, there exists a positive constant $\delta>0$ such that $T_1(\epsilon)+\delta \in(0,T]\cap(0,T^\epsilon)$, and the system (\ref{zhankai-equ}) has a solution in time $[T_1(\epsilon),T_1(\epsilon)+\delta]$ satisfying
$$
\begin{aligned}
&\|\Phi^\epsilon\|_{C([T_1(\epsilon),T_1(\epsilon)+\delta];L^\infty(I))}\le 1.\\
\end{aligned}
$$
Thus, we choose $T_1=T_1(\epsilon)+\delta$ in Theorem \ref{assumption}, we get (\ref{1234}) in time interval $[0,T_1(\epsilon)+\delta]$.
\begin{itemize}
\item If $T< T^\epsilon$, the above steps can be extended to time interval $[0,T]$.

\item If $T\ge T^\epsilon$, we end our proof by contradiction. Assume that the maximun existence time $T^\epsilon < T$. Then from Theorem \ref{assumption},  we can extend the result of (\ref{1234}) in time interval $[0,T^\epsilon]$. Using the continuity of the solution again, we conclude that the solution of system (\ref{zhankai-equ}) can be extend to $[0,T^\epsilon+\delta']$ for a small $\delta'>0$ satisfying $T^\epsilon +\delta' \le T$, which is a contradiction with the definition of maximun existence time $T^\epsilon$. Thus, we end our proof.
\end{itemize}

\subsection{{Proof of Theorem \ref{following}}}

Under the result of Corollary \ref{following-4}, we get
$$
\begin{aligned}
&\|u^\epsilon-u^{I,0} \|_{L_{T}^\infty L^\infty}^2\le C\|\epsilon(u^{B,2}+u^{b,2})+\epsilon\sqrt{\epsilon}(u^{B,3}+u^{b,3})+M^\epsilon + \sqrt{\epsilon}U^\epsilon\|_{L_{T}^\infty L^\infty}^2\\
\le & C\epsilon^2+ C {\epsilon}\|U^\epsilon\|_{L_{T}^\infty L^2}^2 + C {\epsilon}\|U^\epsilon\|_{L_{T}^\infty L^2}  \|U^\epsilon_x\|_{L_{T}^\infty L^2} \le C \epsilon^{\frac{7}{4}},
\end{aligned}
$$
$$
\begin{aligned}
&\|\rho^\epsilon-\rho^{I,0} \|_{L_{T}^\infty L^\infty}^2\le C\|\sqrt{\epsilon}(\rho^{B,1}+\rho^{b,1})+{\epsilon}(\rho^{B,2}+\rho^{b,2})+L^\epsilon +\sqrt{\epsilon} \Phi^\epsilon\|_{L_{T}^\infty L^\infty}^2\\
\le & C\epsilon+ C {\epsilon}\|\Phi^\epsilon\|_{L_{T}^\infty L^2}^2 + C {\epsilon}\|\Phi^\epsilon\|_{L_{T}^\infty L^2}  \|\Phi^\epsilon_x\|_{L_{T}^\infty L^2} \le  C \epsilon,
\end{aligned}
$$
$$
\begin{aligned}
&\|\rho^\epsilon_x-\rho^{I,0}_x -\sqrt{\epsilon}(\rho^{B,1}+\rho^{b,1})_x \|_{L_{T}^\infty L^\infty}^2\le C\|\epsilon(\rho^{B,2}+\rho^{b,2})_x+L^\epsilon_x + \sqrt\epsilon\Phi^\epsilon_x\|_{L_{T}^\infty L^\infty}^2\\
\le & C\epsilon+ C\epsilon \|\Phi^\epsilon_x\|_{L_{T}^\infty L^2}^2 + C \epsilon\|\Phi^\epsilon_x\|_{L_{T}^\infty L^2}  \|\Phi^\epsilon_{xx}\|_{L_{T}^\infty L^2} \le C \epsilon .
\end{aligned}
$$
Hence, the proof of Theorem \ref{following} is completed.

\appendix  % 从这里开始是附录
\renewcommand{\thesection}{Appendix\ \Alph{section}}  % 章节标题格式
\renewcommand{\theequation}{A.\arabic{equation}}     % 公式编号 A.1
\renewcommand{\thetable}{A.\arabic{table}}           % 表格编号 A.1
\renewcommand{\thefigure}{A.\arabic{figure}}         % 图片编号 A.1

% 重置你的自定义环境的计数器
\renewcommand{\theTheorem}{A.\arabic{Theorem}}
\renewcommand{\theLemma}{A.\arabic{Lemma}}
\renewcommand{\theDefinition}{A.\arabic{Definition}}
\renewcommand{\theProposition}{A.\arabic{Proposition}}
\renewcommand{\theCorollary}{A.\arabic{Corollary}}
\renewcommand{\theRemark}{A.\arabic{Remark}}
\renewcommand{\theclaim}{A.\arabic{claim}}  % 添加这一行

% 重置所有计数器为0
\setcounter{equation}{0}
\setcounter{Theorem}{0}
\setcounter{Lemma}{0}
\setcounter{Definition}{0}
\setcounter{Proposition}{0}
\setcounter{Corollary}{0}
\setcounter{Remark}{0}
\setcounter{claim}{0}
\setcounter{table}{0}
\setcounter{figure}{0}

	\section{Formal derivation of the boundary layer equations}\label{appendixA}
For any $(\rho^{I,i},u^{I,i})$, $(\rho^{B,j},u^{B,j})$, $(\rho^{b,k},u^{b,k})$, $i,j,k \in \mathbb{N}$, firstly, putting (\ref{equ-NS}) and (\ref{Plantl})  into (\ref{equ-epsilon}), and using Taylor expansion
$$\begin{aligned}
(\rho^\epsilon)^{\gamma-2}=& [\rho^{I,0} +\rho^{B,0}+\rho^{b,0}+\rho^\epsilon-(\rho^{I,0}+\rho^{B,0}+\rho^{b,0})]^{\gamma-2}\\
=& (\rho^{I,0}+\rho^{B,0}+\rho^{b,0} )^{\gamma-2}+ (\gamma-2) (\rho^{I,0}+\rho^{B,0}+\rho^{b,0} )^{\gamma-3} \sum_{i=1}^{\infty}\epsilon^{\frac{i}{2}}(\rho^{I,i} +\rho^{B,i}+\rho^{b,i})\\
&+ O \left([\sum_{i=1}^{\infty}\epsilon^{\frac{i}{2}}(\rho^{I,i} +\rho^{B,i}+\rho^{b,i})]^2\right),
\end{aligned}
$$
we have
\begin{equation}\label{zhankaiNSK}
\begin{aligned}
&\sum_{i=0}^{\infty}\epsilon^{\frac{i}{2}} (\rho_t^{I,i} +\rho_t^{B,i}+\rho_t^{b,i})+\sum_{i,j=0}^{\infty}\epsilon^{\frac{i+j}{2}} (\rho^{I,i} +\rho^{B,i}+\rho^{b,i})(u_x^{I,j} +{\epsilon^{-\frac{1}{2}}}u_y^{B,j}+{\epsilon^{-\frac{1}{2}}}u_z^{b,j})\\
&\sum_{i,j=0}^{\infty}\epsilon^{\frac{i+j}{2}} (\rho_x^{I,i} +{\epsilon^{-\frac{1}{2}}}\rho_y^{B,i}+{\epsilon^{-\frac{1}{2}}}\rho_z^{b,i})(u^{I,j} +u^{B,j}+u^{b,j})=0,\\
\end{aligned}
\end{equation}
\begin{equation}\label{zhankaiNSK-1}
\begin{aligned}
&\sum_{i,j=0}^{\infty}\epsilon^{\frac{i+j}{2}} (\rho^{I,i} +\rho^{B,i}+\rho^{b,i})(u_t^{I,j} +u_t^{B,j}+u_t^{b,j})\\
&+\sum_{i,j,k=0}^{\infty}\epsilon^{\frac{i+j+k}{2}} (\rho^{I,i} +\rho^{B,i}+\rho^{b,i})(u^{I,j} +u^{B,j}+u^{b,j})(u_x^{I,k} +{\epsilon^{-\frac{1}{2}}}u_y^{B,k}+{\epsilon^{-\frac{1}{2}}}u_z^{b,k})\\
&+\gamma \{(\rho^{I,0}+\rho^{B,0}+\rho^{b,0} )^{\gamma-2}+ (\gamma-2) (\rho^{I,0}+\rho^{B,0}+\rho^{b,0} )^{\gamma-3} \sum_{i=1}^{\infty}\epsilon^{\frac{i}{2}}(\rho^{I,i} +\rho^{B,i}+\rho^{b,i})\\
&+ O \left([\sum_{i=1}^{\infty}\epsilon^{\frac{i}{2}}(\rho^{I,i} +\rho^{B,i}+\rho^{b,i})]^2\right)\}\sum_{j,k=0}^{\infty}\epsilon^{\frac{j+k}{2}} (\rho^{I,j} +\rho^{B,j}+\rho^{b,j}) (\rho_x^{I,k} +{\epsilon^{-\frac{1}{2}}}\rho_y^{B,k}+{\epsilon^{-\frac{1}{2}}}\rho_z^{b,k})\\
&-\sum_{i=0}^{\infty}\epsilon^{\frac{i}{2}} (u_{xx}^{I,j} +{\epsilon^{-1}}u_{yy}^{B,j}+{\epsilon^{-1}}u_{zz}^{b,j})\\
&-\epsilon\sum_{i,j=0}^{\infty} \epsilon^{\frac{i+j}{2}} (\rho^{I,i} +\rho^{B,i}+\rho^{b,i})(\rho^{I,j}_{xxx} +\epsilon^{-\frac{3}{2}}\rho_{yyy}^{B,j}+\epsilon^{-\frac{3}{2}}\rho_{zzz}^{b,j})=0.
\end{aligned}
\end{equation}

From (\ref{Plantl}) and (\ref{initial-value}), we formally have the initial value
\begin{equation}\label{ini-B-I}
\begin{aligned}
&%(\rho^{I,0},u^{I,0})|_{t=0}=(\rho_0^{I,0},u_0^{I,0}),
(\rho^{I,1},u^{I,1})|_{t=0}=(0,0),~
(\rho^{B,j},u^{B,j})|_{t=0}=(0,0),~(\rho^{b,j},u^{b,j})|_{t=0}=(0,0),~for~j=0,1,2.
\end{aligned}
\end{equation}

From Taylor expansion near $x=0$ and $x=1$ as follows:
$$\begin{aligned}
u^{I,0}(x,t)=&u^{I,0}(\sqrt{\epsilon}y,t)=\overline{u^{I,0}} + \sqrt{\epsilon} y\overline{u^{I,0}_x}  +  \frac{\epsilon y^2}{2} \overline{u^{I,0}_{xx}} +  \frac{\epsilon \sqrt{\epsilon} y^3}{6} \overline{u^{I,0}_{xxx}}+O(\epsilon^2),\\
u^{I,0}(x,t)=&u^{I,0}(\sqrt{\epsilon}z+1,t)=\widetilde{u^{I,0}} + \sqrt{\epsilon} z\widetilde{u^{I,0}_x}  +  \frac{\epsilon z^2}{2} \widetilde{u^{I,0}_{xx}} +  \frac{\epsilon \sqrt{\epsilon} z^3}{6} \widetilde{u^{I,0}_{xxx}}+O(\epsilon^2),\\
\rho^{I,0}(x,t)=&\rho^{I,0}(\sqrt{\epsilon}y,t)=\overline{\rho^{I,0}} + \sqrt{\epsilon} y\overline{\rho^{I,0}_x}  +  \frac{\epsilon y^2}{2} \overline{\rho^{I,0}_{xx}} +  \frac{\epsilon \sqrt{\epsilon} y^3}{6} \overline{\rho^{I,0}_{xxx}}+O(\epsilon^2),\\
\rho^{I,0}(x,t)=&\rho^{I,0}(\sqrt{\epsilon}z+1,t)=\widetilde{\rho^{I,0}} + \sqrt{\epsilon} z\widetilde{\rho^{I,0}_x}  +  \frac{\epsilon z^2}{2} \widetilde{\rho^{I,0}_{xx}} +  \frac{\epsilon \sqrt{\epsilon} z^3}{6} \widetilde{\rho^{I,0}_{xxx}}+O(\epsilon^2),\\
(\rho^{I,0})^s(x,t)=&(\rho^{I,0})^s(\sqrt{\epsilon}y,t)=\overline{(\rho^{I,0})^s} + \sqrt{\epsilon} ys(\overline{\rho^{I,0}})^{s-1}\overline{\rho^{I,0}_x}  +  \frac{\epsilon y^2}{2} s(\overline{\rho^{I,0}})^{s-1}\overline{\rho^{I,0}_{xx}} \\
&+  \frac{\epsilon y^2}{2} s(s-1)(\overline{\rho^{I,0}})^{s-2}(\overline{\rho^{I,0}_{x}} )^2+O(\epsilon\sqrt{\epsilon}),\\
(\rho^{I,0})^s(x,t)=&\rho^{I,0}(\sqrt{\epsilon}z+1,t)=\widetilde{(\rho^{I,0})^s} + \sqrt{\epsilon} zs(\widetilde{\rho^{I,0}})^{s-1}\widetilde{\rho^{I,0}_x}  +  \frac{\epsilon z^2}{2} s(\widetilde{\rho^{I,0}})^{s-1}\widetilde{\rho^{I,0}_{xx}} \\
&+  \frac{\epsilon z^2}{2} s(s-1)(\widetilde{\rho^{I,0}})^{s-2}(\widetilde{\rho^{I,0}_{x}})^2+O(\epsilon\sqrt{\epsilon}),\\
\end{aligned}
$$
we have
\begin{table}[H]
\centering
\small
\caption{$\rho^{\epsilon}$}
\label{tab:rho_estimates}
\begin{tabular}{lccc}
\toprule
 & \textbf{$\rho^{\epsilon}_t$} & \textbf{$\rho^\epsilon_x$} & \textbf{$\rho^\epsilon$} \\
\midrule
Order $\epsilon^{-\frac{1}{2}}$ & $-$ & $\rho^{B,0}_y$  & $-$ \\
\addlinespace[0.5em]  % 0.5em 行距
Order $\epsilon^{0}$ & $\overline{\rho^{I,0}_t}+\rho^{B,0}_t$, & $\overline{\rho^{I,0}_x}+\rho^{B,1}_y$,  & $\overline{\rho^{I,0}}+\rho^{B,0}$, \\
\addlinespace[0.5em]  % 0.5em 行距
Order $\epsilon^{\frac{1}{2}}$ & $\overline{\rho^{I,1}_t}+y\overline{\rho^{I,0}_{xt}}+\rho^{B,1}_t $,  & $\overline{\rho^{I,1}_x}+y\overline{\rho^{I,0}_{xx}}+\rho^{B,2}_y$,  & $\overline{\rho^{I,1}}+y\overline{\rho^{I,0}_{x}}+\rho^{B,1}$,  \\
\addlinespace[0.5em]  % 0.5em 行距
Order $\epsilon^{1}$ & {\parbox{3cm}{\centering $\overline{\rho^{I,2}_t}+y\overline{\rho^{I,1}_{xt}}+\frac{y^2}{2}\overline{\rho^{I,0}_{xxt}}+\rho^{B,1}_t, $}}
  & {\parbox{3cm}{\centering $\overline{\rho^{I,2}_x}+y\overline{\rho^{I,1}_{xx}}+\frac{y^2}{2}\overline{\rho^{I,0}_{xxx}}+\rho^{B,3}_{y}$,}}   & {\parbox{3cm}{\centering $\overline{\rho^{I,2}}+y\overline{\rho^{I,1}_{x}}+\frac{y^2}{2}\overline{\rho^{I,0}_{xx}}+\rho^{B,2}$.}}   \\
\bottomrule
\end{tabular}
\end{table}

\begin{table}[H]
\centering
\small
\caption{ $u^\epsilon$}
\label{tab:rho_estimates}
\begin{tabular}{lccc}
\toprule
 & \textbf{$u^{\epsilon}_t$} & \textbf{$u^\epsilon_x$} & \textbf{$u^\epsilon$} \\
\midrule
Order $\epsilon^{-\frac{1}{2}}$ & $-$ & $u^{B,0}_y$  & $-$ \\
\addlinespace[0.5em]  % 0.5em 行距
Order $\epsilon^{0}$ & $\overline{u^{I,0}_t}+u^{B,0}_t$, & $\overline{u^{I,0}_x}+u^{B,1}_y$,  & $\overline{u^{I,0}}+u^{B,0}$, \\
\addlinespace[0.5em]  % 0.5em 行距
Order $\epsilon^{\frac{1}{2}}$ & $\overline{u^{I,1}_t}+y\overline{u^{I,0}_{xt}}+u^{B,1}_t $,  & $\overline{u^{I,1}_x}+y\overline{u^{I,0}_{xx}}+u^{B,2}_y$,  & $\overline{u^{I,1}}+y\overline{u^{I,0}_{x}}+u^{B,1}$,  \\
\addlinespace[0.5em]  % 0.5em 行距
Order $\epsilon^{1}$ & {\parbox{3cm}{\centering $\overline{u^{I,2}_t}+y\overline{u^{I,1}_{xt}}+\frac{y^2}{2}\overline{u^{I,0}_{xxt}}+u^{B,1}_t, $}}
  & {\parbox{3cm}{\centering $\overline{u^{I,2}_x}+y\overline{u^{I,1}_{xx}}+\frac{y^2}{2}\overline{u^{I,0}_{xxx}}+u^{B,3}_{y}$,}}   & {\parbox{3cm}{\centering $\overline{u^{I,2}}+y\overline{u^{I,1}_{x}}+\frac{y^2}{2}\overline{u^{I,0}_{xx}}+u^{B,2}$.}}   \\
\bottomrule
\end{tabular}
\end{table}

\begin{table}[H]
\centering
\small
\caption{$\rho^{\epsilon}_{xxx}$, $u^\epsilon_{xx}$ and $\gamma (\rho^\epsilon)^{\gamma-2}$ }
\label{tab:uxx_estimates}
\begin{tabular}{lccc}
\toprule
 & \textbf{$\rho^{\epsilon}_{xxx}$} & \textbf{$u^\epsilon_{xx}$} &  $\gamma (\rho^\epsilon)^{\gamma-2}$\\
\midrule
Order $\epsilon^{-\frac{3}{2}}$ & $\rho^{B,0}_{yyy},$ & $-$  & $-$\\
\addlinespace[0.5em]  % 0.5em 行距
Order $\epsilon^{-1}$ & $\rho^{B,1}_{yyy},$ & $u^{B,0}_{yy},$  & $-$ \\
\addlinespace[0.5em]  % 0.5em 行距
Order $\epsilon^{-\frac{1}{2}}$ & $\rho^{B,2}_{yyy}$, & $u^{B,1}_{yy},$  &  $-$\\
\addlinespace[0.5em]  % 0.5em 行距
Order $\epsilon^{0}$ & $\overline{\rho^{I,0}_{xxx}}+\rho^{B,3}_{yyy}$, & $\overline{u^{I,0}_{xx}}+u^{B,2}_{yy}$, & {\parbox{3.5cm}{\centering $\gamma (\overline{\rho^{I,0}}+\rho^{B,0})^{\gamma-2} $,}}   \\
\addlinespace[0.5em]  % 0.5em 行距
Order $\epsilon^{\frac{1}{2}}$ & $...... $,  & $\overline{u^{I,1}_{xx}}+y\overline{u^{I,0}_{xxx}}+u^{B,3}_{yy}$,  & $.......$ \\
\bottomrule
\end{tabular}
\end{table}
\noindent where the expansion near $x=1$ is similar.

Order $\epsilon^{-1}$ of (\ref{zhankaiNSK-1}):
$$
\begin{aligned}
u^{B,0}_{yy} + u^{b,0}_{zz}=0,
\end{aligned}
$$
i.e.
\begin{equation}\label{u-B0}
u^{B,0}=0,~u^{b,0}=0.
\end{equation}

Order $\epsilon^{0}$ of (\ref{zhankaiNSK}):
Taking $z\rightarrow -\infty$ and $y\rightarrow +\infty$ respectively, we get
\begin{equation}\label{rho-B0}
\begin{aligned}
&\rho^{B,0}_t +\rho^{B,0} \overline{u^{I,0}_x}+\rho_y^{B,0}(\overline{u^{I,1}}+y\overline{u^{I,0}_x}+u^{B,1} )+(\overline{\rho^{I,0} }+\rho^{B,0} )u^{B,1}_y =0,
\end{aligned}
\end{equation}
\begin{equation}\label{rho-b0}
\begin{aligned}
&\rho^{b,0}_t+\rho^{b,0}\widetilde{u^{I,0}_x}
+\rho_z^{b,0}(\widetilde{u^{I,1} }+y\widetilde{u^{I,0}_x }+u^{b,1})+(\widetilde{\rho^{I,0}} +\rho^{b,0})u^{b,1}_z =0.
\end{aligned}
\end{equation}
From $\rho^\epsilon_x|_{x=0,1}= 0$ and (\ref{Plantl})$_1$, it holds that
\begin{equation}\label{rho-I0-bou}
\rho^{B,0}_y|_{y=0}= 0,~\rho^{b,0}_z|_{z=0}= 0.
\end{equation}

Order $\epsilon^{-\frac{1}{2}}$ of (\ref{zhankaiNSK-1}):
Taking $z\rightarrow -\infty$ and $y\rightarrow +\infty$ respectively, we get
\begin{equation}\label{u-B1-yy}
\begin{aligned}
&u_{yy}^{B,1}\!+(\overline{\rho^{I,0}}+ \rho^{B,0})\rho_{yyy}^{B,0}-\gamma (\overline{\rho^{I,0}}+\rho^{B,0})^{\gamma-1} \rho^{B,0}_y =0,
\end{aligned}
\end{equation}
\begin{equation}\label{u-b1-zz}
\begin{aligned}
&u_{zz}^{b,1}+(\widetilde{\rho^{I,0}}+\rho^{b,0})\rho_{zzz}^{b,0}-\gamma (\widetilde{\rho^{I,0}}+\rho^{b,0})^{\gamma-1}\rho^{b,0}_z=0.
\end{aligned}
\end{equation}

It is easy to see that system (\ref{ini-B-I})$_3$ (\ref{rho-B0}), (\ref{rho-I0-bou})$_1$ and (\ref{u-B1-yy}) has a unique soultion
\begin{equation}\label{sol-rho-B0-u-B1}
(\rho^{B,0},u^{B,1})=(0,0).
\end{equation}
% and the system (\ref{ini-B-I})$_4$, (\ref{rho-b0}), (\ref{rho-I0-bou})$_2$ and  (\ref{u-b1-zz}) has a unique solution

Similarly, we have
\begin{equation}\label{sol-rho-b0-u-b1}
(\rho^{b,0},u^{b,1})=(0,0).
\end{equation}

System of $(\rho^{I,1}, u^{I,1})$:
\begin{equation}\label{I-1}
\left\{
\begin{aligned}
&\rho^{I,1}_t +(\rho^{I,0}u^{I,1}+\rho^{I,1}u^{I,0})_x=0,\\
&\rho^{I,0} u_t^{I,1}+\rho^{I,1}u_t^{I,0}+\rho^{I,1} u^{I,0}u_x^{I,0}+\rho^{I,0} u^{I,1}u_x^{I,0}+\rho^{I,0} u^{I,0}u_x^{I,1}\\
&+\gamma (\rho^{I,0})^{\gamma-1}\rho^{I,1}_x+\gamma(\gamma-1) (\rho^{I,0})^{\gamma-2}\rho^{I,1}\rho^{I,0}_x=u^{I,1}_{xx},\\
& (\rho^{I,1},u^{I,1})|_{t=0}=(0,0),\\
& u^{I,1}|_{x=0,1}=0,
\end{aligned}
\right.
\end{equation}
where we get the boundary condition from (\ref{equ-epsilon})$_4$, (\ref{sol-rho-B0-u-B1}) and (\ref{sol-rho-b0-u-b1}).

Obviously, $(\rho^{I,1}, u^{I,1})=(0,0)$ is the solution of system (\ref{I-1}). Morevoer, estimate directly, the soultion is unique.

Order $\epsilon^{\frac{1}{2}}$ of (\ref{zhankaiNSK}) and order $\epsilon^{0}$ of (\ref{zhankaiNSK-1}):
Taking $z\rightarrow -\infty$ and $y\rightarrow +\infty$ respectively, we get
$$\left\{
\begin{aligned}
&\rho^{B,1}_t+y\overline{u^{I,0}_x}\rho^{B,1}_y+\overline{\rho^{I,0}}u^{B,2}_y+\rho^{B,1}\overline{u^{I,0}_x}=0,\\
&\gamma (\overline{\rho^{I,0}})^{\gamma-1} \rho^{B,1}_y
=u^{B,2}_{yy}+\overline{\rho^{I,0}}\rho^{B,1}_{yyy} ,
\end{aligned}
\right.
$$
$$\left\{
\begin{aligned}
&\rho^{b,1}_t+z\widetilde{u^{I,0}_x}\rho^{b,1}_z+\widetilde{\rho^{I,0}}u^{b,2}_z+\rho^{b,1}\widetilde{u^{I,0}_x}=0,\\
&\gamma (\widetilde{\rho^{I,0}})^{\gamma-1}\rho^{b,1}_z
=u^{b,2}_{zz} +\widetilde{\rho^{I,0}}\rho^{b,1}_{zzz} .
\end{aligned}
\right.
$$
Estimate them directly, we have
\begin{equation}\label{rhoB1-uB2}
\left\{
\begin{aligned}
&\rho^{B,1}_t+y\overline{u^{I,0}_x}\rho^{B,1}_y+[\gamma (\overline{\rho^{I,0}})^{\gamma}+\overline{u^{I,0}_x}]\rho^{B,1}=(\overline{\rho^{I,0}})^2\rho^{B,1}_{yy},\\
&u^{B,2}=-\overline{\rho^{I,0}}\rho^{B,1}_{y}-\int_y^{+\infty}\gamma (\overline{\rho^{I,0}})^{\gamma-1} \rho^{B,1}(\xi,t)\mathrm{d}\xi ,
\end{aligned}
\right.
\end{equation}
\begin{equation}\label{rhob1-ub2}
\left\{
\begin{aligned}
&\rho^{b,1}_t+z\widetilde{u^{I,0}_x}\rho^{b,1}_z+[\gamma (\widetilde{\rho^{I,0}})^{\gamma}+\widetilde{u^{I,0}_x}]\rho^{b,1}=(\widetilde{\rho^{I,0}})^2\rho^{b,1}_{zz},\\
&u^{b,2}
= -\widetilde{\rho^{I,0}}\rho^{b,1}_{z}+\int_{-\infty}^z\gamma (\widetilde{\rho^{I,0}})^{\gamma-1}\rho^{b,1}(\zeta,t)\mathrm{d}\zeta .
\end{aligned}
\right.
\end{equation}
From $\rho^\epsilon_x|_{x=0,1}= 0$ and (\ref{Plantl})$_1$, it holds that
\begin{equation}\label{rho-B1-y-bou}
\rho^{B,1}_y|_{y=0}= -\overline{\rho^{I,0}_x},~\rho^{b,1}_z|_{z=0}= -\widetilde{\rho^{I,0}_x}.
\end{equation}

Order $\epsilon$ of (\ref{zhankaiNSK}) and order $\epsilon^{\frac{1}{2}}$ of (\ref{zhankaiNSK-1}):
Taking $z\rightarrow -\infty$ and $y\rightarrow +\infty$ respectively, using
$$\begin{aligned}
\rho^{I,2}_t + \rho^{I,2}_x u^{I,0} +  \rho^{I,2} u^{I,0}_x+ \rho^{I,0}_x u^{I,2} +  \rho^{I,0} u^{I,2}_x=0,
\end{aligned}
$$
 it holds that
$$
\left\{
\begin{aligned}
&\rho^{B,2}_t+\overline{\rho^{I,0}_x}u^{B,2}+\overline{\rho^{I,0}}u^{B,3}_y +\rho^{B,1}u^{B,2}_y+ \rho^{B,1}_y(\overline{u^{I,2}}+u^{B,2})+\rho^{B,2}\overline{u^{I,0}_x}+y\overline{u^{I,0}_{xx}}\rho^{B,1}\\
&+y\overline{\rho^{I,0}_x}u^{B,2}_y+y\overline{u^{I,0}_{x}}\rho^{B,2}_y+\frac{y^2}{2}\overline{u^{I,0}_{xx}}\rho^{B,1}_y=0,\\
&\gamma (\overline{\rho^{I,0}})^{\gamma-1} \rho^{B,2}_y+\gamma (\gamma-1) (\overline{\rho^{I,0}})^{\gamma-2} \rho^{B,1} (\overline{\rho^{I,0}_x}+\rho^{B,1}_y)+\gamma (\gamma-1)y (\overline{\rho^{I,0}})^{\gamma-2} \overline{\rho^{I,0}_x}\rho^{B,1}_y \\
&= u_{yy}^{B,3}+\overline{ \rho^{I,0} }\rho^{B,2}_{yyy}+ \rho^{B,1} \rho^{B,1}_{yyy}+y\overline{\rho^{I,0}_x}\rho^{B,1}_{yyy},
\end{aligned}\right.
$$
$$
\left\{
\begin{aligned}
&\rho^{b,2}_t+\widetilde{\rho^{I,0}_x}u^{b,2}+\widetilde{\rho^{I,0}}u^{b,3}_z+\rho^{b,1}u^{b,2}_z+\rho^{b,1}_z(\widetilde{u^{I,2}}+u^{b,2})+\rho^{b,2}\widetilde{u^{I,0}_x}+z\widetilde{u^{I,0}_{xx}}\rho^{b,1}\\
&+z\widetilde{\rho^{I,0}_x}u^{b,2}_z+z\widetilde{u^{I,0}_{x}}\rho^{b,2}_z+\frac{z^2}{2}\widetilde{u^{I,0}_{xx}}\rho^{b,1}_z=0,\\
&\gamma (\widetilde{\rho^{I,0}})^{\gamma-1} \rho^{b,2}_z+\gamma (\gamma-1) (\widetilde{\rho^{I,0}})^{\gamma-2} \rho^{b,1} (\widetilde{\rho^{I,0}_x}+\rho^{b,1}_z)+\gamma (\gamma-1) z(\widetilde{\rho^{I,0}})^{\gamma-2} \widetilde{\rho^{I,0}_x}\rho^{b,1}_z \\
&=u_{zz}^{b,3}+ \widetilde{\rho^{I,0}} \rho^{b,2}_{zzz}+ \rho^{b,1}\rho^{b,1}_{zzz}+z\widetilde{\rho^{I,0}_x} \rho^{b,1}_{zzz}.
\end{aligned}\right.
$$
Thus, we get
\begin{equation}\label{rhoB2-uB3}
\left\{
\begin{aligned}
&\rho^{B,2}_t+\overline{\rho^{I,0}_x}u^{B,2} +\rho^{B,1}u^{B,2}_y+ \rho^{B,1}_y(\overline{u^{I,2}}+u^{B,2})+\rho^{B,2}\overline{u^{I,0}_x}+y\overline{u^{I,0}_{xx}}\rho^{B,1}\\
&+y\overline{\rho^{I,0}_x}u^{B,2}_y+y\overline{u^{I,0}_{x}}\rho^{B,2}_y+\frac{y^2}{2}\overline{u^{I,0}_{xx}}\rho^{B,1}_y \\
&=-\gamma (\overline{\rho^{I,0}})^{\gamma} \rho^{B,2}-\frac{1}{2}\gamma (\gamma-1) (\overline{\rho^{I,0}})^{\gamma-1} (\rho^{B,1})^2 -\gamma (\gamma-1)y (\overline{\rho^{I,0}})^{\gamma-1} \overline{\rho^{I,0}_x}\rho^{B,1}\\
&+(\overline{ \rho^{I,0} })^2\rho^{B,2}_{yy}+\overline{ \rho^{I,0} }\rho^{B,1} \rho^{B,1}_{yy}-\frac{1}{2}\overline{ \rho^{I,0} }(\rho^{B,1}_y )^2+y\overline{ \rho^{I,0} }\overline{\rho^{I,0}_x}\rho^{B,1}_{yy}-\overline{ \rho^{I,0} }\overline{\rho^{I,0}_x}\rho^{B,1}_{y},\\
&u^{B,3}
=-\gamma (\overline{\rho^{I,0}})^{\gamma-1} Q_2-\frac{1}{2}\gamma (\gamma-1) (\overline{\rho^{I,0}})^{\gamma-2} \int_y^{+\infty} (\rho^{B,1})^2(\xi,t) \mathrm{d}\xi\\
& -\gamma (\gamma-1) (\overline{\rho^{I,0}})^{\gamma-2} \overline{\rho^{I,0}_x}(yQ_{1,1}+Q_{1,2})-\overline{ \rho^{I,0} }\rho^{B,2}_{y}-\rho^{B,1} \rho^{B,1}_{y}\\
&-\frac{3}{2}\int_y^{+\infty} (\rho^{B,1}_\xi )^2(\xi,t)\mathrm{d}\xi-y\overline{\rho^{I,0}_x}\rho^{B,1}_{y}+2\overline{\rho^{I,0}_x}\rho^{B,1},\\
\end{aligned}\right.
\end{equation}
$$
\left\{
\begin{aligned}
&\rho^{b,2}_t+\widetilde{\rho^{I,0}_x}u^{b,2}+\rho^{b,1}u^{b,2}_z+\rho^{b,1}_z(\widetilde{u^{I,2}}+u^{b,2})+\rho^{b,2}\widetilde{u^{I,0}_x}+z\widetilde{u^{I,0}_{xx}}\rho^{b,1}\\
&+z\widetilde{\rho^{I,0}_x}u^{b,2}_z+z\widetilde{u^{I,0}_{x}}\rho^{b,2}_z+\frac{z^2}{2}\widetilde{u^{I,0}_{xx}}\rho^{b,1}_z\\%=-\widetilde{\rho^{I,0}}u^{b,3}_z\\
&=-\gamma (\widetilde{\rho^{I,0}})^{\gamma} \rho^{b,2}-\frac{1}{2}\gamma (\gamma-1) (\widetilde{\rho^{I,0}})^{\gamma-1} (\rho^{b,1})^2 -\gamma (\gamma-1)z (\widetilde{\rho^{I,0}})^{\gamma-1} \widetilde{\rho^{I,0}_x}\rho^{b,1}\\
&+(\widetilde{ \rho^{I,0} })^2\rho^{b,2}_{zz}+\widetilde{ \rho^{I,0} }\rho^{b,1} \rho^{b,1}_{zz}-\frac{1}{2}\widetilde{ \rho^{I,0} }(\rho^{b,1}_z )^2+z\widetilde{ \rho^{I,0} }\widetilde{\rho^{I,0}_x}\rho^{b,1}_{zz}-\widetilde{ \rho^{I,0} }\widetilde{\rho^{I,0}_x}\rho^{b,1}_{z},\\
&u^{b,3}
=\gamma (\widetilde{\rho^{I,0}})^{\gamma-1} q_2+\frac{1}{2}\gamma (\gamma-1) (\widetilde{\rho^{I,0}})^{\gamma-2} \int_{-\infty}^{z} (\rho^{b,1})^2(\zeta,t) \mathrm{d}\zeta\\
& +\gamma (\gamma-1) (\widetilde{\rho^{I,0}})^{\gamma-2} \widetilde{\rho^{I,0}_x}(zq_{1,1}-q_{1,2})-\widetilde{ \rho^{I,0} }\rho^{b,2}_{z}-\rho^{b,1} \rho^{b,1}_{z}\\
&+\frac{3}{2}\int_{-\infty}^{z} (\rho^{b,1}_\zeta )^2(\zeta,t)\mathrm{d}\zeta-z\widetilde{\rho^{I,0}_x}\rho^{b,1}_{z}+2\widetilde{\rho^{I,0}_x}\rho^{b,1},\\
\end{aligned}\right.
$$
where we define
$$Q_{1,1}(y,t)=\int_y^{+\infty}\rho^{B,1}(\xi,t) \mathrm{d}\xi ,~Q_{1,2}(y,t)=\int_y^{+\infty} \int_s^{+\infty} \rho^{B,1}(\xi,t) \mathrm{d}\xi \mathrm{d}s,~
Q_2=\int_y^{+\infty}\rho^{B,2}(\xi,t) \mathrm{d}\xi ,
$$
$$
q_{1,1}(y,t)=\int_{-\infty}^z \rho^{b,1}(\zeta,t) \mathrm{d}\zeta , ~q_{1,2}(y,t)=\int_{-\infty}^z  \int_{-\infty}^s \rho^{b,1}(\zeta,t) \mathrm{d}\zeta\mathrm{d}s,~q_2= \int_{-\infty}^z \rho^{b,2}(\zeta,t)\mathrm{d}\zeta.
$$

From $\rho^\epsilon_x|_{x=0,1}= 0$ and (\ref{Plantl})$_1$, it holds that
\begin{equation}\label{rho-B2-bou}
\rho^{B,2}_y|_{y=0}= -\overline{\rho^{I,1}_x}=0,~\rho^{b,2}_z|_{z=0}= -\widetilde{\rho^{I,1}_x}=0.
\end{equation}

From the boundary condition (\ref{equ-epsilon})$_3$ and  (\ref{equ-NS})$_3$, there holds
\begin{equation}\label{uI2}
\begin{aligned}
\overline{u^{I,2}}=& u^{I,2}(0,t)=-u^{B,2}(0,t)
=\overline{\rho^{I,0}}\overline{\rho^{I,0}_x}+\int_0^{+\infty}\gamma (\overline{\rho^{I,0}})^{\gamma-1} \rho^{B,1}(\xi,t)\mathrm{d}\xi,\\
\end{aligned}
\end{equation}
$$\begin{aligned}
\widetilde{u^{I,2}}=& u^{I,2}(1,t)=-u^{b,2}(1,t)
= \widetilde{\rho^{I,0}}\widetilde{\rho^{I,0}_x}-\int_{-\infty}^0\gamma (\widetilde{\rho^{I,0}})^{\gamma-1}\rho^{b,1}(\zeta,t)\mathrm{d}\zeta .
\end{aligned}
$$

\section*{Acknowledgments}
This work is supported by the National Natural Science Foundation of China $\#12471209$ and \#12101140, and by the Guangzhou Basic and Applied Research Projects $\#SL2024A04J01206$.

\end{document}